\documentclass[11pt]{article}
\usepackage{amsmath,amssymb,amsfonts,amsthm}
\newcounter{item}[section]
\newcounter{kirshr}
\newcounter{kirsha}
\newcounter{kirshb}
\newenvironment{enumroman}{\setcounter{kirshr}{1}
\begin{list}{(\roman{kirshr})}{\usecounter{kirshr}} }{\end{list}}
\newenvironment{enumarab}{\setcounter{kirshb}{1}

\begin{list}{(\arabic{kirshb})}{\usecounter{kirshb}} }{\end{list}}
\newtheorem{theorem}{Theorem}[section]

\newtheorem{lemma}[theorem]{Lemma}
\newtheorem{corollary}[theorem]{Corollary}
\newenvironment{demo}[1]{\noindent{\bf #1.}\upshape\mdseries}
{\nopagebreak{\hfill\rule{2mm}{2mm}\nopagebreak}\par\normalfont}
\theoremstyle{definition}

\newtheorem{example}[theorem]{Example}
\newtheorem{definition}[theorem]{Definition}

\def\Nr{{\sf{Nr}}}

\def\Sg{{\mathfrak{Sg}}}
\def\Fm{{\mathfrak{Fm}}}
\def\H{{\mathfrak{H}}}
\def\K{{\mathfrak{K}}}
\def\CA{{\bf CA}}
\def\RCA{{\bf RCA}}

\def\QA{{\bf QA}}
\def\QEA{{\bf QEA}}
\def\Rd{{\ Rd}}
\def\QEA{{\bf PEA}}

\def\(R)RA{{\bf (R)RA}}
\def\RA{{\bf RA}}

\def\R{\mathbb{R}}
\def\N{\mathbb{N}}

\def\C{\mathbb{C}}
\def\A{{\mathfrak{A}}}
\def\B{{\mathfrak{B}}}
\def\C{{\mathfrak{C}}}
\def\D{{\mathfrak{D}}}
\def\R{{\mathfrak{R}}}
\def\P{{\mathfrak{P}}}
\def\L{{\mathfrak{L}}}
\def\Rd{{\mathfrak{Rd}}}

\def\T{{\bf T}}

\def\T{{\bf T}}

\def\RK{{\bf RK}}
\def\Ra{{\sf Ra}}
\def\set#1{ \{#1\}}
\def\Mo{{\sf Mo}}
\def\restr #1{{\restriction_{#1}}}
\def\ws{winning strategy}

 \def\CA{{\sf CA}}
\def\B{{\sf B}}
\def\G{{\sf G}}
\def\w{{\sf w}}
\def\y{{\sf y}}
\def\g{{\sf g}}
\def\b{{\sf b}}
\def\r{{\sf r}}
\def\K{{\sf K}}

\def\pa{$\forall$}
\def\pe{$\exists$}

\def\ef{Ehren\-feucht--Fra\"\i ss\'e}

\def\M{{\mathfrak{M}}}

\def\Ca{{\mathfrak{Ca}}}

\def\M{{\mathfrak{M}}}

\def\A{{\mathfrak{A}}}
\def\B{{\mathfrak{B}}}
\def\C{{\mathfrak{C}}}
\def\D{{\mathfrak{D}}}
\def\Ig{{\mathfrak{Ig}}}
\def\M{{\mathfrak{M}}}
\def\dom{{\sf dom}}
\def\rng{{\sf rng}}
\def\Rd{{\sf{Rd}}}
\def\At{{\sf At}}
\def\Ra{{\mathfrak{Ra}}}
\def\Tm{{\mathfrak{Tm}}}
\def\Cm{{\mathfrak{Cm}}}
\def\E{{\mathfrak{E}}}
\def\Bl{{\mathfrak{Bl}}}

\def\ef{Ehren\-feucht--Fra\"\i ss\'e}

\def\Id{{\sf Id}}
\def\Rl{{\mathfrak{Rl}}}
\def\F{{\mathfrak{F}}}
\def\Sc{{\mathfrak{Sc}}}
\def\QEA{{\sf QEA}}
\def\RQEA{{\sf RQEA}}
\def\RCA{{\sf RCA}}
\def\QEA{{\bf PEA}}
\def\PA{{\bf PA}}
\def\R{{\sf R}}
\def\L{{\sf L}}

\def\U{\mathfrak{U}}

\def\Df{{\sf Df}}

\def\Rd{{\mathfrak{Rd}}}
\def\PEA{{\sf QEA}}
\def\la#1{\langle#1\rangle}

\def\s{{\sf s}}
\def\Sc{{\sf Sc}}
\def\nodes{{\sf nodes}}

\def\c{{\sf c}}

\def\T{{\sf T}}
\def\QEA{{\sf PEA}}

\def\cyl#1{{\sf c}_{#1}}

\def\cyl#1{{\sf c}_{#1}}
\def\sub#1#2{{\sf s}^{#1}_{#2}}
\def\diag#1#2{{\sf d}_{#1#2}}

\def\V{{\sf V}}
\def\de{Dedekind-MacNeille}
\def\PA{{\sf PA}}
\def\QA{{\sf QA}}
\def\RK{{\sf RK}}
\def\Z{{\mathbb{Z}}}

\def\Bb{{\sf Bb}}

\def\Ra{{\sf Ra}}
\def\Mo{{\sf M}}
\def\QEA{{\sf QEA}}\def\Mat{{\sf Mat}}

\def\RA{{\sf RA}}
\def\G{{\sf G}}
\def\GGG{{\cal{G}}}
\def\T{{\cal T}}
\def\b#1{{\bar{ #1}}}

\def\w{{\sf w}}
\def\g{{\sf g}}
\def\y{{\sf y}}
\def\r{{\sf r}}
\def\bb{{\sf b}} 

\title{Splitting methods in algebraic logic\\
{\it Proving results on non--atom--canonicity, non--finite axiomatizability and non--first oder 
definability for cylindric and relation algebras}}
\author{Tarek Sayed Ahmed\\
Department of Mathematics, Faculty of Science,\\ 
Cairo University, Giza, Egypt.
 }
\begin{document}
\maketitle

\begin{abstract}
We deal with various splitting methods in algebraic logic. The word `splitting' refers to splitting some of the atoms in a given relation or  cylindric algebra 
each into one or more subatoms obtaining a bigger algebra, where the number of subatoms obtained after 
splitting is adjusted for a certain combinatorial purpose. This number (of subatoms) can be an infinite cardinal. 
The idea  originates with Leon Henkin.
Splitting methods existing in a scattered form in the literature, possibly under different names, 
proved useful in obtaining (negative) results on non--atom canonicity, non--finite axiomatizability and non--first order
definability for various classes of relation and cylindric algebras.
In a unified framework, we give several known and new examples of each. Our framework covers Monk's splitting, Andr\'eka's splitting,
and, also,  so--called blow up and blur constructions involving splitting (atoms) in 
finite Monk--like algebras and 
rainbow algebras. 
\end{abstract}

\section{Introduction}

Fix $2<n<\omega$. We deal with relation algebras $\RA$s and many cylindric--like algebras like Pinter's algebras and quasi-polyadic algebras (with and without equality).
We focus in the introduction only on cylindric algebras of dimension $n$ ($\CA_n$s).
The idea of splitting one or more atoms in an  algebra to get a (bigger) superalgebra, where splitting is understood as in the first line of the abstract,   
originates with Henkin \cite[p.378, footnote 1]{HMT1}.
The first major use of splitting techniques proving that 
the varieties $\sf RRA$ and  $\RCA_n$ (representable $(\CA_s$)
are not finitely axiomatizable is due to Monk.  Monk proved his seminal results on $\RA$s and $\CA_n$ by constructing finite non--representable 
algebras (referred to together with variations thereof in the literature as Monk--like algebras), whose ultraproduct is representable.

The idea involved in the construction of a non--representable finite Monk (--like) $\A\in \CA_n$s is not so hard.
Such $\A$ is finite, hence atomic, more precisely its Boolean reduct is  atomic.
The algebra $\A$ is obtained by splitting some atoms in a finite $\CA_n$ each into one or more subatoms.
The new atoms are given colours, and
cylindrifications and diagonals are re-defined by stating that monochromatic triangles
are inconsistent. If the atoms resulting after splitting are `enough', that is, a Monk's algebra has many more atoms than colours,
it follows by using a fairly standard form of 
Ramsey's Theorem that any representation of $\A$ will contain a monochromatic triangle, 
so $\A$, 
by definition, cannot be representable.

The second major modification of splitting methods is due to Jonsson who 'splitted atoms in relation algebras', showing that $\sf RRA$ is not {\it only not fninitely axiomatizable},
but, in fact any universal axiomatization of $\sf RRA$ must contain infinitely many variables. Plainly this result is substantialy stronger than Monk's result.

In the cylindric paradigm, And\'reka modified such splitting methods {\it re-inventing} (Andreka's) splitting. 
In this new setting,   Andr\'eka proved a plethora of 
{\it relative non--finite axiomatizability results} in the following sense.

Let $\K$ be a variety having signature $t$, and let $\L$ be a variety having signature $t'\subseteq t$, such that if 
$\A\in \K$ then the reduct of $\A$ obtained by discarding the operations in $t\sim t'$, $\Rd_{t'}\A$ for short,  is in $\L$.
We say that a set of first order formulas
$\Sigma$ in the signature $t$ axiomatizes $\K$ over $\L$, if  
for any algebra $\A$ in the signature $t$ whenever 
$\A\models \Sigma$ and $\Rd_{t'}\A\in \L$, then $\A\in \K$. This means that $\Sigma$ `captures' the properties of the operations in $t\sim t'$. 
A relative non--finite axiomatizability result is of the form:  There is no set `of a special form' of 
first order formulas satisfying a `finitary condition' that axiomatizes
$\K$ over $\L$.  Such special forms may be equations, or universal formulas. By finitary, we exclusively mean that $\Sigma$ is finite (this makes no sense if the signature at hand 
is infinite), or $\Sigma$ is a finite schema in the sense of Monk's schema \cite{Monk}, \cite[Definition 5.6.11-5.6.12]{HMT2},  
or $\Sigma$ contains only finitely many variables. The last two cases apply equally well to varieties having infinite signature
like $\RCA_{\omega}$. In the last case a finite schema  
is understood in the sense of \cite[Definition 4.1.4]{HMT2} namely,  in a two-sorted sense, one sort for ordinals $<\omega$, the other sort for the usual 
first order situation.  

A typical result of the last form is that for $2<n<\omega$, there is no set of universal formulas containing 
only finitely many variables that axiomatizes 
the variety of representable polyadic equality algebras of dimension $n$  over the variety 
of representable polyadic equality algebras of the same dimension \cite{Andreka}; a result that we lift the transfinite.
We use a yet other modified version of the method of 
`splitting atoms'  invented by Andr\'eka in \cite{Andreka}. Relative non--finite axiomatizability results excluding universal formulas containing only finitely many variables 
are obtained by  by splitting atoms in representable algebras, thereby obtaining non--representable 
ones having various representable proper subreducts.   

Though splitting techniques are associated with non--finite axiomatizability results, 
in this paper we argue and indeed demonstrate 
that there are several subtle re--incarnations of this technique existing in the literature,
proving results on notions like {\it non--atom canonicity and non--finite first order definability (definitions are provided below).}
Furthermore, the construction used in all 
such cases involves a variation on the following single theme:

{\bf Split some (possibly all) atoms in an  algebra (that need not be finite nor even atomic)  
each into one or more subatoms forming a bigger superalgebra that constitutes the starting point 
for serving  the purpose at hand.}

We also provide new ones.

Let ${\sf Nr}_n$ denote the operator of forming $n$--neat reducts and $\sf Ra$ 
denote the operator of forming relation algebra reducts. 
Here the word `splitting'  refers to splitting one or more atoms in a given 
$\RA$ or $\CA_n$ into several (possibly infinite) subatoms obtaining a bigger algebra. 

Our first family of examples involve so--called blow--up-and blur constructions, where some of the atoms in a finite 
algebra is each split to infinitely many subatoms, giving a weakly representable atom structure. 
Such technique proves useful in showing non--atom canonicity
for varieties of $\RA$s and $\CA_n$.

The second splitting 
method encountered will be splitting only one atom in a representable 
$\CA_n$ to finitely many getting a non--representable one. 
Such technique proves useful for proving non--finite axiomatizability results for varieties of $\CA_n$s, like 
$\RCA_n$s and its approximations $\bold S\Nr_n\CA_{n+k}$. 

The third splitting method proves useful in proving non--first oder definability of a given class $\bold K$.
Here the atoms of a finite algebra are split twice,  
getting two elementary equivalent infinite algebrs $\A$ and $\B$ such that 
$\A\in \bold K$ and $\B\notin \bold K$. This will be applied to the cass when $\bold K=\Nr_n\CA_{\omega}$ and $\bold K=\Ra\CA_{\omega}$.

Finally, given a finite ordinal $n>4$, 
we deal with splitting atoms in finite relation algebras, each into finitely many subatoms getting relations algebras 
that have $n$--locally `well behaved' representations
where the degree of `well behaveness' is measured by $n$. We show following Hirsch and Hodkinson, that such sophisticated splitting 
techniques give rise to  relation algebras that have $n$--well behaved representations, but no $n+1$--well behaved 
ones. 

\section{Preliminaries}

\subsection{Basics}

Algebras will be denoted by
Gothic letters, and when we write $\A$, then we will be tacitly assuming
that  the corresponding Roman letter $A$  denotes  the universe
of $\A$.
However, in some occasions we will identify (notationally)
an algebra with its universe. 
Fix an ordinal $\alpha$.

{\bf Cylindric--like algebras:} We write $\PA_\alpha$ $(\PEA_\alpha)$ and $\Sc_\alpha$ shrt hand for the classes of polyadic algebras
(with equality) and Pinter's algebras of dimension $\alpha$. $\sf Df_\alpha$ denotes the class of diagonal free cylindric algebras of
dimension $\alpha$. Here the extra non--Boolean operations are just the $\alpha$--cylindrifiers, so all algebras considered have a $\sf Df$ reduct.
The standard reference for all such classes of cylindric--like algebras is \cite{HMT2}.

We use the notation $\QA_{\alpha}$, ($\QEA_{\alpha}$) for the class
of quasi--polyadic (equality) algebras of dimension $\alpha$.
These are term definitionally equivalent to $\PA_\alpha (\PEA_\alpha)$ when $\alpha<\omega$, 
but for $\alpha\geq \omega$, they
are not quite like $\PA_{\alpha}$ ($\PEA_{\alpha})$, for their signature
contains only substitutions  indexed by replacements and transpositions, witness Figure 1. 
We deal only $\QA$s and $\QEA$s.

Given an algebra $\A$, $\Rd_{ca}\A$ denotes the cylindric reduct of $\A$ if it has one, $\Rd_{sc}\A$
denotes the $\Sc$ reduct of $\A$ if it has one, and
$\Rd_{df}\A$ denotes the reduct of $\A$ obtained by discarding all the operations except for cylindrifiers.
It is always the case that $\Rd_{df}\A\in \sf Df_{\alpha}$.

For any $\sf K$ of the above cylindric--like classes, and any ordinal $\alpha$, 
$\sf RK_\alpha$ denotes the class of representable $\K_\alpha$s. By the same token 
$\RA$ denotes the class of relation algebras and $\sf RRA$ denotes the class 
of   representable $\RA$s.
We write $\subseteq$ for inclusion, and $\subsetneq$ for
proper inclusion. $\prod$ denotes infimum, and $\sum$ denotes (its dual) supremum.  For algebras $\A$ and $\B$ having a Boolean reduct,
we write $\A\subseteq _c \B$, if for all $X\subseteq \A$, whenever $\sum^{\A} X=1$,
then $\sum^{\B} X=1$. We say that $\A$ is a {\it strong subalgebra of $\B$}.  We avoid the possibly confusing term 
`complete subalgebra'  for this lends itself to the interpretation that $\A$ is a complete algebra, 
and a subalgebra of $\B$.   
This is plainly what is not meant by $\subseteq_c$. 
For a class $\bold K$ having a Boolean reduct,
we write $\bold S_c\bold K$ for $\{\A: (\exists \B\in \K) \A\subseteq_c \B\}$.

\begin{figure}
\[\begin{array}{l|l}
\mbox{class}&\mbox{extra operators}\\
\hline
\Df_\alpha&\cyl i:i<\alpha\\
\Sc_\alpha& \cyl i, \s_i^j :i, j<\alpha\\
\CA_\alpha&\cyl i, \diag i j: i, j<\alpha\\
\PA_\alpha&\cyl i, \s_\tau: i<\alpha,\; \tau\in\;{}^\alpha\alpha\\
\PEA_\alpha&\cyl i, \diag i j,  \s_\tau: i, j<\alpha,\;  \tau\in\;{}^\alpha\alpha\\
\QA_\alpha, &  \cyl i, \s_{[i/j]}, \s_{[i, j]} :i, j<\alpha  \\
\QEA_\alpha &\cyl i, \diag i j, \s_i^j, \s_{[i, j]}: i, j<\alpha
\end{array}\]
\caption{Non-Boolean operators for the classes\label{fig:classes}}
\end{figure}

{\bf Set algebras and neat embeddings:} 
For cylindric set algebras of dimension $\alpha$, 
we follow the notation of \cite{HMT2}, often without warning. 
For example $\sf Cs_\alpha$ stands for the class of cylindric set agebras of dimension $\alpha$, $\sf Gs_\alpha$ stands for the class of generalized set algebras of dimenson $\alpha$.
It is known that $\RCA_\alpha=\bold I \sf Gs_\alpha$.

The class $\sf D_\alpha$ is the class of set algebras having the same signature as $\CA_\alpha$; if $\A\in \sf D_\alpha$, then  the 
top element of $\A$ is a set of $\alpha$--ary sequences $V\subseteq {}^\alpha U$ (some non--empty set $U$),
such that if $s\in V$,  and $i<j<\alpha$, then $s\circ [i|j]\in V$. 
The operations of $\A$ with top element $V$, whose domain is a subset of $\wp(V)$, 
are like the operations in cylindric set algebras of dimension $\alpha$,
but relativized to the top element $V.$  

In our treatment of the notion of neat reducts, and the related one of neat embeddings, 
we follow the terminolgy and notation of \cite{HMT2, Sayedneat}.
The notion of {\it neat reducts} and the related one of {\it neat embeddings} are both important in algebraic logic for the
very simple reason that they are very much tied
to the notion of representability, via the so--called neat embedding theorem of Henkin's which says that (for any ordinal 
$\alpha$),  we have $\sf RCA_{\alpha}=\bold S\Nr_{\alpha}\CA_{\alpha+\omega}$.
\begin{definition} 
Assume that $\alpha<\beta$ are ordinals, that $\K$ is any class between $\Df$ and $\QEA$, and that 
$\B\in \K_{\beta}$. Then the {\it $\alpha$--neat reduct} of $\B$, in symbols
$\mathfrak{Nr}_{\alpha}\B$, is the
algebra obtained from $\B$, by discarding
cylindrifiers and diagonal elements whose indices are in $\beta\sim \alpha$, and restricting the universe to
the set $Nr_{\alpha}B=\{x\in \B: \{i\in \beta: {\sf c}_ix\neq x\}\subseteq \alpha\}.$
\end{definition}

It is straightforward to check that $\Nr_{\alpha}\B\in \K_{\alpha}$. 
If $\A\in \K_\alpha$ and $\A\subseteq \mathfrak{Nr}_\alpha\B$, with $\B\in \K_\beta$, then we say that $\A$ {\it neatly embeds} in $\B$, and 
that $\B$ is a {\it $\beta$--dilation of $\A$}, or simply a {\it dilation} of $\A$ if $\beta$ is clear 
from context. 

For relation algebra reducts we follow \cite{HHbook2}. 
If $\beta>2$, $\B\in \CA_{\beta}$ and 
$\R=\Ra\B$ is the  algebra with $\RA$ signature defined like in \cite[Definition 5.3.7]{HMT2}, 
we refer to $\B$ also as a $\beta$--dilation of $\R$; 
if $\beta>3$, then it is known that $\Ra\B\in \RA$ \cite[Theorem 5.3.8]{HMT2}.

{\bf Atom structures and atom--canonicity:} 
We recall the notions of {\it atom structures} and {\it complex algebra} in the framework of Boolean algebras 
with operators of which $\CA$s  are a special case \cite[Definition 2.62, 2.65]{HHbook}. 
The action of the non--Boolean operators in a completely additive (where operators distribute over arbitrary joins componentwise) 
atomic Boolean algebra with operators, $(\sf BAO)$ for short, is determined by their behavior over the atoms, and
this in turn is encoded by the {\it atom structure} of the algebra.

\begin{definition}\label{definition}(\textbf{Atom Structure})
Let $\A=\langle A, +, -, 0, 1, \Omega_{i}:i\in I\rangle$ be
an atomic $\sf BAO$ with non--Boolean operators $\Omega_{i}:i\in I$. Let
the rank of $\Omega_{i}$ be $\rho_{i}$. The \textit{atom structure}
$\At\A$ of $\A$ is a relational structure
$$\langle At\A, R_{\Omega_{i}}:i\in I\rangle$$
where $At\A$ is the set of atoms of $\A$ 
and $R_{\Omega_{i}}$ is a $(\rho(i)+1)$-ary relation over
$At\A$ defined by
$$R_{\Omega_{i}}(a_{0},
\cdots, a_{\rho(i)})\Longleftrightarrow\Omega_{i}(a_{1}, \cdots,
a_{\rho(i)})\geq a_{0}.$$
\end{definition}
\begin{definition}(\textbf{Complex algebra})
Conversely, if we are given an arbitrary first order structure
$\mathcal{S}=\langle S, r_{i}:i\in I\rangle$ where $r_{i}$ is a
$(\rho(i)+1)$-ary relation over $S$, called an {\it atom structure}, we can define its
\textit{complex
algebra}
$$\Cm(\mathcal{S})=\langle \wp(S),
\cup, \setminus, \phi, S, \Omega_{i}\rangle_{i\in
I},$$
where $\wp(S)$ is the power set of $S$, and
$\Omega_{i}$ is the $\rho(i)$-ary operator defined
by$$\Omega_{i}(X_{1}, \cdots, X_{\rho(i)})=\{s\in
S:\exists s_{1}\in X_{1}\cdots\exists s_{\rho(i)}\in X_{\rho(i)},
r_{i}(s, s_{1}, \cdots, s_{\rho(i)})\},$$ for each
$X_{1}, \cdots, X_{\rho(i)}\in\wp(S)$.
\end{definition}
It is easy to check that, up to isomorphism,
$\At(\Cm(\mathcal{S}))\cong\mathcal{S}$ alway. If $\A$ is
finite then of course
$\A\cong\Cm(\At\A)$. 
For algebras $\A$ and $\B$ having the same signature expanding that of Boolean algebras, 
we say that $\A$ is dense in $\B$ if $\A\subseteq \B$ and for all non--zero $b\in \B$, there is a non--zero 
$a\in A$ such that $a\leq b$.

An atom structure will be denoted by $\bf At$.  An atom structure $\bf At$ has the signature of $\K_\alpha$, $\alpha$ an ordinal, 
if  $\Cm\bf At$ has the signature of $\K_\alpha$. 

\begin{definition}\label{canonical} 
Let $V$ be a variety of $\sf BAO$s. Then $V$ is {\it atom--canonical} if whenever $\A\in V$ and $\A$ is atomic, then $\mathfrak{Cm}\At\A\in V$.
The {\it  \de\ completion} of  $\A\in \V$, is the unique (up to isomorphisms that fix $\A$ pointwise)  complete  
$\B\in \K_n$ such that $\A\subseteq \B$ and $\A$ is {\it dense} in $\B$. 
\end{definition}

From now on fix $2<n<\omega$ and $\K$ any variety having signature between $\Df$ and $\QEA$.
An atom structure will be denoted by $\bf At$.  An atom structure $\bf At$ has the signature of $\K_n$,  
if  $\Cm\bf At$ has the signature of $\K_n$, in which case we say that $\bf At$ is an $n$--dimensional atom structure of type 
$\K$.  If $\A\in \K_n$ is atomic and completely additive (the non--Boolean unary operations distribute over joins),  then 
the complex algebra of its atom structure, in symbols $\mathfrak{Cm}\At\A$ 
is the {\it \de\ completion of $\A$.}
If $\A\in \K_n$, then its atom structure will be denoted by $\At\A$ with domain the set of atoms of $\A$ denoted by $At\A$.

{\it Non atom--canonicity} can be proved 
by finding {\it weakly representable atom structures} that are not {\it strongly representable}.

\begin{definition} An atom structure $\bf At$ of dimension $n$ and type $\K$ is {\it weakly representable} if there is an atomic 
$\A\in \RK_{n}$ such that $\At\A=\bf At$.  The atom structure  $\bf At$ is {\it strongly representable} if for all $\A\in \K_{n}$, 
$\At\A={\bf At} \implies \A\in {\sf RK}_n$.
\end{definition}
These two notions (strong and weak representability) do not coincide for cylindric algebras as proved by Hodkinson \cite{Hodkinson}.
In theorem \ref{can}, we generalize Hodkinson's result by showing that for any $\K$ having signature between $\Sc$ and $\QEA$, 
there are two atomic $\K_n$s sharing the same atom structure, one is representable and the other its \de\ completion is even outside 
$\bold S{\sf N}r_n\K_{n+3}(\supsetneq {\sf RK}_n$). In particular, there is a complete algebra outside $\bold S{\sf Nr}_n\K_{n+3}$ 
having a dense representable  subalgebra, so that 
$\bold S{\sf Nr}_n\K_{n+3}$ is not atom--canonical.

{\bf Games, networks and rainbows:}
To define certain games to be used in the sequel,  
we recall the notions of {\it atomic networks} and {\it atomic games} \cite{HHbook, HHbook2}. We require that 
networks that are `symmetric':

\begin{definition}\label{game} Fix finite $n>1$. 

(1) An {\it $n$--dimensional atomic network} on an atomic algebra $\A\in \QEA_n$  is a map $N: {}^n\Delta\to  \At\A$, where
$\Delta$ is a non--empty set of {\it nodes}, denoted by $\nodes(N)$, satisfying the following consistency conditions for all $i<j<n$: 
\begin{itemize}
\item If $\bar{x}\in {}^n\nodes(N)$  then $N(\bar{x})\leq {\sf d}_{ij}\iff x_i=x_j$,
\item If $\bar{x}, \bar{y}\in {}^n\nodes(N)$, $i<n$ and $x\equiv_i y$, then  $N(\bar{x})\leq {\sf c}_iN(\bar{y})$,
\item (Symmetry): if $\bar{x}\in {}^n\nodes(N)$, then  $\s_{[i, j]}N(\bar{x})=N(\bar{x}\circ [i, j]).$
\end{itemize}
Let $i<n$. For $n$--ary sequences $\bar{x}$ and $\bar{y}$ and $n$--dimensional atomic networks $M$ and $N$,  we write $\bar{x}\equiv_ i\bar{y}$ $\iff \bar{y}(j)=\bar{x}(j)$ for all $j\neq i$
and we write $M\equiv_i N\iff M(y)=N(y)$ for all $\bar{y}\in {}^{n}(n\sim \{i\})$.

(2)   Assume that $\A\in \QEA_n$ is  atomic and that $m, k\leq \omega$. 
The {\it atomic game $G^m_k(\At\A)$, or simply $G^m_k$}, is the game played on atomic networks
of $\A$ using $m$ nodes and having $k$ rounds \cite[Definition 3.3.2]{HHbook2}, where
\pa\ is offered only one move, namely, {\it a cylindrifier move}: 

Suppose that we are at round $t>0$. Then \pa\ picks a previously played network $N_t$ $(\nodes(N_t)\subseteq m$), 
$i<n,$ $a\in \At\A$, $x\in {}^n\nodes(N_t)$, such that $N_t(\bar{x})\leq {\sf c}_ia$. For her response, \pe\ has to deliver a network $M$
such that $\nodes(M)\subseteq m$,  $M\equiv _i N$, and there is $y\in {}^n\nodes(M)$
that satisfies $\bar{y}\equiv _i \bar{x}$ and $M(y)=a$.  

(3) We write $G_k(\At\A)$, or simply $G_k$, for $G_k^m(\At\A)$ if $m\geq \omega$.

(4) The $\omega$--rounded game $F^m(\At\A)$ or simply $F^m$ is like the game $G_{\omega}^m(\At\A)$ 
except that \pa\ has the option 
to reuse the $m$ nodes in play.
\end{definition}
\begin{lemma}\label{rep} Let $2<n<\omega$ and $\A\in {\sf QEA}_n$ be atomic with countable many atoms. 
Then \pe\ has a \ws\ in $G_{\omega}(\At\A)\iff \A$ is completely representable. In particular, if $\A$ is finite, then 
\pe\ has a \ws\ in $G_{\omega}(\At\A)\iff \A$ 
is representable.
\end{lemma}
\begin{proof} \cite[Theorem 3.3.3]{HHbook2}. 
\end{proof} 
Strictly speaking, \cite[Theorem 3.3.3]{HHbook2} is formulated for 
$\CA_n$s, but it can be easily checked that it works for $\QEA_n$s. 

Fo rainbow constructions on relation algebras, we follow the text book 
\cite{HHbook}.  
The construction for both cases (relation algebras and $\CA$s) is based on two relational structure $\sf G$ (the greens) and $\sf R$ (the reds).
For rainbow cylindric algebras, we use the graph version of the above defined games
played on coloured graphs \cite{HH} using the correspondence between networks and coloured graphs.
 
A \ws\ in both games -- for relation and cylindric algebras --of either player is dictated by a \ws\ of the same player in a simple private \ef\ forth  
game played on the relational structures $\sf G$ and $\sf R$ denoted by ${\sf EF}_r^p(\sf G, R)$ where 
$r$ is the number of rounds and $p$ is the number of pebble pairs 
in play \cite[Definition 16.2]{HHbook2}. 

Given relational structures 
$\sf G$ and $\sf R$ the rainbow 
atom structure of the rainbow $\CA_n$ consists of equivalence classes of surjective maps $a:n\to \Delta$, where $\Delta$ is a coloured graph
in the rainbow signature, and the equivalence relation relates two such maps $\iff$  they essentially define the same graph \cite[4.3.4]{HH};
the nodes are possibly different but the graph structure is the same. We let $[a]$ denote the equivalence class containing $a$.
The accessibility binary relation corresponding 
to the $i$th  cylindrifier $(i<n)$ is defined by:  $[a] T_i [b]\iff a\upharpoonright n\sim \{i\}=b\upharpoonright n\sim \{i\},$ 
and the accessibility unary relation corresponding to the $ij$th diagonal element ($i<j<n$) 
is defined by: $[a]\in D_{ij}\iff a(i)=a(j)$. We consider {\it quasi--polyadic equality atom structures of dimension $n$} 
by expanding the $\CA_n$ rainbow atom structure defining the accessibilty (binary relations) corresponding to 
transpositions ($[i, j]$, $i<j<n$) as follows:
$[a]S_{[i, j]}[b]\iff\ a=b\circ [i,j]$.
Certain finite coloured graphs play an essential role in `rainbow games'. 
The board of a rainbow game are coloured graphs:
Such special coloured graphs are called {\it cones}:

{\it Let $i\in {\sf G}$, and let $M$ be a coloured graph consisting of $n$ nodes
$x_0,\ldots,  x_{n-2}, z$. We call $M$ {\it an $i$ - cone} if $M(x_0, z)=\g_0^i$
and for every $1\leq j\leq m-2$, $M(x_j, z)=\g_j$,
and no other edge of $M$
is coloured green.
$(x_0,\ldots, x_{n-2})$
is called  the {\bf base of the cone}, $z$ the {\bf apex of the cone}
and $i$ the {\bf tint of the cone}.}

The \ws\ of \pa\ in the rainbow game played on coloured graphs played between \pe\ and \pa\  
is bombarding \pe\ with $i$--cones, $i\in \sf G$, having the 
same base 
and distinct green tints.  To respect the rules of the game \pe\ has to choose a red label for appexes of two succesive cones.  
Eventually, running out of `suitable reds',  \pe\ is forced to play an inconsistent triple of reds where indices do not match.
Thus \pa\ wins on a red clique (a graph all of whose edges are lablled by a red) with 
the \ws\ for ether player  dictated by her(his) \ws\ in a simple private \ef\ forth  
game played on the relational structures $\sf G$ and $\sf R$ with $r\leq \omega$ rounds and $p\leq \omega$ pairs of pebbles (recalled and denoted above by ${\sf EF}_r^p(\sf G, \sf R)$). 
The $n$--dimensional rainbow complex $\QEA_n$ based on $\sf G$ and $\sf R$ will be denoted by $\A_{\G, \R}$. 
The dimension $n$ will always be clear from context.

{\bf A technical lemma:} We need the following lemma to be used in the sequel.  We defer the highly 
technical proof to the appendix.
We restrict the game $F^m$ to the signature of $\Sc$s.
\begin{lemma}\label{Thm:n}
Let $2<n<m$. If either:
\begin{itemize}
\item $\A\in \Sc_n$ and $\A\in S_c\Nr_n\Sc_m^{\sf ad}$ or, 
\item $\A\in \QA_n$ and $\A\subseteq_c \mathfrak{Nr}_n\C$, $\C\in \QA_m$ and $\s_0^1$ is completely additive in $\C$ or,
\item $\K$ is any class having signature between $\CA$ and $\QEA$, 
$\A\in \K_n$ and $\A\in \bold S_c\Nr_n\K_m$,
\end{itemize} 
then \pe\ has a \ws\ in $F^m(\At\A).$ 
\end{lemma}

\section{Blow up and blur constructions; splitting some of the atoms in a finite algebra getting 
a weakly representable atom strucure that is not strongly representable}

Throughout this section, unless otherwise indicated, $n$ is a finite ordinal $>2$. 
Here we show in theorem \ref{can} that for any $k\geq 3$ and any variety $\K$ having signature between $\Sc$ and $\QEA$
the varieties $\bold S\Nr_n\K_{n+k}$ and ${\sf RDf}_n$ are {\it not atom--canonical} from which it readily follows 
(by a result of Venema \cite[Theorem 2.96]{HHbook}) that they are not 
Sahlqvist axiomatizable. \\
 
{\bf General Idea of  blow up and blur construction:} The idea of a blow up and blur construction in (more than in) a nut shell is the following. Let $2<n<\omega$. 
Let $\bold L=\CA_n$ or $\RA$ for short.
Assume that $\bf RL\subseteq \K\subseteq \bold L$, and that $\bold S\K=\K$, that is $\K$ is closed under forming subalgebras.

\begin{itemize}

\item One starts with an atomic algebra $\C\in \bold L$ (usually finite) outside $\K$. 
Then one {\it blows up and blur} $\C$, by splitting
some of its atoms each to infinitely many, 
getting a new infinite atom structure $\bf At$. 
In this process a (finite) set of `blurs' are involved in a way to be clarified next.

\item These blurs {\it do not blur} the complex algebra $\Cm\bf At$, in the sense that $\C$ is `there on this global level',
$\C$ {\it embeds into} $\Cm\bf At$.
Thus the  algebra $\mathfrak{Cm}\At$ will not be in $\K$
because $\C\notin \K$, $\C\subseteq \mathfrak{Cm}\bf At$ and $\bold S\K=\K$.
Here the {\it completeness (existence of arbitray joins) of the complex algebra} will play a major role,
because every splitted atom of $\C$,  is mapped to {\it the join} of
its splitted copies which exist in $\mathfrak{Cm}\bf At$, because it is complete; the other atoms are mapped to themselves.

\item Such precarious joins prohibiting membership in $\K$ {\it do not }exist in the term algebra $\Tm\bf At$, the subalgebra
of $\Cm\bf At$ generated by the atoms, becuase it is not complete; only joins of finite or cofinite subsets of the atoms do,
so that now {\it `blurs' blur $\C$} on 
the level of the {\it term algebra}; more succintly, $\C$ {\it does not embed} in $\Tm\bf At.$

\item In fact, it can (and will be) be arranged that $\Tm \bf At$ will not only be in $\K$, but actually it will be in (the possibly smaller) 
class $\bf RL$. This is where the blurs play another crucial role. 
Basically non-principal ultrafilters, the blurs, together with the principal ultrafilters generated by the atoms
in $\bf At$ will be  used as colours to represent  $\Tm\bf At\A$.
In the process of representation, one cannot use {\it only} principal ultrafilters,
because $\Tm\bf At$ cannot be completely representable; for else
this induces a representation of $\mathfrak{Cm}\bf At\A$. But using the blurs one can actually {\it completely represent}
$\Tm\bf At^+$ the {\it canonical extension 
of $\Tm \bf At$.} 
\end{itemize}

Let us get more concrete giving some specific examples to this subtle 
construction that proves highly efficient in proving non--atom canonicity.

\subsection{Blowing up and blurring a finite Maddux algebra}

In what follows we construct cylindric agebras from atomic relation algebras that posses  {\it cylindric basis} using so called blow up and blur constructions. 
Unless otherwise indicated $n$ wil be a finite ordinal $>2$. 

Let $\R$ be an atomic  relation algebra.  An {$n$--dimensional basic matrix}, or simply a matrix  
on $\R$, is a map $f: {}^2n\to \At\R$ satsfying the 
following two consistency 
conditions $f(x, x)\leq \Id$ and $f(x, y)\leq f(x, z); f(z, y)$ for all $x, y, z<n$. For any $f, g$ basic matrices
and $x, y<m$ we write $f\equiv_{xy}g$ if for all $w, z\in m\setminus\set {x, y}$ we have $f(w, z)=g(w, z)$.
We may write $f\equiv_x g$ instead of $f\equiv_{xx}g$.  

\begin{definition}\label{b}
An {\it $n$--dimensional cylindric basis} for an atomic relaton algebra 
$\R$ is a set $\cal M$ of $n$--dimensional matrices on $\R$ with the following properties:
\begin{itemize}
\item If $a, b, c\in \At\R$ and $a\leq b;c$, then there is an $f\in {\cal M}$ with $f(0, 1)=a, f(0, 2)=b$ and $f(2, 1)=c$
\item For all $f,g\in {\cal M}$ and $x,y<n$, with $f\equiv_{xy}g$, there is $h\in {\cal M}$ such that
$f\equiv_xh\equiv_yg$. 
\end{itemize}
\end{definition}
One can construct a $\CA_n$ in a natural way from an $n$--dimensional cylindric basis which can be viewed as an atom structure of a $\CA_n$
(like in \cite[Definition 12.17 ]{HHbook} addressing hyperbasis).
For an atomic  relation algebra $\R$ and $l>3$, we denote by ${\sf Mat}_n(\At\R)$ the set of all $n$--dimensional basic matrices on $\R$.
${\sf Mat}_n(\At\R)$ is not always an $n$--dimensional cylindric basis, but sometimes it is,
as will be the case described next. On the other 
hand, ${\sf Mat}_3(\At\R)$ is always a  $3$--dimensional cylindric basis; a result of Maddux's, so that 
$\Cm{\sf Mat}_3(\At\R)\in \CA_3$.

Non--atom canonicity of $\bold S{\sf Nr}_n\CA_m$ for $m>n$  follows from 
the existence of 
finite relation algebras having a so--called $n$--blur and no infinite $m$--dimensional hyperbasis (to be defined next).
In the limiting case, upon identifying an infinite $\omega$--hyperbasis with an ordinary representation, such  (finite) relation 
algebras exists; furthermore for each $n\leq l<\omega$, such an $\R_l$ exists 
having a {\it strong} $l$--blur which is stronger than having merely an $l$--blur as the name suggests. 
Such algbras were denoted by $\mathfrak{E}_k(2, 3)$ in \cite[Lemma 5]{ANT} where
$k$, the finite number of non--identity atoms, depends recursively on $l$. 
Constructed by Maddux,  a triple $(a, b, c)$ of non--identity atoms of $\mathfrak{E}(2, 3)$ 
is consistent $\iff$ $|\{a, b, c\}|\neq 1$. That is, only monochromatic triangles are forbidden. 
It is known, that such relation algebra (for $k<\omega$) 
if representable  will have to be represented on a finite base.
((Strong) $l$--blurness is defined in next \ref{strongblur}.)
   
For the next lemma, we refer the reader to \cite[Definition 12.11]{HHbook} 
for the definition of hyperbasis for relation algebras. We use:
\begin{lemma}\label{i} Let $\R$ be  finite relation algebra and $3<n<\omega$. 
Then $\R$ has an $n$--dimensional infinite hyperbasis $\iff\ \R$ has an infinite $n$--flat representation.
\end{lemma}
\begin{proof} \cite[Theorem 13.46, the equivalence $(7)\iff (11)$]{HHbook}.
\end{proof}

The next definition to be used in the sequel is taken from \cite{ANT}:

To violate various forms of omitting types theorems we build $\CA_n$s from finite 
relation algebras.
For a relation algebra $\R$, we let $\R^+$ denotes its canonical extension.

\begin{lemma}\label{i} Let $\R$ be  a relation algebra and $3<n<\omega$.  Then the following hold:
\begin{enumerate}
\item $\R^+$ has an $n$--dimensional infinite basis $\iff\ \R$ has an infinite $n$--square representation.

\item $\R^+$ has an $n$--dimensional infinite hyperbasis $\iff\ \R$ has an infinite $n$--flat representation.
\end{enumerate}
\end{lemma}
\begin{proof} \cite[Theorem 13.46, the equivalence $(1)\iff (5)$ for basis, and the equivalence $(7)\iff (11)$ for hyperbasis]{HHbook}.
\end{proof} 
The following definition to be used in the sequel is taken from \cite{ANT}:
\begin{definition}\label{strongblur}\cite[Definition 3.1]{ANT}
Let $\R$ be a relation algebra, with non--identity atoms $I$ and $2<n<\omega$. Assume that  
$J\subseteq \wp(I)$ and $E\subseteq {}^3\omega$.
\begin{enumerate}
\item We say that $(J, E)$  is an {\it $n$--blur} for $\R$, if $J$ is a {\it complex $n$--blur} defined as follows:   
\begin{enumarab}
\item Each element of $J$ is non--empty,
\item $\bigcup J=I,$
\item $(\forall P\in I)(\forall W\in J)(I\subseteq P;W),$
\item $(\forall V_1,\ldots V_n, W_2,\ldots W_n\in J)(\exists T\in J)(\forall 2\leq i\leq n)
{\sf safe}(V_i,W_i,T)$, that is there is for $v\in V_i$, $w\in W_i$ and $t\in T$,
we have
$v;w\leq t,$ 
\item $(\forall P_2,\ldots P_n, Q_2,\ldots Q_n\in I)(\forall W\in J)W\cap P_2;Q_n\cap \ldots P_n;Q_n\neq \emptyset$.
\end{enumarab}
and the tenary relation $E$ is an {\it index blur} defined  as 
in item (ii) of \cite[Definition 3.1]{ANT}.

\item We say that $(J, E)$ is a {\it strong $n$--blur}, if it $(J, E)$ is an $n$--blur,  such that the complex 
$n$--blur  satisfies:
$$(\forall V_1,\ldots V_n, W_2,\ldots W_n\in J)(\forall T\in J)(\forall 2\leq i\leq n)
{\sf safe}(V_i,W_i,T).$$ 
\end{enumerate}
\end{definition}
The following theorem concisely summarizes 
the construction in \cite{ANT} and says some more easy facts.

\begin{theorem}\label{ANT} Let $2<n\leq l<\omega$.
Let $\R$ be a finite relation algebra with an $l$--blur $(J, E)$ where $J$ is the $l$--complex blur and $E$ is the index blur, as in definition \ref{strongblur}.  
\begin{enumerate}
\item Then for ${\cal R}={\sf Bl}(\R, J, E)$, with atom structure $\bf At$ obtained by blowing up and blurring $\R$ 
(with underlying set is denoted by $At$ on \cite[p.73]{ANT}), the set of $l$ by $l$--dimensional matrices 
${\bf At}_{ca}={\sf Mat}_l(\bf At)$ is an $l$--dimensional cylindric basis, that is a weakly representable atom structure \cite[Theorem 3.2]{ANT}. 
The algebra ${\sf Bl}_l(\R, J, E)$, with last notation as in \cite[Top of p. 78]{ANT} having atom structure ${\bf At}_{ra}$ is in $\RCA_l$. Furthermore, 
$\R$ embeds into $\Cm{\bf At}$ which embeds into $\Ra\Cm({\bf At}_{ca}).$ 

\item For very $n<l$, 
there is an $\R$ having a strong $l$--blur 
but no finite representations. Hence $\bf At$ obtained by blowing up and blurring $\R$ and 
the $\CA_n$ atom structure ${\bf At}_{ca}$ as in the previous item are 
not strongly representable.

\item Let $m<\omega$. If $\R$ is  as in the hypothesis, 
$(J, E)$ is a strong $l$--blur, and $\R$ has no $m$--dimensional hyperbasis, 
then $l<m$.

\item If $n=l<m<\omega$ and $\R$ as above has no infinite $m$--dimensional hyperbasis, then $\Cm\At{\sf Bb}_l(\R, J, E)\notin 
\bold S{\sf Nr}_n\CA_m$, and the latter class is not atom--canonical.

\item If $2<n\leq l<m\leq \omega$, and $(J, E)$ is a strong $m$--blur,  definition \ref{strongblur},
then $(J, E)$ is a strong $l$--blur,  ${\sf Bl}_l(\R, J, E)\cong \mathfrak{Nr}_l{\sf Bl}_m(\R, J, E)$ and 
${\cal R}\cong \Ra {\sf Bl}_l(\R, J, E)\cong \Ra{\sf Bl}_m(\R, J, E).$ 
\end{enumerate}
\end{theorem}
\begin{proof} Cf.  \cite[ For notation, cf. p.73, p.80, and for proofs cf. Lemmata 3.2, 4.2, 4.3]{ANT}. 
We start by an outline of (1).  Let $\R$ be as in the hypothesis. The idea is to blow up and blur $\R$ in place of the Maddux algebra 
$\mathfrak{E}_k(2, 3)$ dealt with in \cite[Lemma 5.1]{ANT}   (where $k<\omega$ is the number of non--identity atoms 
and it depends on $l$). 

Let $3<n\leq l$. We blow up and blurr $\R$ as in the hypothesis. $\R$ is blown up by splitting all of the atoms each to infinitely many.
$\R$ is blurred by using a finite set of blurs (or colours) $J$. This can be expressed by the product ${\bf At}=\omega\times \At \R\times J$,
which will define an infinite atom structure of a new
relation algebra. (One can view such a product as a ternary matrix with $\omega$ rows, and for each fixed $n\in \omega$,  we have the rectangle
$\At \R\times J$.)
Then two partitions are defined on $\bf At$, call them $P_1$ and $P_2$.
Composition is re-defined on this new infinite atom structure; it is induced by the composition in $\R$, and a ternary relation $E$
on $\omega$, that `synchronizes' which three rectangles sitting on the $i,j,k$ $E$--related rows compose like the original algebra $\R$.
This relation is definable in the first order structure $(\omega, <)$.
The first partition $P_1$ is used to show that $\R$ embeds in the complex algebra of this new atom structure, namely $\Cm \bf At$, 
The second partition $P_2$ divides $\bf At$ into {\it finitely many (infinite) rectangles}, each with base $W\in J$,
and the term algebra denoted in \cite{ANT} by ${\sf Bb}(\R, J, E)$ over $\bf At$, consists of the sets that intersect co--finitely with every member of this partition.

On the level of the term algebra $\R$ is blurred, so that the embedding of the small algebra into
the complex algebra via taking infinite joins, do not exist in the term algebra for only finite and co--finite joins exist
in the term algebra. 
The algebra ${\sf Bb}(\R, J, E)$ is representable using the finite number of blurs. These correspond to non--principal ultrafilters
in the Boolean reduct, which are necessary to
represent this term algebra, for the principal ultrafilter alone would give a complete representation,
hence a representation of the complex algebra and this is impossible.
Thereby, in particular, as stated in theorem \ref{ANT} an atom structure that is weakly representable but not strongly representable is obtained.

Because $(J, E)$ is a complex set of $l$--blurs, this atom structure has an $l$--dimensional cylindric basis, 
namely, ${\bf At}_{ca}={\sf Mat}_l(\bf At)$. The resulting $l$--dimensional cylindric term algebra $\Tm{\sf Mat}_l(\bf At)$, 
and an algebra $\C$ having tatom structure ${\bf At}_{ca}$ denoted in \cite{ANT} by 
${\sf Bb}_l(\R, J, E)$, such that $\Tm{\sf Mat}_l({\bf At})\subseteq \C\ \subseteq \Cm{\sf Mat}_l(\bf At)$ 
is shown to be  representable.

For (2):  The Maddux relation algebra $\mathfrak{E}_k(2,3)$ 
Like in \cite[Lemma 5.1]{ANT},  one take $l\geq 2n-1$, $k\geq (2n-1)l$, $k\in \omega$, and then take the finite
integral relation algebra ${\mathfrak E}_k(2, 3)$ 
where $k$ is the number of non-identity atoms in
${\mathfrak E}_k(2,3)$. 
with $k$ depending on $l$ as in \cite[Lemma 5.1]{ANT} is the required $\R$ in (2).

We prove (3). Let $(J, E)$ be the strong $l$--blur of $\R$. Assume  for contradiction that $m\leq l$. Then we get by \cite[item (3), p.80]{ANT},  
that  $\A={\sf Bb}_n(\R, J, E)\cong \mathfrak{Nr}_n{\sf Bb}_l(\R, J, E)$.  But the cylindric $l$--dimensional algebra ${\sf Bb}_l(\R, J, E)$ is atomic,  having atom structure  
${\sf Mat}_l \At({\sf Bb}(\R, J, E))$, so $\A$ has an atomic $l$--dilation.
Hence $\A=\Nr_n\D$ where $\D\in \CA_l$ is atomic.
But $\R\subseteq_c \Ra\Nr_n\D\subseteq_c \Ra\D$. 
Hence $\R$ has a complete $l$--flat representation, 
hence a complete $m$--flat representation, because $m<l$ and $l\in \omega$. 
This is a contradiction.

We prove (4). Assume that $\R$ is as in the hypothesis.
Take $\B={\sf Bb}_n(\R, J, E)$. Then by the above $\B\in \RCA_n$. We claim that $\C=\Cm\At\B\notin \bold S{\sf Nr}_n\CA_{m}$.
To see why, suppose for contradiction that $\C\subseteq \mathfrak{Nr}_n\D$,
where $\D$ is atomic, with $\D\in \CA_{m}$.  Then $\C$ has a (necessarily infinite $m$--flat representation), hence 
$\Ra\C$ has an infinite $m$--flat representation as an $\RA$.  But $\R$ embeds into $\Cm\At({\sf Bb}(\R, J, E)$) which, in turn, 
embeds  into $\Ra\C$, so $\R$ has an infinite $m$--flat representation. By lemma \ref{i}, 
$\R$ has a $m$--dimensional infinite hyperbases which
contradicts the hypothesis.    

Now we prove (the last) item (5). For $2<n\leq l<m<\omega.$
If the $m$--blur happens to be {\it strong}, in the sense of definition \ref{strongblur} and $n\leq l<m$
then we get by \cite[item (3) pp. 80]{ANT},  that ${\sf Bb}_l(\R, J, E)\cong \mathfrak{Nr}_l{\sf Bb}_m(\R, J, E)$.
This is proved by defining  an embedding 
$h:\Rd_l{\sf Bb}_m(\R, J, E)\to {\sf Bb}_l(\R, J, E)$ 
via  $x\mapsto \{M\upharpoonright l\times l: M\in x\}$ and showing that  
$h\upharpoonright  \mathfrak{Nr}_l{\sf Bb}_m(\R, J, E)$ 
is an isomorphism onto ${\sf Bb}_l(\R, J, E)$ \cite[p.80]{ANT}. 
Surjectiveness uses the condition $(J5)_l$.  The resulting $l$--dimensional cylindric term algebra $\Tm{\sf Mat}_l(\bf At)$, 
and an algebra $\C$ having tatom structure ${\bf At}_{ca}$ denoted in \cite{ANT} by 
$\Bb_l(\R, J, E)$, such that $\Tm{\sf Mat}_l({\bf At})\subseteq \C\ \subseteq \Cm{\sf Mat}_l(\bf At)$ 
is shown to be  representable. The complex algebra $\Cm{\sf Mat}_l(\bf At)=\Cm \At\C$
is outside in $\bold S{\sf Nr}_n\CA_{m}$, because $\R$ 
embeds into $\Cm\bf At$ which embeds into $\Ra\Cm{\sf Mat}_l(\bf At)$, so if 
$\Cm{\sf Mat}_l({\bf At})\in \bold S{\sf Nr}_n\CA_{m}$, 
then $\R\in {\sf Ra}\bold S{\sf Nr}_n\CA_m\subseteq \bold S{\sf Ra}\CA_m$
which is contrary to assumption.
\end{proof}

We prove a non -finite axiomatizability result using the construction in \cite{ANT} recalled in the second item of heorem \ref{ANT}.
 Let ${\sf LCA}_n$ denote the elementary class of ${\sf RCA}_n$s satisfying the Lyndon conditions. 
We stipulate that $\A\in {\sf LCA}_n\iff$ $\A$ is atomic and $\At\A$ satifies the Lyndon conditions \cite[Definition 3.5.1]{HHbook2}. 
In the same sense, let $\sf LCRA(\subseteq \sf RRA)$ denote  the elementary class of $\sf RA$s satsfying the Lyndon conditions \cite[Definition pp. 337]{HHbook}.
We now use the above construction
proving non--atom canonicity to prove non--finite 
axiomatizability.
We use `bad' non--representable 
Monk--like algebras converging to a `good' representable one. 
In the process, we recover the results of Monk and Maddux on non--finite axiomatizability of both 
${\sf RCA}_n$  $(2<n<\omega)$  
and $\sf RRA$.

Let $2<n<\omega$. Then $\sf LCRA$ and ${\sf LCA}_n$ are 
not finitely axiomatizable.
For each $2<n\leq l<\omega$, 
let $\R_l$ be the finite Maddux algebra $\mathfrak{E}_{f(l)}(2, 3)$ with strong $l$--blur
$(J_l,E_l)$ and $f(l)\geq l$ denoted by $l<k<\omega$ in \cite[Lemma 5]{ANT}.  
Let ${\cal R}_l={\sf Bb}(\R_l, J_l, E_l)\in {\sf RRA}$ and let $\A_l=\mathfrak{Nr}_n{\sf Bb}_l(\R_l, J_l, E_l)\in \RCA_n$. 
Then  $(\At{\cal R}_l: l\in \omega\sim n)$, and $(\At\A_l: l\in \omega\sim n)$ are sequences of weakly representable atom structures 
that are not strongly representable with a completely representable 
ultraproduct. The (complex algebra) sequences $(\Cm \At{\cal R}_l: l\in \omega\sim n)$, 
$(\Cm\At\A_l: l\in \omega\sim n$) are typical examples of what Hirsch and Hodkinson call `bad Monk (non--representable) algebras' 
converging to  `good (representable) one, namely their (non-trivial) ultraproduct. 

Also, for $2<n\leq k <m<\omega$, $\A_k=\mathfrak{Nr}_k\A_m$. Such sequences witness the non--finite axiomatizability of the class 
representable agebras and the elementary closure of the class completely representable ones, namely, 
the class of algebras satisfying the Lyndon conditions. This recovers Monk's and Maddux's classical results  
on non--finite axiomatizability of $\sf RRA$s and $\RCA_n$s
since algebras considered are generated by a 
single $2$--dimensional elements. Hirsch and 
Hodkinson apply the words `good and bad' to graphs on which Monk--like algebras are based. A graph is bad if it has a finite colourng, else it is good.
What plays the role of  colourings 
here is the number of blurs used.

\subsection{Blowing up and blurring 
finite rainbow algebras}

We briefly review the blow up and blur construction in \cite[17.32, 17.34, 17.36]{HHbook} for relation algebras.
This time we blow up and blur a finite {\it rainbow} relation algebra (as opposed to a Maddux one).
We pave the way for another construction involving blowing up and blurring a finite rainbow $\CA_n$ ($2<n<\omega$). 
We follow the notation in \cite[ lemmas 17.32, 17.34, 17.35, 17.36]{HHbook2} with the sole exception that we denote by 
$m$ (instead of $\bold K_m$) 
the complete irreflexive graph on $m$ defined the obvious way; that is we identify this graph with its set of vertices. We denote the rainbow relation 
algebra based on $\sf G$ and $\sf R$, by $\bold R_{\sf G, \R}$
The definitions of the term algebra $\T$, and the ultrafilters $\delta$ and $\rho$ mentioned in the statement of the theorem can 
be found in \cite[p. 532]{HHbook}
and  the ultrafilters will be recalled in the proof.
\begin{theorem}\label{b2} Let $\R=\bold R_{m, n}$, $2<n<m<\omega$.
Let $\T$ be the term algebra obtained by splitting the reds as in {\it op.cit}. Then $\T$ has exactly two non--principal 
ultrafilters $\delta$ and $\rho$ where  $\rho$ is a (flexible) non--principal ultrafilter consisting of reds with distinct indices and $\delta$ is the reds with common indices.
Furthermore, $\T$ is representable, but $\mathfrak{Cm}\At \T\notin  \bold S\Ra\CA_{6}$, least representable.
\end{theorem}
\begin{demo}{Sketch} The blown up and blurring is done by splitting the red atoms in the  finite rainbow algebra given in the hypothesis getting  
a weakly representable 
atom structure $\alpha$ 
such that 
$\Cm\alpha\notin \bold S\Ra\CA_{m+2}$ for any $k\geq 6$.
Let $\R=\bold R_{m, n}$ with $3<n<m<\omega$ be as in the hypothesis. Then  \pe\ has a \ws\ in $G_{m+1}^{m+2}(\At\bold R)$, since it clealy 
has a \ws\ 
in the \ef\ game ${\sf EF}_m^m(m, n)$ because $m$ is `longer' than $n$. Here we use the {\it rainbow theorem} \cite[Theorem 18.5]{HHbook}. 
Thus $\bold R\notin \RA_{m+2}$ by \cite[Propsition 12.25, Theorem 13.46 $(4)\iff (5)$]{HHbook}. 
For $5\leq l<\omega$, $\RA_l$  is the class of relation algebras whose canonical extensions have an $l$--dimensional relational basis \cite[Definition 12.30]{HHbook}.
This last class is a variety  \cite[Proposition 12.31]{HHbook}
(properly)  containing
the variety $\bold S\Ra\CA_l$ \cite[Remark 15.13]{HHbook}. 
Hence we readily conclude that $\bold R_{m, n}\notin  \bold S\Ra\CA_{m+2}$. 
Now one `splits'  every red atom  to $\omega$--many copies obtaining the infinite atomic countable 
(term) relation algebra $\T$ with atom structure $\alpha$, cf. \cite[item (4) top of p. 532]{HHbook}. 
Then ${\cal C}=\mathfrak{Cm}\alpha\notin \bold S\Ra\CA_{m+2}$ 
because $\bold R$ embeds into $\Cm\alpha$ by mapping ever red
to the join of its copies, and $\bold S\Ra\CA_{m+2}$ is closed under $\bold S$. 

Now we (completely) represent (the canonical extension of) the term algebra.
Let $D=\{r_{ll}^n: n<\omega, l\in n\}$, and $R=\{r_{lm}^n, l,m\in n, l\neq m\}$.
If $X\subseteq R$, then $X\in \T$ $\iff$ $X$ is finite or cofinite in $R$ and same for subsets of $D$ \cite[Lemma 17.35]{HHbook}..
Let $\delta=\{X\in T: X\cap D \text { is cofinite in $D$}\}$,
and $\rho=\{X\in T: X\cap R\text { is cofinite in $R$}\}$.

Then these are the only non--principal ultrafilters;  they together with the principal ultrafilters of $\T$ corresponding to the atoms, suffice 
to (completely) represent (the canonical extension) $\T$ as follows \cite[Lemma 17.6]{HHbook}: 
Let $\Delta$ be the graph $n\times \omega\cup m\times \{\omega\}$.
Let $\B$ be the rainbow algebra obtained from $\bold R_{m, \Delta}$ by
deleting all red atoms $\r_{ij}$ where $i,j$ are
in different connected components of $\Delta$.

Clearly \pe\ has a \ws\ in ${\sf EF}_{\omega}^{\omega}(m, m)$, and so (using the rainbow theorem) she has a \ws\ in
$G_{\omega}^{\omega}\At(\bold R_{m, m})$.
But $\At\bold R_{m, m}\subseteq \At\B\subseteq \At\bold R_{m, \Delta}$, and so $\B$ is representable.
One  next defines a bounded morphism from $\At\B$ to the ultrafilters of $\T$ which constitute the domain of $\T^+$. 
The two non--principal ultrafilters are images of elements from
$m\times \{\omega\}$, by mapping the red with equal double index,
to $\delta$, and those with distinct indices to $\rho$.
The first copy is reserved to define the rest of the red atoms the obvious way.
The other atoms are the same in both structures.
\end{demo}

Observe that in the above proof the (smaller) parameter $n$ does not play any role and that it 
can well be fixed to be $3$ getting the same result (for any $m>3$). 
One might be tempted to obtain the analogous result for $\CA_n$s, $2<n<\omega$, 
by using the above construction for relation algebras resorting perhaps to  the construction of Hodkinson in \cite{ch}. 
The construction in {\it op.cit}  constructs $\CA$s and $\QEA$s (of every finite dimension $>2$)
from a given atom structure of $\sf RA$s like was done in theorem \ref{ANT}  using cylindric basis. 
But we hasten to add that this {\it cannot be done} with the construction of Hodkinson's as it stands, 
because the atomic $\sf RA$ does not embed in
$\sf Ra$ reduct of the atomic $\CA_n$ constructed from it, if $n\geq 6$. 
Nevertheless, the construction used in theorem \ref{b2}, outlined above,  can be used to show that $\bold S\Nr_n\CA_{n+k}$ for $k\geq 6$ 
is not closed under \de\ completions which is {\it weaker} than non--atom canonicity. 
The latter property implies the former, but not vice--versa. We omit the fairly straightforward proof of this.\footnote{This is an insightful 
observation of Maddux's.}

Our next new blow up and blur construction restricted to $\CA$s, showing that for $2<n<\omega$ and $k\geq 3$, $\bold S\Nr_n\CA_{n+k}$ is not atom--canonical,  has affinity
to the previous proof of theorem \ref{b2}. Worthy of note is that the term `blow up and blur' was not used by Hirsch and Hodkinson 
in their construction recalled in (the proof of) theorem \ref{b2}.  
The Maddux algebra, denoted in \cite[Lemma 5.1]{ANT} and in theorem \ref{ANT} 
by $\R=\mathfrak{E}_k(2, 3)$, renders a cylindric basis of dimension $n$, for any finite $n>3$ by adjusting its number of non--identity atoms
$k$ for the term algebra obtained after blowing it up and blurring $\R$, but it only 
witnesses
non--atom canonicity for $\sf RRA=\bold S\Ra\CA_{\omega}$ and $\sf RCA_n$. 
On the other hand, taking the special case dealt with in \cite[Lemma 17.32]{HHbook} the relation algebra (when $m=4$ and $n=3$), $\bold R_{4, 3}$ used in the proof of 
theorem \ref{b2} does not have an $n$--dimensional cylindric base for $n>3$, hence it works only for relation
algebras and $\CA_3$s and not for higher finite dimensions.
Nevertheless, it has the supreme advantage that it witnesses non atom--canonicity of $\bold S\Ra\CA_k$ for each $k\geq 6$ 
where $k$ is determined by the number of `pebble pairs' in the \ef\ forh game
${\sf EF}_4^4(4, 3).$ This is the sharpest result obtained by taking in the relation algebra $\bold R_{m, n}$ 
the least possible values of $m$ and $n$, 
namely, $m=4$ and $n=3.$
 
Let $2<n<\omega$. We want a finite $n$--dimensional cylindric algebra
that also witnesses non atom--canonicity of $\bold S\Nr_n\CA_{n+k}$ for some finite $k>1$ 
and,  we further want to control this $k$ as much as possible; the smaller $k$ is the better.
Rainbows offer solace here by adjusting the number of `pebble pairs' in the same \ef\ forth game 
and lifting it to rainbow $\CA_n$s. This game will be ${\sf EF}_{n+1}^{n+1}(n+1, n)$ lifted to the 
$n$--dimensional rainbow algebra $\A_{n+1, n}$. 

Note that this is the same game used for relation algebras  
when $n=3$. We could not lift the construction from $\RA$s to $\CA$s by applying 
Hodkinson's construction in \cite{ch}
to the relation algebra $\T$, {\it  but what we can (and will) do is to `lift' the parameters $4$ and $3$  appearing as indicies 
in $\bold R_{4,3}$}.  In the $\CA_n$ context, we will blow up and blur 
the rainbow algebra $\A_{n+1, n}$.  Since, roughly, $\CA_3$ is the natural vehicle for relation algebras, 
it seems that this is most  natural (generalization) thing to do. 

Blowing up and blurring  
$\A_{n+1, n}$ will give a weakly representable atom structure $\alpha$, 
such that $\Cm\alpha\notin \bold S\Nr_n\CA_{n+3}$. The case $\bold S\Nr_n\CA_{n+2}$ not covered by theorem \ref{can} 
was approached conditionally (and differently) in theorem \ref{ANT} by blowing up a finite relation algebra having an $n$--blur
and no infinite dimenional $n+2$--dimensional hyperbasis (equivalently having no  $n+2$--flat representation).

In \cite{Hodkinson}, Hodkinson proves that ${\sf RCA}_n$ ($2<n<\omega)$
is not atom--canonical. Hodkinson's proof is semantical; ours is syntactical implemented by blowing up and blurring
a finite rainbow polyadic-equality algebra, in which `the number' of greens is $n+1$ and
the reds $n$.
The blow up and blur addition,  will allow us to refine and indeed strengthen
Hodkinson's result,  showing that for any class $\bold K$, such that $\bold S\Nr_n\CA_{n+3}\subseteq \bold K\subseteq \RCA_n$, $\bold K$
is  not atom--canonical.

This applies to the infinitely many varieties $\bold S\Nr_n\CA_{n+k}$, $k\geq 3$ \cite{t}.
The next lemma will enable to obtain the result (on non--atom canonicity) for ${\sf RDf}_n$ 
by bouncing it back to the (known) $\RCA_n$ case. It is generally useful to transfer results from $\RCA_n$s to their diagonal free reducts.
Also, it generalizes a result of Johnson \cite[Theorem 5.4.26]{HMT2}. Henceforth, we write $\Rd_{df}$ short hand for 
`diagonal free reduct'.

\begin{lemma}\label{dfb} Let $2<n<\omega$. Assume that $\A\in \CA_n$, $\Rd_{df}\A$ is a diagonal free  
cylindric set algebra (of dimension $n$) with base $U$,  and $R\subseteq U\times U$ are as in the hypothesis of \cite[Theorem 5.1.49]{HMT2}. 
Let $E=\{x\in A:  (\forall x, y\in {}^nU)(\forall i <n)(x_iR y_i\implies (x\in X\iff y\in X))\}$.
Then $\{x\in \A: \Delta x\neq n\}\subseteq E$ and $E\in \CA_n$ 
is closed under infinite intersections. In particular, if $\A\in \CA_n$, is such that its diagonal free reduct is representable,
and $\A$ is generated by $\{x\in \A: \Delta x\neq n\}$ using infinite intersections (together with the other cylindric operations) then
$\A\in \RCA_n$.
\end{lemma}
\begin{proof} \cite[Lemma 5.1.50, Theorem 5.1.51]{HMT2}. 
In the former lemma, using the notation in {\it op.cit}, 
one just has to check that $E$ (as defined above) is closed under infinite intersections. 
This is completely straightforward following 
directly from the  definition of $E.$ 
In more detail, let $X_j: j\in J$ be in $E$. We will show that $\bigcap_{j\in J}X_j\in E$. 
Let $x, y\in {}^{n}U$ such that $x_iRy_i$ for all $i<\alpha$, and assume that $x\in \bigcap_{j\in J} X_j$. Then $x\in X_j$ for every $j\in J$. 
Now fix  $i\in J$. Then $x\in X_i$, and $X_i\in E$, so by definition of $E$ we get that  $y\in X_i$. Since $i$ was arbitrary, 
we get that $y\in \bigcap_{j\in J}X_j.$ By symmetry we are done.
\end{proof}

The following lemma to be used in the proof of theorem \ref{can} is proved for \cite{mlq}. 
\begin{lemma}\label{n}
Let $2<n<m<\omega$. Let $\A\in \CA_n$ be atomic.  Then 
Then \pe\ has a \ws\ in $G^m_{\omega}(\At\A)\iff  \A\subseteq_c \mathfrak{Nr}_n\D$, $\D\in {\sf D}_m$.
\end{lemma}

Having the necessary tools at hand, now we ready to formulate and prove:
\begin{theorem}\label{can} 
Let $n$ be a finite ordinal $>2$ and $\K$ is a class between $\Sc$ and $\QEA$.
Assume that $m\geq n+3$. Then the varieties 
$\bold S{\sf Nr}_n\K_m$, ${\sf RDf}_n$, 
$\bold S\Nr_n{\sf G}_m$, and 
$\bold S\Nr_n{\sf D}_m$ are not atom--canonical. In particular, such varieties are not closed under \de\ completions 
and are not Sahlqvist axiomatizable.   
\end{theorem}
\begin{proof}
\cite{mlq}  We start with finite dimensions.  We first give the general idea for $\CA_n$ with $2<n<\omega$.
We use a rainbow construction. With a slight abuse of notation we denote the rainbow $\CA_n$ 
based on $\sf G$ (the greens) and $\sf R$ by $\A_{\sf G, \R}$. 
Fix $k\geq 3$  (possibly infinite).  We start with the 
rainbow algebra $\A_{n+k-2, 1}$ which is finite $\iff$ $k<\omega$. 
The dimension $d=n+k$ (here if $k$ is infinite we mean ordinal addition so that $n+k=k$)
for which we can prove that $\V=\bold S\Nr_n\CA_{n+k}$ is not atom--canonical,
is determined by the number of greens we start off with, which we denote by  $\sf num(g)$. 
We have ${\sf num({\sf g})}=n+k-2$ so that
$n+k=2+ {\sf num(\sf g)}$. 

Here the number $2$ is the {\it increase resulting from  lifting  the \ws\ of \pa\ in the \ef\ forth private game between \pa\ and \pe,
namely, ${\sf EF}_{n+k-2}^{n+k-2}(n+k-2, n)$ played
on the complete irreflexive graphs
$n+k-2$ (the greens) and $n$ (the reds),   to the number of nodes used by \pa\ to implement his \ws\ in the `graph game' played 
on (networks of) the rainbow algebra $\A_{n+k-2, n}$ which is $n+k$.}

We blow up and blur $\A_{n+k-2, n}$ by splitting the ` $n$ red atoms'. This is not entirely  accurate because
we will  be  splitting  red  graphs (a graph graph being one with at least one edge labelled by a red) each into $\omega$ many.
The outcome of this splitting will be a representable
countable atomic algebra $\A$ similar to the term algebra used by
Hodkinson in \cite[Definition 4.1]{Hodkinson}; when $k<\omega$;  the only difference is  that we use only $n+k-2$ greens not infinitely many,
when $k=\omega$ it is the same.
As long as their number outfits the reds, \pa\ can win in a finite rounded game. The \de\ completion of $\A$, call it $\C$, will be
outside $\bold S\Nr_n\CA_{n+k}$, because $\A_{n+k-2, n}$ is outside $\bold S\Nr_n\CA_{n+k}$
by the fact that  \pa\  has a \ws\  in the game $G_{\omega}^{n+k}(\At\A_{n+k-2})$ (in only finitely many rounds)
and $\A_{n+k-2, n}$ embeds into $\C$. 
Here, unlike the number of nodes, which determines 
when the complex algebra $\C=\Cm\At\A$ {\bf stops to be representable}, the number of rounds is 
irrelevant. 

So although $\A$ has an ($\omega$--square) representation,
by lemma \ref{Thm:n}, its \de\ completion $\C$ will not even have an $n+k$--square representation.
This implies that $\C$ does not have an $n+k$--dilation, even if we do not require full fledged commutativity 
of cylindrifiers in this $n+k$--dilation; the dilation may not even be a ${\sf D}_{n+k}$.
It is precisely at this point  (dimension$=n+k$)  
that $\C$ stops to be representable and this point is determined by the number of greens we started off with;  recall that ${\sf num}(\g)=n+k-2$. We cannot take the parameter $k$ to be equal to $2$, 
because in this case, ${\sf num}(\g)=n$, and so the particular argument here does not work
work since $n$ is `not longer' than $n$.  Hence we 
approached the case $k=2$  separately, though only conditionally.

We work with $k=3$ and any $\K$ between $\Sc$ and $\QEA$. 
This gives the result for any larger $k$. 
Now for the details. Fix $2<n<\omega$.

{\bf Blowing up and blurring  $\A_{n+1, n}$ forming a weakly representable atom structure $\At$}:
Take the finite polyadic equality algebra rainbow algebra $\A_{n+1, n}$
where the reds $\sf R$ is the complete irreflexive graph $n$, and the greens
are  ${\sf G}=\{\g_i:1\leq i<n-1\}
\cup \{\g_0^{i}: 1\leq i\leq n+1\}$, endowed with the polyadic operations.
Denote its finite atom structure by ${\bf At}_f$; 
so that ${\bf At}_f=\At(\A_{n+1, n})$.

One  then replaces the red colours 
of the finite rainbow algebra of $\A_{n+1, n}$ each by  infinitely many reds (getting their superscripts from $\omega$), obtaining this way a weakly representable atom structure $\bf At$.
The resulting atom structure after `splitting the reds', namely, $\bf At$,  is 
like the weakly but not strongly representable 
atom structure of the atomic, countable and simple algebra $\A$ as defined in \cite[Definition 4.1]{Hodkinson}, the sole difference is that we have $n+1$ greens
and not infinitely many as is the case in \cite{Hodkinson}. We denote our algebra also by $\A$. 

The rainbow signature \cite[Definition 3.6.9]{HHbook2} $L$ now consists of $\g_i: 1\leq i<n-1$, $\g_0^i: 1\leq i\leq n+1$,
$\w_i: i<n-1$,  $\r_{kl}^t: k<l< n$, $t\in \omega$,
binary relations, and $n-1$ ary relations $\y_S$, $S\subseteq n+1$.
There is a shade of red $\rho$; the latter is a binary relation that is {\it outside} the rainbow signature,
but it labels coloured graphs during a `rainbow game'.
 \pe\ can win the rainbow $\omega$--rounded game
and build an $n$--homogeneous model $M$ by using $\rho$ when
she is forced a red;  \cite[Proposition 2.6, Lemma 2.7]{Hodkinson}.
From now on, forget about $\rho$; having done its task as a colour  to (weakly) represent $\A$, it will play no further role.

Having $M$ at hand, one constructs  two atomic $n$--dimensional set algebras based on $M$, sharing the same atom structure and having 
the same top element.  
The atoms of each will be the set of coloured graphs, seeing as how, quoting Hodkinson \cite{Hodkinson} such coloured graphs are `literally indivisible'. 
Now $L_n$ and $L_{\infty, \omega}^n$ are taken in the rainbow signature (without $\rho$). Continuing like in {\it op.cit}, deleting the one available red shade, set
$W = \{ \bar{a} \in {}^n M : M \models ( \bigwedge_{i < j <n} \neg \rho(x_i, x_j))(\bar{a}) \},$
and for $\phi\in L_{\infty, \omega}^n$, let
$\phi^W=\{s\in W: M\models \phi[s]\}.$
Here $W$ is the set of all $n$--ary assignments in
$^nM$, that have no edge labelled by $\rho$.
We note that $\rho$ is used by \pe\ infinitely many times during the game forming a `red clique' in $M$ \cite{Hodkinson}.

Let $\A$  be the relativized set algebra with domain
$\{\varphi^W : \varphi \,\ \textrm {a first-order} \;\ L_n-
\textrm{formula} \}$  and unit $W$, endowed with the
usual concrete quasi--polyadic operations read off the connectives.
Classical semantics for {\it $L_n$ rainbow formulas} and their
semantics by relativizing to $W$ coincide \cite[Proposition 3.13]{Hodkinson} {\it but not with respect to 
$L_{\infty,\omega}^n$ rainbow formulas}.
This depends essentially on \cite[Lemma 3.10]{Hodkinson}, which is the heart and soul of the proof in \cite{Hodkinson}, and for what matters this proof.
The referred to lemma says that any permutation $\chi$ of $\omega\cup \{\rho\}$,
$\Theta^{\chi}$ as defined in  \cite[Definitions 3.9, 3.10]{Hodkinson} is an $n$ back--and--forth system
induced by any permutation of $\omega\cup \{\rho\}$.

Hence the set algebra $\A$ is isomorphic to a quasipolyadic set algebra 
(of dimension $n$) having top element $^nM$, so $\A$
is simple, in fact its $\Df$ reduct is simple.
Let $\E=\{\phi^W: \phi\in L_{\infty, \omega}^n\}$
\cite[Definition 4.1]{Hodkinson}
with the operations defined like on $\A$ the usual way.
$\Cm\bf At$ is complete and, so like in \cite[Lemma 5.3]{Hodkinson}
we have an isomorphism from $\Cm\bf At$  to $\E$ defined
via $X\mapsto \bigcup X$.
We have $\At\A=\At\Tm(\At\A)=\bf At$ (where $\Tm(\At\A)$ denotes the subalgebra of $\Cm\At\A$ generated by the atoms; the term algebra) 
and $\Tm\At\A\subseteq \A$, hence $\Tm\At\A$ is representable.
The atoms of $\A$, $\Tm\At\A$ and $\Cm\At\A=\Cm \bf At$ are the coloured graphs whose edges are {\it not labelled} by $\rho$.
These atoms are uniquely determined by $\sf MCA$ formulas in the rainbow signature of $\bf At$  as in
\cite[Definition 4.3]{Hodkinson}.

{\bf Embedding $\CA_{n+1, n}$ into the $\Cm\bf At$ the \de\ completion of $\Tm\bf At$}: Now to embed $\A_{n+1,n}$ into $\Cm{\bf At}=\Cm\At\A$, we need some preparing to do. 
To start with, we Identify $\r$ with $\r^0$, so that we consider that $\bf At_f\subseteq \bf At$.  Let ${\sf CRG}_f$ be the class of coulored graphs on 
$\bf At_f$ and $\sf CRG$ be the class of coloured graph on $\bf At$. By the above identification, we 
can assume that  ${\sf CRG}_f\subseteq \sf CRG$.
Write $M_a$ for the atom that is the (equivalence class of the) surjection $a:n\to M$, $M\in \sf CGR$.
Here we identify $a$ with $[a]$; no harm will ensue.
We define the (equivalence) relation $\sim$ on $\bf At$ by
$M_b\sim N_a$, $(M, N\in {\sf CGR}):$
\begin{itemize}
\item $a(i)=a(j)\Longleftrightarrow b(i)=b(j),$

\item $M_a(a(i), a(j))=\r^l\iff N_b(b(i), b(j))=\r^k,  \text { for some $l,k$}\in \omega,$

\item $M_a(a(i), a(j))=N_b(b(i), b(j))$, if they are not red,

\item $M_a(a(k_0),\dots, a(k_{n-2}))=N_b(b(k_0),\ldots, b(k_{n-2}))$, whenever
defined.
\end{itemize}
We say that $M_a$ is a {\it copy of $N_b$} if $M_a\sim N_b$ (by symmetry $N_b$ is a copy of $M_a$.) 
Indeed, the relation `copy of' is an equivalence relation on $\bf At$.  An atom $M_a$ is called a {\it red atom}, if $M$ has at least one red edge. 
Any red atom has $\omega$ many copies, that are {\it cylindrically equivalent}, in the sense that, if $N_a\sim M_b$ with one (equivalently both) red,
with $a:n\to N$ and  $b:n\to M$, then we can assume that $\nodes(N) =\nodes(M)$ 
and that for all $i<n$, $a\upharpoonright n\sim\{i\}=b\upharpoonright n\sim \{i\}$.
In $\Cm\bf At$, we write $M_a$ for $\{M_a\}$ 
and we denote suprema taken in $\Cm\bf At$, possibly finite, by $\sum$.
If $N_b$ is a red copy of $M_a$, then we may denote $N_b$ by $M_a^{(j)}$ $(j\in \omega)$. 
Observe that a red atom $M_a$ has $\omega$ many copies 
forming a countable (infinite) set $\{M_a^{(j)}: j\in \omega\}$ of red graphs. If $M_a$ is a red atom, then by $\sum_j M_a^{(j)}$ we understand  the 
infinite sum of its copies evaluated in $\Cm {\bf At}$. If $M_a$ is not red, then it has only one copy, namely, itself. 
We define the map $\Theta$ from $\A_{n+1, n}=\Cm{\bf At_f}$ to $\Cm\bf At$,
by $\Theta(X)=\bigcup_{x\in \bf At_f}\Theta(x)$ $(X\subseteq \bf At_f$),  
by specifing first its values on ${\sf At_f}$,
via $M_a\mapsto \sum_jM_a^{(j)}$; each atom maps to the suprema of its 
copies.  If $M_a$ is not red,   then by $\sum_jM_a^{(j)}$,  we understand $M_a$.
This map is well-defined because $\Cm\bf At$ is complete. 
We check that $f$ is an injective homomorphim. Injectivity follows from $M_a\leq f(M_a)$, hence $f(x)\neq 0$ 
for every atom $x\in \At(\A_{n+1, n})$.
Now we check presevation of some of the operations. 
The Boolean join is obvious. We check cylindrifiers and substitutions.

(1) Cylindrifiers: Let $i<n$. By additivity of cylindrifiers, we restrict our attention to atoms 
$M_a\in \bf At_f$ with $a:n\to M$, and $M\in \sf CRG_f\subseteq \sf CRG$. Then: 
$$f({\sf c}_i^{\Cm\bf At_f}a)=f (\bigcup_{[c]\equiv_i[a]} M_c)
=\bigcup_{[c]\equiv_i [a]}f(M_c)$$
$$=\bigcup_{[c]\equiv_i [a]}\sum_j M_c^{(j)}
=\sum_j \bigcup_{[c]\equiv_i [a]}M_c^{(j)}$$
$$=\sum _j{\sf c}_i^{\Cm\bf At}M_a^{(j)}
={\sf c}_i^{\Cm\bf At}(\sum_j M_a^{(j)})
={\sf c}_i^{\Cm\bf At}f(a).$$

(2) Substitutions: Let $i, k<n$. By additivity of the ${\sf s}_{[i,k]}$s, we again restrict ourselves to atoms of the form $M_a$ as specified in the previous items.
Now computing we get:
$f({\sf s}_{[i,k]}^{\Cm\bf At_f} M_a)= f(M_{a\circ [i,k]})=  \sum_j^{\Cm\bf At}(M_{a\circ [i,k]}^{(j)})=\sum_j {\sf s}_{[i,k]}^{\Cm\bf At}{} M_a^{(j)}={\sf s}_{[i,k]}^{\Cm \bf At}{}(\sum_{j}M_a^{(j)})=
{\sf s}_{[i,k]}^{\Cm\bf At}{}f(M_a).$

We have proved that $\A_{n+1, n}$ embeds into $\Cm \bf At$, 
so that it is not blurred at the level of the last 
complex algebra.

{\bf \pa\ has  a  \ws\ in $G^{n+3}\At(\CA_{n+1, n})$:} It is straightforward to show that 
\pa\ has \ws\ first in an  \ef\ forth  private game played between \pe\ and \pa\ on the complete
irreflexive graphs $n+1$ and $n$ in 
$n+1$ rounds. This game lifts to a graph game  \cite[pp.841]{HH} on $\At_f$.
 \pa\  lifts his \ws\ from the private \ef\ forth game, to the graph game on ${\bf At}_f=\At(\A_{n+1,n})$
using the standard rainbow strategy \cite{HH}.

In his zeroth move, \pa\ plays a graph $\Gamma$ with
nodes $0, 1,\ldots, n-1$ and such that $\Gamma(i, j) = \w_0 (i < j <
n-1), \Gamma(i, n-1) = \g_i ( i = 1,\ldots, n-2), \Gamma(0, n-1) =
\g^0_0$, and $ \Gamma(0, 1,\ldots, n-2) = \y_{n+1}$. This is a $0$-cone
with base $\{0,\ldots, n-2\}$. In the following moves, \pa\
repeatedly chooses the face $(0, 1,\ldots, n-2)$ and demands a node
$\alpha$ with $\Phi(i,\alpha) = \g_i$, $(i=1,\ldots, n-2)$ and $\Phi(0, \alpha) = \g^\alpha_0$,
in the graph notation -- i.e., an $\alpha$-cone, $\alpha\leq n+2$,  on the same base.
\pe\ among other things, has to colour all the edges
connecting new nodes created by \pa\ as apexes of cones based on the face $(0,1,\ldots, n-2)$. By the rules of the game
the only permissible colours would be red. Using this, \pa\ can force a
win, using $n+3$ nodes.
By lemma \ref{Thm:n}, $\Rd_{sc}\A_{n+1, n}\notin
\bold S{\sf Nr}_n\Sc_{n+3}^{\sf ad}$. Since $\A_{n+1, n}$ is finite then $\Rd_{sc}\A_{n+1, n}$
 is not in $\bold S\Nr_n\Sc_{n+3}$.  Else $\Rd_{sc}\A_{n+1}\subseteq \mathfrak{Nr}_n\D$ 
for some $\D$ in $\Sc_{n+3}$ and we can assume that $A$ generates $\D$, so that $\D$ is finite hence completely additive which is impossible. 
Since $\A_{n+1,n}$ embeds into $\Cm\At\A$,
hence $\Rd_{sc}\Cm\At\A=\Cm\Rd_{sc}\At\A$
is outside $\bold S{\sf Nr}_n\Sc_{n+3}$, too. 
By lemma \ref{dfb},  $\Rd_{df}\Cm\At\A\notin {\sf RDf}_n$.
This proves the non--atom canonicity of ${\sf RDf}_n$, 
since plainly $\Rd_{df}\A\in {\sf RDf}_n$, 
and $\Rd_{df}\Cm\At\A=\Cm\At\Rd_{df}\A$.

The algebra $\A\in \RQEA_n\subseteq \bold S\Nr_n\QEA_m\subseteq \bold S\Nr_n{\sf G}_m$, but $\Cm\At\A\notin \bold S\Nr_n{\sf G_m}$ because, 
by lemma \ref{n}, like $\A_{n+1, n}$, $\Cm\At\A$ does not have an $m$--square representation, 
so $\bold S\Nr_n{\sf G}_m$ is not  atom--canonical. 

The ${\sf D}_m$ case is similar: Let $\Rd_{ca}$ denote `cylindric reduct'. 
Now we work 
with $\Rd_{ca}\A_{n+1, n}$ and $\Rd_{ca}\A$. We have $\Rd_{ca}\A\in \RCA_n$, 
hence has an $k$-- dilation in $\RCA_k$ for all $k>n$. Since, for any such $k$, 
$\RCA_k\subseteq {\sf D}_k$, we get that $\Rd_{ca}\A\in \bigcap_{k>n}\bold S\Nr_n{\sf D}_k$.
But $\Rd_{ca}\A_{n+1, n}$ embeds into $\Rd_{ca}\Cm\At\A$ by the same $\QEA$ embedding defined above restricted to 
cylindric reducts.
By observing that $\Rd_{ca}\Cm\At\A =\Cm\At(\Rd_{ca}\A)$, the former algebra is the \de\ completion of $\Rd_{ca}\A$, and using lemma \ref{n} again, it has no
$n+3$--square representation, hence it is outside $\bold S\Nr_n{\sf D}_m$. 

\section{Splitting in cylindric--like algebras; splitting an atom in a representable 
algebra to get a non--representable one}

In blow up and blur constructions, one splits some (possibly all) of the atoms in a finite non representable 
algebra, each into into infinitely many subatoms to get a representable (term) algeba.  In theorem \ref{ANT} all atoms were split each 
into $\omega$-- many subatoms. In theorems 
\ref{b2}, \ref{can} only the red atoms were split, each into $\omega$--many red subatoms.  In 
the rainbow construction for $\CA$s used in theorem \ref{can}, by a red atom we meant (and still do) the equivalence class of 
a surjection $a:n\to \Delta$ where $\Delta$ is a coloured graph in the rainbow signature of $\A_{n+1, n}$ 
(as 
specified in {\it op.cit}) having at least one edge labelled by a red.

Now we do the exact
opposite. We split an atom in a representable algebra to finitely 
many subatoms getting a non--representable algebra whose `small subalgebras' are representable.

For a start fix the dimension to be $\omega$. We outline the idea of splitting atoms for $\CA_{\omega}$.
This method typically shows, 
that for any postive positive $k$, there is an a non--representable algebra having $\CA_{\omega}$ signature, 
such that all its $k$--generated subalgebras are representable, where by a $k$--generated subalgebra we mean 
a subalgebra generated by at most $k$ elements. From such constructions  
relative non--finite axiomatizability concerning number of variables used 
in universal axiomatizations 
can be easily proved.

Fix such $k$. One chooses a certian finite $m\geq k$ suitably so that the 
`combintorial  part' of the proof, to be further elaborated upon below, 
works.
\begin{itemize} 

\item Start with a system $(U_i: i\in \omega)$  of sets wherethe $U_i$s are pariwise disjoint, 
$|U_0|=m-1$ and $|U_i|=m$ for $i\in \omega\sim \{0\}$.

\item  Take the full set algebra with top element $U=\bigcup_{i\in I}U_i$, call it (identifying it with its universe) $\wp(^{\omega}U)$ 
and then take the subalgebra $\A=\Sg^{\wp(^{\omega}U)}(\{R\})$ where $R$ is the $\omega$--ary relation $\prod_{i\in \omega}U_i$.
Then it can be easily checked that $R$ is an atom is in $\A$.

\item Split the atom $R$ into distinct subatoms $(R_i: i<m)$  forming a bigger algebra $\A_k$ having $\CA_{\omega}$ signature in 
which the $R_i$s, for  $i<m$ are (distinct atoms), 
$\sum R_i=R$ and each $R_i$ $(i<m)$ is {\it cylindrically equivalent} to $R$, in the sense 
that ${\sf c}_jR_i={\sf c}_jR$ for all $j<\omega$. 

\item Then $\A_k$ cannot be represenatble, because any represenation 
of $\A$ forces $|U_0|\geq m$ which is impossible. 
However, the $k$--generated subalgebras of $\A_k$ will all be in $\RCA_{\omega}$ 
as illustrated below.

\end{itemize}

We adress now the finite dimensional case for $\QEA$s.
We use splitting methods in \cite{Andreka}.
\begin{theorem}\label{ss} Let $2<n<\omega$. Then $\RQEA_n$ is not axiomatizable by a set of universal formulas containing finitely many variables.
\end{theorem}
\begin{demo}{\bf Idea} \cite{c} Fix a positive $k$ and finite $n>2$. 
For  $\QEA_n$s, in the presence of only finitely many substitution operations, one takes a fairly simple representable algebra, as it happens a set algebra
$\A$ generated by two $\omega$--ary relations $R$ and $F$, where $R$ is an atom. The top element  of $\A$, is like described above, 
it is of the form $^{\omega}U$ such that $U$ is a 
disjoint union of a family $(U_i: i<n)$ of pairwise disjoint finite non--empty sets, $R=\prod_{i<n}U_i=U_0\times \prod_{i\in n\sim \{0\}}U_i$
and $F=\{s\in {}^nU: s_0, s_1\in U_0 \text{ and } s_1=f(s_0)\}$. Here $f:U_0\to U_0$ is a bijection such that the orbits of $f$ has cardinality 
$K<\omega$,  where  $2^{k\times n!}+1\leq K(n-1)$. (The orbits of a function are the equivalence classes resulting from the equivalence relation 
defined on its domain by $x\sim y\iff f(x)=f(y)$.) 
 
Each permuted version of $R$ obtained by applying a substitution operation (corresponding to a transposition $[i, j]$, $i<j<n$) 
denoted by $\s_{[i, j]}$ 
to $R$, renders another atom in $\A$, namely $\s_{[i,j]}R$,  and these permutated versions of $R$ are pairwise disjoint, that is 
for $i,j,k,l\in n $, $\{i, j\}\neq \{k, l\}\implies \s_{[i,j]}R\cap \s_{[k,l]}R=\emptyset$.
Each of these $n^2-n$ atoms is {\it split} into $m$--many (abstract) subatoms or copies, where $m=K(n-1)$ forming a bigger algebra $\A_{k, n}$ having the same
signature as $\QEA_n$. In particular, $R$ is partitioned  into 
a family $(R_i: i\leq m-1)$ of atoms in the bigger algebra $\A_{k, n}(\supseteq \A$), 
so that $R=\sum _{i\leq m-1}^{\A_{k,n}}R_i$, where $m\geq |U_0|+1$. Furthermore in $\A_{k, n}$ we have  
$\s_{[l,j]}R=\sum _{i\leq m-1}^{\A_{k,n}}\s_{[l,j]}R_i$ for all $l<j<n$. 

Though $\A_{k, n}$ has the same signatures as $\QEA_n$, 
it {\it might not} be a $\QEA_n$, but it contains $\A$ as a subalgebra, and its
Boolean reduct is a Boolean algebra. Such subatoms are also {\it cylindrically equivalent} to {\it their original} 
in $\A_{k, n}$, meaning that  ${\sf c}_i^{\A_{k, n}}\s_{\tau}R_j={\sf c}_i^{\A_{k,n}}\s_{\tau}R$ for all $i<n$, 
$j\leq m-1$, and transposition $\tau:n\to n$. Here $\s_{\tau}R$ is the original of its $m$ copies $(\s_{\tau}R_i :i\leq m-1)$. 
This is  typical theme in splitting arguments, saying roughly that cylindrifiers cannot distinguish 
between the original atom and its splitted copies.

The number of elements in $|U_0|$, hence $m(\geq |U_0|+1)$, depends on the fixed in advance $k$ and the number $n^2-n$ of 
the substitution operations ${\sf s}_{[i, j]}$, $i<j<n$ in the signature of $\QEA_n$.
The cylindric reduct of $\A_{k, n}$ will not be representable, because any representation of $\A_{k, n}$ 
forces $|U_0|\geq m$ \cite[pp.178--179]{Andreka}. 
Here  the generator
$F$  plays the essential role. 

But the splitting does not ruin the representability of the ``small" subalgebras  of the thereby obtained non--representable algebra, as long as $m$ is not `too large'.
By a small subalgebra, we mean a {\it subalgebra generated by 
at most $k$ many elements}; here the inequality $2^{k\times n!}+1\leq K(n-1)$ is used.  This excludes universal axiomatizations using $k$ many variables.
Doing this for every positive $k$, excludes universal axiomatizations using finitely many variables.
For a detailed proof the reader is 
referred to \cite{c}.
\end{demo}

Andr\'eka's splitting argument, just outlined, is in fact an ingenious combination of 
Monk--like  constructions with the concept of {\it dilations} in the sense of \cite[Construction 3.6.69]{HMT2}.
Andr\'eka's construction avoids `colouring' (expressed in Monk's construction by an application of Ramsey's theorem.) 
Roughly, in dilations one adds atoms to an atomic algebra, if it is not down right `impossible' to do so, witness
\cite[Last paragraph p. 88]{HMT2}.  
The finite number $|U_0|$ plays the role of the number of colours used in Monk's original construction of bad Monk algebras. 
Here too, in `Andr\'eka's splitting argument' the incompatibility condition between the number of atoms 
$m$ and the number of colours $|U_0|$ $(m>|U_0|)$, leads to an impossibility in case there is 
a representation of $\A_{k, n}$; for the existence of a representation {\it concretely represents the $m$ atoms below $R$} 
forcing 
$|U_0|\geq m$.This method in the finite dimensional case depends on counting the number 
of substitution operations in the signature, so it does 
not work in the infinite dimensional case.  
In the proof for the infinite dimensional one conquers this difficulty by spitting atoms in infinitely many algebras 
obtaining a chain of subreducts of $\QEA_{\omega}$, such that the signature of each
is an expansion of $\CA_{\omega}$ with {\it finitely many} substitution operators, 
and their directed union has the signature
of $\QEA_{\omega}$. 
This answers a question of Andr\'eka's  posed explicitly in \cite[p.193]{Andreka}; we quote:

{\it ``For $n\geq \omega$, the analogous algebras are called representable quasi-polyadic equality algebras, and their class 
is denoted by $RQPEA_n$. We do not know whether Theorem 6 remains true if we drop the condition $n<\omega$ 
and replace $RPEA_n$ 
by $RPQEA_n$ in it.''}

Here `Theorem 6' that Andr\'eka refers to, formulated on \cite[p.193]{Andreka}, 
is the finite version of theorem \ref{s} formulated and proved next. But first a lemma.
The lemma tells us when can we partition a relation $R$ into $m$ `concrete copies'. Such an $R$ will be an atom in a set algebra 
that will be split into $>m$ abstract copies forming a non--representable algebra.
\begin{lemma}\label{split} Let $m$ be a finite cardinal and $\alpha$ be an ordinal $\geq \omega$. Let $(U_i: i\in \alpha)$ be  system of sets each having cardinality $\geq m$, 
and let $U\supseteq  \bigcup \{U_i: i\in \alpha)$. Then there is a partition $(R_j: j<m)$ of $R=\prod_{i\in I}U_i$ such that ${\sf c}_i^UR_j={\sf c}_i^UR$ 
for all $i<\alpha$ and $j<m$.
\end{lemma}
\begin{proof} \cite[Lemma 3]{Andreka}
\end{proof}

We let $\Rd_{ca}$ denote `cylindric reduct'.

\begin{theorem}\label{s}  The variety $\RQEA_{\omega}$ cannot be axiomatized by a set of universal formulas containing finitely many variables
over $\sf RQA_{\omega}$.
\end{theorem}
\begin{proof} We give a fairly complete sketch omitting some minor details. 
The proof is divided into five parts. 
We may identify notationally set algebras with their universe. 

Throughout the proof fix a positive $k$. 

{\bf (1) Splitting an atom in a set algebra getting a non--representable algebra:} 
This part is similar to the finite dimensional case addressed in 
the proof of the previous theorem \ref{ss}. The difference is that here the set algebra we start off with is simpler; it is generated by a single $\omega$--ary relation
and now we have all of the $\CA_{\omega}$ operations. 
But the {\it substitution operators, like the finite dimensional case are finite}, so that in our combinatorial
arguments used, we can count them.
 
For fixed $2<n<\omega$, take $m\geq 2^{k\times n!+1}$. Suppose that the signature consists of 
$\omega$--many cylindrifier, ${\sf c}_i: i<\omega$,  diagonal constants ${\sf d}_{ij}$,$i<j<\omega$, and $n^2-n$ substitutions $\{\s_{[i, j]} :i<j<n\}$.
Like forming $\A_{k, n}$ in the finite dimensional case, one forms an algebra $\B_{k, n}$ by splitting the $\omega$--ary relation 
$R=\prod_{i\in \omega}U_i$ with $U_0=m-1$ and $|U_i|=m$ for $0<i<\omega$ in the algebra 
$\A_n=\Sg^{\wp(^{\omega}U)}\{R\}$, where $U=\bigcup_{i\in I} U_i$,  into $m$ abstract copies.  Observe that here $R$ depends on $n$, because $m$ depends on $n$ and 
$R$ depends on $U_0=m-1$. Nevertheless, we prefer to use the notation $R$ rather than the more cumbersome notation $R_n$ or $R^n$. 

The resulting $algebra \B_{k, n}$ has signature expanding $\CA_{\omega}$ by the finitely many substitution operators 
$\s_{[i,j]}$, $i<j<n$.  Here ${\wp(^{\omega}U)}$ is taken in the specified signature
with operations interpreted the usual way as in set algebras, e.g. ${\sf S}_{[0, 1]}\{R\}=\{s\in {}^{\omega}U: (s_1, s_0)\in R\}$.
As in the finite dimensional case addressed in the proof of theorem \ref{ss}, it can be easily checked that for all $i<j<n$, ${\sf S}_{[i,j]}R$ is an atom in $\A_n$. In particular, $R$ is partitioned  into 
a family $(R_i: i\leq m-1)$ of atoms in the bigger algebra $\B_{k, n}(\supseteq \A_n$), 
so that $R=\sum _{i< m}^{\B_{k,n}}R_i$, where $m=|U_0|+1$. Furthermore in $\B_{k, n}$ we have  
$\s_{[l,j]}R=\sum _{i<m}^{\B_{k,n}}\s_{[l,j]}R_i$ and 
 ${\sf c}_t^{\A_{k, n}}\s_{[l, j]}R_i={\sf c}_t^{\A_{k,n}}\s_{[l, j]}R$ for all $l,j<n$, $i<m$ 
and $t<\omega$. 
 
Here {\it we do not need the function $F$ among the generators} used in the finite dimensional case, so in a way this part of the
proof is 
simpler.

Then $\Rd_{ca}\B_{k, n}$ will not be representable for the following reasoning: 
One defines the term $\tau(x)= (\bigwedge_{i<m}{\sf s}_i^0{\sf c}_1\ldots {\sf c}_mx\cdot \bigwedge_{i<j<n}-{\sf d}_{ij})$ as in \cite[Top of p.157]{Andreka}.
Then $\A_n\models \tau(R)=0$ hence $\B_{k, n}\models \tau(R)=0$ because $\A\subseteq \B_{k, n}$.
But any representation of $\B_{k, n}$, $h$ say, such that $h(R_n)\neq \emptyset$ forces that $\tau(h(R))\neq 0$ which is impossible.
The underlying idea here is that in case of the existence of such an $h$, then 
the $m$ atoms below $R$ {\it in the presence of diagonal elements} 
used in defining the term $\tau$, forces $U_0$ 
to have $m$ elements. 

This technique does not work in the finite dimensional case, where we had {\it only finitely many diagonal elements}.
This  was the reason why we added the function $F$ to the set of generator of $\A_n$ as 
specifed in the proof of theorem \ref{ss}; to implement `the counting' 
leading to an impossibility.

{\bf (2) Representability of the $k$--generated subalgebras:} Now we show that the $k$ generated subalgebras are representable. 
Let  $G\subseteq \B_{k, n}$, $|G|\leq k$. Let $\mathfrak{P}=(R_{l}:  l<m)$ be the abstract partition of $R$ in the bigger algebra $\B_{k, n}$ obtained by splitting $R$ in $\A_n$ 
into $m$ (abstract) atoms
$(R_{l}: l<m)$.  
One defines the following relation on $\mathfrak{P}$: 
For $l, t<m$, $R_{l}\sim R_{t}\iff (\forall g\in G)(\forall i, j<n)({\sf s}_{[i, j]}R_{l}\leq g\iff {\sf s}_{[i,j]}R_{t}\leq g)$.
Then it is straightforward to check that $\sim$ is an equivalence relation on $\mathfrak{P}$
having $p<m$ many equivalence classes, because $|G|\leq k$, 
$n^2-n<n!$ and (recall that) 
$m\geq 2^{k\times n!+1}$. 

One next takes $B=\{a\in B_{k,n}: (\forall l,t< m)(\forall i, j\in n)(R_{l}\sim R_{t},
{\sf s}_{[i, j]}R_{l}\leq a\implies {\sf s}_{[i, j]}R_{t}\leq a)\}$, then $G\subseteq B$,  $R\in B$, and $B$ is closed under the operations, so that $\A_n\subseteq \B\subseteq \B_{k, n}$, where
$\B$ is the algebra with universe $B$.
Furthermore, $\B$ is the smallest such subalgebra of $\B_{k, n}$, where for each $i, j<n$, ${\sf s}_{[i, j]}R$ is partitioned into $p<m$ many parts cylindrically equivalent to 
${\sf s}_{[i,j]}R$.  The non--representability of the algebra $\B_{k, n}$ can be pinned down to the existence of `one more extra atom'. 
(Here we use the notation $\s_{[i, j]}$ {\it and  not ${\sf S}_{[i, j]}$ $(i<j<n)$}, because we still do not know that $\B$ is representable).

Identifying set algebras with their domain, for an algebra $\A$ and a non--zero $a\in \A$, we say that 
a representation $h:\A\to \wp({^{\omega}U})$ {\it respects the  non--zero element $a$}
if $h(a)\neq \emptyset$. 
Using that $|\mathfrak{P}|=m$, we showed that a representation of $h$ of $\B_{k, n}$ that respects $R$, has to respect the atoms below it,
and this forces that $|U_0|\geq m$, which contradicts the construction of $\A_n$. As indicated above, the presence of diagonal elements is essential to do the counting of the elements
in $U_0$. 

But this cannot happen with $\B$, because $p<m$ (by the condition $|G|\leq k)$), so that this `one more extra atom and possibly more' vanish in $\B$. 
Representing $\B$ is done by embedding it into a representable algebra $\C$
having the same top element as $\A_n$, namely, $^{\omega}U$, where $R\in \C$ is partitioned  {\it concretely into $m-1$ real atoms}, 
that is, there exists $R_l\subseteq {}^{\omega}U$, $l<m-1$ {\it real atoms} in $\C$ such that 
for all $i<j<n$, ${\sf S}_{[i, j]}R= {\sf S}_{[i, j]}\bigcup_{l<m-1}R_l= \bigcup_{l<m-1} {\sf S}_{[i, j]}R_l$ and ${\sf C}_iR_l={\sf C}_iR$ 
for all $l<m-1$ and $i<\omega$. This concrete partition exists by lemma \ref{split}  because $|U_0|=m-1$
and by the condition $|G|\leq k$, the value of $p$ (depending on $G$) cannot exceed $m-1$. The largest possible value of $p$ is $m-1$,
and if $p=m-1$,  then $\B\cong \C$.

{\bf (3) Forming a $\QEA_{\omega}$ as a limit; taking the directed union:} One does this for each $2<n<\omega$.
 The constructed non--representable algebras 
form a chain; for $2<n_1<n_2$, $\B_{k, n_1}\subseteq \Rd\B_{k, n_2}$, where the last algebra is the reduct obtained from 
$\B_{k, n_2}$ by discarding substitution operations not in the signature of $\B_{k, n_1}$, 
that is the substitution operations ${\sf s}_{[i,j]}: i, j\geq n_1, i\neq j$. 
Then one takes the 
directed union  $\B_k=\bigcup_{n\in \omega\sim 3} \B_{k, n}$ having the signature of $\QEA_{\omega}$.
The cylindric reduct of $\B_k$ is not representable because the cylindric reduct of every $\B_{k, n}$ is not representable.
But the $k$--generated subalgebras of $\B_k$ are representable.
Indeed, let $|G|\leq k$. Then $G\subseteq \B_{k,n}$ for some finite $n>2$. If $\Sg^{\A_{k}}G$ is not representable then there exists 
$l\geq n$ such that $G\subseteq \B_{k,l}$ and 
$\Sg^{\B_{k,l}}G$ is not representable, which is a contradiction.

One can construct such algebras $\B_k$ having the signature of $\QEA_{\omega}$ for each positive $k$. 
This excludes existence of any universal axiomatization of $\RQEA_{\omega}$ using
finitely many 
variables \cite{Andreka, ST}. 

Worthy of note is that in \cite{ST} the proof of the last result is sketched, but in this sketch only one splitting is done. 
An atom in an $\sf RQEA_{\omega}$ is split into $\omega$ many atoms. It is clear that the resulting algebra is not representable, 
but it is not clear why its $k$--generated subalgebras
should be representable. 
So here instead of doing just one splitting, we perform `infinitely many' as illustrated above.

{\bf (4) The representability of the diagonal free reduct:}
But one can even go further, by showing that the diagonal free reduct is in 
${\sf RQA}_{\omega}$ giving the required relative non--finite axiomatizability result by reasoning as follows. 
Refinements of this idea 
gives the stronger results of relative non--finite axiomatizabiliy. For fixed $k$ and $n\in \omega\sim 3$, 
when we discard the diagonal elements, the diagonal free reducts of the non--representable algebras 
$\B_{k, n}$ described above {\it are representable}. We give the general idea first. The $\B_{k, n}$ is not representable because $m=|U_0|+1$. So now one adds 
`one extra element or more' to 
$|U_0|$ forming $W_0$ to be able to represent the abstract partition of $R$ including the `one extra atom' that was responsible for non--representabitly of $\A_{k, n}$ when 
$U_0$ did not have `enough elements'; recall that $U_0$ had $m-1$ elements. 
The diagonal free reduct of $\A_{k, n}$ can now be represented by a set algebra obtained by splitting an $\omega$--ary 
relation $R=W_0\times \prod_{i\in \omega} U_i$
where $|W_0|\geq m$ and $|U_i|=m+1$, $i\in \omega$, in a set algebra generated by $R$,  
into $m$ real atoms,   Here, in the {\it absence of diagonal elements},  we cannot count the elements in $|W_0|$, 
so adding this element to $U_0$ does not clash with the concrete interpretation of the 
other operations. 

Now let us give the details. For the sake of brevity let $G_n$ denote the set of transpositions on $n$. 
Let $U=\bigcup_{i\in \omega}U_i$, where $(U_i: i\in \omega)$ is a family of pairwise disjoint sets such that 
$U_0=m-1$, $m-1<|U_i|<\omega$, $R=\{s\in {}^{\omega}U: s_i\in U_i\}$ and $\B_{k, n}$ be the algebra obtained by
splitting $R$ in $\Sg^{\wp(^{\omega}U)}\{R\}$ into $m$ atoms $(R_j: j< m)$ be as before.
Let $W$ be a proper superset of $U$, that is $U\subsetneq W$.  
Let $W_0=U_0\cup (W\sim U)$, and $W_i=U_i$ for $0<i<\omega$. 
Let $t: W\to U$ be a surjective function which is the identity on $U$ and which maps $W_0$ to $U_0$. Define $g:{}^{\omega}W\to {}^{\omega}U$ 
by $g(s)=t\circ s$ for all $s\in {}^{\omega}W$ and for  $x\subseteq {}^{\omega}U,$  define
$h(x)=\{s\in {}^{\omega}U: g(s)\in x\}.$
Then $h: \wp(^{\omega}U)\to \wp(^{\omega}W)$ respects all quasi--polyadic equality operations except possibly diagonal elements.
Furthermore, $h(R)=\prod_{i<\omega}W_i$.

Since $|W_i|\geq m$ for all $i<\omega$,  then by lemma \ref{split} there is a partition $(S_j: j<m)$ 
of $S=\prod_{i<\omega}W_i$ into $m$ parts, such that ${\sf c}_iS_j={\sf c}_iS$ for all $ i<\omega$ and $j< m$.
Then $({\sf s}_{\sigma}S_j: j< m)$ is an analogous partition of ${\sf s}_{\sigma}S$ for $\sigma\in G_n$.
For $\sigma\in G_n$ and $j<m$,  let $X_{\sigma,j}={\sf s}_{\sigma}^{\B_{k,n}}R_j$.
To define the required representation $\bar{h}:\B_{k,n}\to \wp(^{\omega}W)$, it suffices to define it on $\A$ and $X_{\sigma, j}$ for each $\sigma\in G_n$ and $j<m$, 
since every element of $\B_{k, n}$ is a finite union of such elements: 
$$\bar{h}(a)=a, \ \ a\in \A_n,$$
$$\bar{h}(X_{\sigma j})={\sf s}_{\sigma}S_j, \sigma \in G_n, j< m,$$
and
$$\bar{h}(x+y)=\bar{h}(x)+\bar{h}(y), x,y\in \B_{k,n}.$$
Then $\bar{h}$ is the required represenation \cite[p.194]{Andreka}.\\

{\bf (5) Relative non--finite axiomatizability; the required result:}
Now we show that any universal axiomatization of $\RQEA_{\omega}$, 
must contain a formula with more than $k$ variables and containing at least one diagonal constant getting the required relative non--finite axiomatizability result. 
Fix $n\in \omega\sim 3$. Let $\Sigma_n^v$ be the set of universal formulas using only $n$ substitutions and $k$ variables valid in $\RQEA_{\omega}$, 
and let $\Sigma_n^d$ be the set of universal formulas
using only $n$ substitutions and no diagonal elements valid in $\RQEA_{\omega}$.  By $n$ substitutions we understand the set 
$\{{\sf s}_{[i,j]}: i,j\in n\}.$
Then $\B_{k,n}\models \Sigma_n^v\cup \Sigma_n^d$. $\B_{k,n}\models \Sigma_n^v$ because the $k$ generated subalgebras
of $\B_{k,n}$ are representable, while $\B_{k,n}\models \Sigma_n^d$ because $\A_{k,n}$ has a representation that preserves all operations except
for diagonal elements.  Indeed, let $\phi\in \Sigma_n^d$, then there is a representation of $\B_{k,n}$ in which all operations 
are the natural ones except for the diagonal elements. 
This means that (after discarding the diagonal elements) there is a injective homomorphism 
$h:\A^d\to \P^d$ where $\A^d=\langle B_{k,n}, +, \cdot , {\sf c}_k, {\sf s} _i^j, {\sf s}_{[i,j]}\rangle_{k\in \omega, i,j\in n}\text { and } 
\P^d=\langle {\cal B}({}^{\omega}W), {\sf C}_k^W, {\sf S}_i^j{}^W, {\sf S}_{[i,j]}^W \rangle_{k\in \omega, i,j\in n},$ 
for some infinite set $W$. 

Now let $\P=\langle{\cal B}(^{\omega}W), {\sf C}_k^W, {\sf S}_i ^j{}^W, {\sf S}_{[i,j]}^W, {\sf D}_{kl}^W\rangle_{k,l\in \omega, i,j\in n}.$
Then we have that $\P\models \phi$ because $\phi$ is valid 
and so  $\P^d\models \phi$ due to the fact that  no diagonal elements  occur in $\phi$. 
It thus follows that $\A^d\models \phi$, because $\A^d$ is isomorphic to a subalgebra of $\P^d$ and $\phi$ is quantifier free. Therefore 
$\B_{k,n}\models \phi$.

Let $\Sigma^v=\bigcup_{n\in \omega\sim 3}\Sigma_n^v 
\text { and }\Sigma^d=\bigcup_{n\in \omega\sim 3}\Sigma_n^d.$ 
It follows that  $\B_k\models \Sigma^v\cup \Sigma^d.$ For if not, then there exists a quantifier free  formula 
$\phi(x_1,\ldots, x_m)\in \Sigma^v\cup \Sigma^d$,
and $b_1,\ldots, b_m$ such that $\phi[b_1,\ldots, b_n]$ does not hold in $\A_k$. We have $b_1,\ldots, b_m\in \B_{k,i}$ for some $2<i<\omega$. 
Take $n$ large enough $\geq i$ so that
$\phi\in \Sigma_n^v\cup \Sigma_n^d$.   
Then $\phi$ does not hold in $\B_{k,n}$, which is 
a contradiction.
Let $\Sigma$ be  a set of quantifier free formulas axiomatizing  $\RQEA_{\omega}$, then $\A_k$ does not model $\Sigma$ since $\A_k$ is not 
representable, so there exists a formula $\phi\in \Sigma$ such that
$\phi\notin \Sigma^v\cup \Sigma^d.$ 
Then $\phi$ contains more than $k$ variables and a diagonal constant occurs in $\phi$.
\end{proof}

For fixed positive $k$, various other reducts of $\B_{k, n}$ 
can be proved to be representable, for each $2<n<\omega$. 
For example, there are representations that preserve all operations except for infinitely many cylindrifiers, representations that preserve all
operations except the Boolean $\cap$ and $\cup$, and representations that preserve all operations except for finitely 
many substitution operations and diagonal elements having a common pre-assigned index in $\omega$.   
Using several variations on this theme, 
it can be shown, without much difficulty,  that all complexity results in \cite{Andreka} generalize to $\RQEA_{\omega}$ by combining the
method of `infinite splitting' given above  together with the techniques used 
in  \cite{c}. Using transfinite induction such results lift to an arbitrary infinite ordinal $\alpha$, one of which is the following result. 
If $\Sigma$ is any universal axiomatization of $\RQEA_{\alpha}$, $l,k, k' <\alpha$, then $\Sigma$ 
contains infinitely formulas in which one of the Boolean join or Boolean meet 
occurs,  a diagonal or a substitution operation with an index $l$ occurs, more
than $k'$ cylindrifications, and more  than $k$ variables occur. Even stronger results can 
be  obtained if $\Sigma$ is an equational 
axiomatization.

\subsection{Splitting  the atoms in a finite algebras twice 
getting atomic algebras $\A$ and $\B$ such that $\A\equiv \B$, $\A\in \bold K$, $\B\notin \bold K$, for some class $\bold K$.}

Fix $1<n<\omega$ and $\K$ any class between $\Sc$ and $\QEA$. 
Here we use the same tecnique to show that the classes $\Ra\CA_m$ for $5\leq m$ 
and $\Nr_n\K_k$, for $k\geq 1$ are not elementary. 
The idea for both cylindric and relation algebras is the same given briefly in the tiltle taking $\bold K=\Nr_n\K_n$ for cylindric--like algebras and
$\bold K=\Ra\CA_5$ for relation algebras. 

\begin{itemize}

\item Start with a finite algebra $\F$. Split the atoms twice geting two distinct algebras $\A$ and $\B$. 
The algebra $\A$ 
is obtained from $\F$ by splitting each of its atoms to infinitely many, getting finitely many {\it infinite atomic `components'} 
each having the same uncountable cardinality. 
This symmetry makes $\A$ in $\Nr_n\CA_{\omega}$ in the $\CA$ case, and makes it in $\Ra\CA_{\omega}$
in the relation algebra case.

\item On the other hand, the algebra $\B$ is obtained from $\F$ by splitting each of its atoms  (also) to infinitely many, but now getting finitely many infinite atomic components; 
for which the  cardinalities of the atoms in different components {\it are not the same.}
This assymmetry makes $\B$ outside $\bold K_{ca}=\Nr_n\CA_k$ (some finite $k$) in the $\CA$ case, and outside $\bold K_{ra}=\Ra\CA_m$ (some finite $m$)
in the $\RA$ case. 

\item But $L_{\infty, \omega}$, {\it a fortiori} $L_{\omega, \omega}$ 
will miss out on this infinite cardinality twist, so that $\A$ and $\B$ will 
be elementary equivalent. 

\item It readily follows that $\bold K_{ca}$ and $\bold K_{ra}$ are not elementary.

\end{itemize}
 
{\bf Cylindric algebra case:} We modify the construction in \cite[Lemma 5.4.1, Theorem 5.4.2]{Sayedneat} that addresses only $\CA$s to obtain the required results.
First we make the algebras $\A$ and $\B$ dealt with in the above cited theorem atomic (they are not atomic as they are), 
we count in dimension two, and we give a unified proof that works for any $\K$ between $\Sc$ and $\QEA$. 
From our new construction, we get the stronger result that the modified $\A$ and $\B$ are not only elementary equivalent;
below we show that $\A\equiv_{\infty, \omega}\B$. The proof in \cite[Theorem 5.4.2]{Sayedneat}  does not give this stronger result
which will be achieved by playing an \ef\ game on atom structures.

Let $L$ be a signature consisting of the unary relation
symbols $P_0,P_1,\ldots, P_{n-1}$ and
uncountably many $n$--ary predicate symbols. $\Mo$ is as in \cite[Lemma 5.1.3]{Sayedneat}, but the tenary relations are replaced by 
$n$--ary ones, and we require that the interpretations of the $n$--ary relations in $\Mo$ 
are {\it pairwise disjoint} not only distinct. This can be fixed. 
In addition to pairwise
disjointness of $n$--ary relations, we require their symmetry, 
that is, permuting the variables does not change
their semantics.

For $u\in {}^n n$, let $\chi_u$
be the formula $\bigwedge_{u\in {}^n n}  P_{u_i}(x_i)$. We assume that the $n$--ary relation symbols are indexed by (an uncountable set) $I$
and that  there is a binary operation $+$ on $I$, such that $(I, +)$ is an abelian group, and for distinct $i\neq j\in I$,
we have $R_i\circ R_j=R_{i+j}$. 
For $n\leq k\leq \omega$, let $\A_k=\{\phi^{\Mo}: \phi\in L_k\}(\subseteq \wp(^k\Mo))$,
where $\phi$ is taken in the signature $L$, and $\phi^{\Mo}=\{s\in {}^k\Mo: \Mo\models \phi[s]\}$. 

Let $\A=\A_n$, then $\A\in \sf Pes_n$ by the added symmetry condition.
Also $\A\cong \Nr_n\A_{\omega}$; the isomorphism is given by
$\phi^{\Mo}\mapsto \phi^{\Mo}$. The map is obviously an injective homomorphism; it is surjective, because $\Mo$ 
(as stipulated in \cite[ item (1) of lemma 5.1.3]{Sayedneat}), 
has quantifier elimination.

For $u\in {}^nn$, let $\A_u=\{x\in \A: x\leq \chi_u^{\Mo}\}.$ Then
$\A_u$ is an uncountable and atomic Boolean algebra (atomicity follows from the new disjointness condition)
and $\A_u\cong {\sf Cof}(|I|)$, the finite--cofinite Boolean algebra on $|I|$.
Define a map $f: \Bl\A\to \bold P_{u\in {}^nn}\A_u$, by
$f(a)=\langle a\cdot \chi_u\rangle_{u\in{}^nn+1}.$

Let $\P$ denote the
structure for the signature of Boolean algebras expanded
by constant symbols $1_u$, $u\in {}^nn$, ${\sf d}_{ij}$, and unary relation symbols
${\sf s}_{[i,j]}$ for each $i,j\in n$.
Then for each $i<j<n$, there are quantifier free formulas
$\eta_i(x,y)$ and $\eta_{ij}(x,y)$ such that
$\P\models \eta_i(f(a), b)\iff b=f({\sf c}_i^{\A}a),$
and $\P\models \eta_{ij}(f(a), b)\iff b=f({\sf s}_{[i,j]}a).$

The one corresponding to cylindrifiers is exactly like the $\CA$ case \cite[pp.113-114]{Sayedneat}.
For substitutions corresponding to
transpositions, it is simply $y={\sf s}_{[i,j]}x.$  The diagonal elements and the Boolean operations are easy to interpret. 
Hence, $\P$ is interpretable in $\A$, and the interpretation is one dimensional and
quantifier free. For $v\in {}^nn$, by the Tarski--Sk\"olem downward theorem, 
let  $\B_v$ be a countable elementary subalgebra of $\A_v$. (Here we are using the countable signature of $\PEA_n$).
Let $S_n (\subseteq {}^nn)$ be the set of permuations in $^nn$.

Take $u_1=(0, 1, 0, \ldots, 0)$ and $u_2=(1, 0, 0, \ldots, 0)\in {}^nn$. 
Let  $v=\tau(u_1,u_2)$  where $\tau(x,y)={\sf c}_1({\sf c}_0x\cdot {\sf s}_1^0{\sf c}_1y)\cdot {\sf c}_1x\cdot {\sf c}_0y$. We call $\tau$ an approximate
witness. It is not hard to show that $\tau(u_1, u_2)$ is actually the composition of $u_1$ and $u_2$, 
so that $\tau(u_1, u_2)$ is the constant zero map; which we denote by $\bold 0$; it is also in $^nn$. 
Clearly for every $i<j<n$,  ${\sf s}_{[i,j]} {}^{^nn}\{\bold 0\}=\bold 0\notin \{u_1, u_2\}$.

We can assume without loss that 
the Boolean reduct of $\A$ is the following product:
$$\A_{u_1}\times \A_{u_2}\times  \A_{\bold 0}\times\bold P_{u\in V\sim J} \A_u,$$
where $J=\{u_1, u_2, \bold 0\}$.
Let $$\B=((\A_{u_1}\times \A_{u_2}\times \B_{\bold 0}\times\bold P_{u\in V\sim J} \A_u), 1_u, {\sf d}_{ij}, {\sf s}_{[i,j]}x)_{i,j<n},$$
recall that  $\B_{\bold 0}\prec \A_{\bold 0}$ and $|\B_{\bold 0}|=\omega$,
inheriting the same interpretation.  Then by the Feferman--Vaught theorem,
we get that
$\B\equiv \A$.
Now assume for contradiction, that $\Rd_{sc}\B=\Nr_n\D,$ with $\D\in \Sc_{n+1}$.
Let $\tau_n(x,y)$, which we call an {\it $n$--witness}, be defined by ${\sf c}_n({\sf s}_n^1{\sf c}_nx\cdot {\sf s}_n^0{\sf c}_ny).$
By a straightforward, but possibly tedious computation, one can obtain  $\Sc_{n+1}\models \tau_n(x,y)\leq \tau(x,y)$ 
so that the approximate witness {\it dominates} the $n$--witness.

The term $\tau(x,y)$ does not use any spare dimensions, and it `approximates' the term $\tau_n(x,y)$ that
uses the spare dimension $n$. The algebra $\A$ can be viewed as splitting the atoms of the atom structure $(^nn, \equiv_, \equiv_{ij}, D_{ij})_{i,j<n}$ each to uncountably many atoms.
On the other hand, $\B$ can be viewed as splitting the same atom structure, each  atom -- except for one atom that is split into countably many atoms -- 
is also split into uncountably many atoms (the same as in $\A)$. 

On the `global' level, namely, in the complex algebra of the finite
(splitted) atom structure $^nn$, these two terms are equal, the approximate witness
is the $n$--witness. The complex algebra $\Cm{}(^{n}n)$ does not `see' the $n$th dimension.
But in the algebras $\A$ and $\B$, obtained after splitting,  the $n$--witness becomes then a {\it genuine witness},  not an approximate one.  The approximate witness 
{\it strictly dominates} the $n$--witness.  The $n$--witness using the spare dimension $n$, detects the cardinality twist that $L_{\infty, \omega}$, {\it a priori}, 
first order logic misses out on.
If the $n$--witness were term definable (in the case we have a full neat reduct of an algebra in only one extra dimension), then
it takes two uncountable component to an uncountable one,
and this is not possible for $\B$, because in $\B$, the target 
component  is forced to be 
countable. 

Now for $x\in \B_{u_1}$ and  $y\in \B_{u_2}$, we have
$$\tau_n^{\D}(x, y)\leq \tau_n^{\D}(\chi_{u_1}, \chi_{u_2})\leq \tau^{\D}(\chi_{u_1}, \chi_{u_2})=\chi_{\tau^{\wp(^nn)}}(u_1,u_2)=\chi_{\tau(u_1, u_2)}=\chi_{\bold 0}.$$
But for $i\neq j\in I$,
$\tau_n^{\D}(R_i^{\sf M}\cdot \chi_{u_1}, R_j^{\sf M}\cdot \chi_{u_2})=R_{i+j}^{\sf M}\cdot \chi_v$, and so    $\B_{\bold 0}$ will be uncountable,
which is impossible.

We show that \pe\ has a \ws\ in an \ef-game over $(\A, \B)$ concluding that $\A\equiv_{\infty}\B$.
At any stage of the game,
if \pa\ places a pebble on one of
$\A$ or $\B$, \pe\ must place a matching pebble,  on the other
algebra.  Let $\b a = \la{a_0, a_1, \ldots, a_{n-1}}$ be the position
of the pebbles played so far (by either player) on $\A$ and let $\b
b = \la{b_0, \ldots, b_{n-1}}$ be the the position of the pebbles played
on $\B$.  \pe\ maintains the following properties throughout the
game.
\begin{itemize}
\item For any atom $x$ (of either algebra) with
$x\cdot \bold 1_{0}=0$ then $x \in a_i\iff x\in b_i$.
\item $\b a$ induces a finite partion of $\r(0)$ in $\A$ of $2^n$
 (possibly empty) parts $p_i:i<2^n$ and $\b b$ induces a partition of
 $\bold 1_{0}$ in $\B$ of parts $q_i:i<2^n$.  $p_i$ is finite iff $q_i$ is
 finite and, in this case, $|p_i|=|q_i|$.
\end{itemize}

It is easy to see that \pe\ can maintain the two properties in every round. In this back--and--forth game, \pe\ will always find a matching pebble, 
because the pebbles in play are finite.  For each $w\in {}^nn$ 
the component $\B_w=\{x\in \B: x\leq \bold 1_v\}(\subseteq \A_w=\{x\in \A: x\leq \bold 1_v\}$) 
contains infinitely many atoms. 
For any $w\in V$, $|\At\A_w|=|I|$, while 
for  $u\in V\sim \{\bold 0\}$, $\At\A_u=\At\B_u$. For 
$|\At\B_{\bold 0}|=\omega$, but it is still an infinite set.
Therefore $\A\equiv_{\infty}\B$. 

In words: if \pa\ places a pebble on an atom below $\bold 1_{\bold 0}$ of one structure then \pe\ places
her matching pebble on some atom below $\bold 1_{\bold 0}$  of the other structure
so that the new pebble is equal to an old pebble
iff \pa's pebble is equal to the matching old pebble and it is distinct
from all the old pebbles if \pa's pebble is.  Since there are only
finitely many pebbles in play and both $\A$ and $\B$ have
infinitely many atoms below $\bold 1_{\bold 0}$  she can do this.\\
\end{proof}

{\bf Relation algebra case:} Now we give an example of splitting atoms in a Monk--like relation algebra proving that $\Ra\CA_n$, for $n\geq 5$
and closely related classes are not closed under 
$\equiv_{\infty, \omega}$.
We start by defining certain (intergral) relation algebras. The integral relation algebra (in which $\sf Id$ is an atom) defined next
by listing its forbidden triples.

\begin{example}\label{ra}
Take $\R$ to be a symmetric, atomic relation algebra with atoms
$$\Id, \r(i),
\y(i), \bb(i):i<\omega.$$
Non-identity atoms have colours, $\r$ is red,
$\bb$ is blue, and $\y$ is yellow. All atoms are self-converse.
Composition of atoms is defined
by listing the forbidden triples.
The forbidden triples are (Peircean transforms)
or permutations of $(\Id, x, y)$ for $x\neq y$, and
$$(\r(i), \r(i), \r(j)), \; (\y(i), \y(i), \y(j)), \; (\bb(i), \bb(i), \bb(j))\; \; i\leq j < \omega$$
$\R$ is the complex algebra over this atom structure.

Let $\alpha$ be an ordinal.  $\R^\alpha$ is obtained from $\R$ by
splitting the atom $\r(0)$ into $\alpha$ parts $\r^k(0):k<\alpha$
and then taking the full complex algebra.
In more detail, we put red atoms $r^{k}(0)$ for $k<\alpha.$
In the altered algebra the forbidden triples are
$(\y(i), \y(i), \y(j)), (\bb(i), \bb(i), \bb(j)), \ \   i\leq j<\omega,$
$(\r(i), \r(i), \r(j)),  \ \  0<i\leq j<\omega,$
$(\r^k(0), \r^l(0), \r(j)), \ \   0<j<\omega, k,l<\alpha,$\\
$(\r^k(0), \r^l(0), \r^m(0)), \ \  k,l,m<\alpha.$
These algebras were used in \cite{bsl} to show that $\Ra\CA_k$ for all $k\geq 5$ is not elementary.
\end{example}
 
\begin{theorem} \label{thm:el} Let $\mathfrak{n}\geq 2^{\aleph_0}$. 
Let $\A=\R^{\mathfrak{n}}$ and $\B=\R^{\omega}$ (with notation as in example \ref{ra}). Then $\A\equiv_{\infty, \omega}\B$, $\B\in \Ra\CA_{\omega}$, $\A\notin \Ra\CA_5$ and 
$\A$ does not have a complete representation. In particular,  
$\sf CRRA$, and any class $\bold K$  such that $\Ra\CA_{\omega}\subseteq \bold K\subseteq \Ra\CA_5$, are 
not closed under $\equiv_{\infty, \omega}$.
\end{theorem}
\begin{proof}
For an ordinal $\alpha$,  $\R^{\alpha}$ is as defined in example \ref{ra}.
In $\R^\alpha$, we use the
following abbreviations:
$r(0) = \sum_{k<\alpha}\r^k(0)$
$\r = \sum_{i<\omega}\r(i)$
$\y = \sum_{i<\omega}\y(i)$
$\bb = \sum_{i<\omega}\bb(i).$
These suprema exist because they are taken in the complex algebras which are complete.
The \emph{index} of $\r(i), \y(i)$ and $\bb(i)$ is $i$ and the index of
$\r^k(0)$ is also $0$.
Now let  $\B = \R^\omega$ and $\A=\R^{\mathfrak{n}}$ with $\mathfrak{n}\geq 2^{\aleph_0}$,
 be as in the hypothesis of the last item of  theorem \ref{main}. We claim that
$\B\in \Ra\CA_{\omega}$ and $\A\equiv \B$.
For the first required, we show that $\B$ has a cylindric bases by exhibiting a \ws\ for  \pe\
in the the cylindric-basis game, which is a simpler version of the hyperbasis game
\cite[Definition 12.26]{HHbook}. 

Now, let $\H$ be an $\omega$-dimensional cylindric basis for $\B$. Then  
$\Ca\H\in \CA_{\omega}$.  Consider the
cylindric algebra $\C = \Sg^{\Ca\H}\B$, the subalgebra of $\Ca\H$ generated by $\B$.  In principal, new two dimensional elements that 
were not originally in $\B$,  
can be created in $\C$ using the spare dimensions in $\Ca(\H)$.
But next we exclude this possibility. We show that $\B$ exhausts the $2$--dimensional elements of  $\Ra\C$, more concisely, 
we show that 
$\B=\Ra\C$.
We have proved that $\B\in \Ra\CA_{\omega}$ and $\A\equiv \B$. 
In \cite{bsl}, it is proved that $\A\notin \Ra\CA_5$. So we get the first required, namely, that any class $\bold K$, such that 
$\Ra\CA_\omega\subseteq \bold K\subseteq \Ra\CA_5$ is not 
closed under $\equiv_{\infty, \omega}$

Now we show that $\sf CRRA$ is not closed under $\equiv_{\infty, \omega}$, strengthening the result in \cite{HH} that 
only shows that 
$\sf CRRA$ is not closed under elementary equivalence proving 
the remaining required. 

Since $\B\in \Ra\CA_{\omega}$ has countably many atoms, 
then $\B$ is completely representable \cite[Theorem 29]{r}.
For this purpose, we show that $\A$ is not completely representable. We work with the term algebra, $\Tm\At\A$, since the latter is completely representable 
$\iff$ the complex algebra is.
Let  $\r = \{\r(i): 1\leq i<\omega\}\cup \{\r^k(0): k<2^{\aleph_0}\}$, $\y = \{\y(i):  i\in \omega\}$, $\bb^+ = \{\bb(i): i\in \omega\}.$
It is not hard to check every element of $\Tm\At\A\subseteq \wp(\At\A)$ has the form  
$F\cup R_0\cup B_0\cup Y_0$, where $F$ is a finite set of atoms, $R_0$ is either empty or a co-finite subset of $\r$, $B_0$ 
is either empty or a co--finite subset of $\bb$, and $Y_0$ is either empty or a co--finite subset 
of $\y$. 
We show that the existence of a complete representation necessarily forces a 
monochromatic triangle, that we avoided at the start when defining $\A$.

Let $x, y$ be points in the
representation with $M \models \y(0)(x, y)$.  For each $i< 2^{\aleph_0}$, there is a
point $z_i \in M$ such that $M \models {\sf red}(x, z_i) \wedge \y(0)(z_i, y)$ (some red $\sf red\in \r$).
Let $Z = \set{z_i:i<2^{\aleph_0}}$.  Within $Z$ each edge is labelled by one of the $\omega$ atoms in
$\y^+$ or $\bb^+$.  The Erdos-Rado theorem forces the existence of three points
$z^1, z^2, z^3 \in Z$ such that $M \models \y(j)(z^1, z^2) \wedge \y(j)(z^2, z^3)
\wedge \y(j)(z^3, z_1)$, for some single $j<\omega$  
or three  points $z^1, z^2, z^3 \in Z$ such that $M \models \bb(l)(z^1, z^2) \wedge \bb(l)(z^2, z^3)
\wedge \bb(l)(z^3, z_1)$, for some single $l<\omega$.  
This contradicts the
definition of composition in $\A$ (since we avoided monochromatic triangles).
We have proved that $\sf CRRA$ is not closed under $\equiv_{\infty, \omega}$, since $\A\equiv_{\infty, \omega}\B$, 
$\A$ is not completely representable,   but $\B$ is completely representable.
\end{proof}

\section{Applications using rainbows}

In this section we work only with $\CA$s giving several applications of rainbow constructions for cylindric--like algebras. The other cases are very similar.
For the definitions of pseudo--elementary and pseudo--universal,  the reader is referred to \cite[Definition 9.1]{HHbook}.
We denote the class of completely representable $\CA_n$s by $\sf CRCA_n$. We write $\bold S_d$ for the operation of forming dense subalgebrs.
It known that if $\bold K$ is pseudo--universal $\implies \bold K$ is elementary and closed under $\bold S$, cf. \cite[Chapter 10]{HHbook} for an extensive overview of such notions. 
\begin{theorem}\label{rainbow}
Let $2<n<\omega$ and let $k\geq 3$.
\begin{enumarab} 
\item For any class $\bold K$,
such that ${\sf CRCA}_n\cap {\bf S}_d{\sf Nr}_n\CA_{\omega}\subseteq \bold K\subseteq \bold S_c{\sf Nr}_n\CA_{k}$, 
$\bold K$ is not elementary.
Furthemore, any class $\bold L$ such that $\At({\sf Nr}_n\CA_{\omega})\subseteq \bold L\subseteq \bold \At(\bold S_c{\sf Nr}_n\CA_{n+3})$ 
is not elementary. Finally, ${\bf El}{\sf Nr}_n\CA_{\omega}\nsubseteq {\bf S}_d{\sf Nr}_n\CA_{\omega}\iff$ 
any class $\bold L$ such that  ${\sf Nr}_n\CA_{\omega}\subseteq \bold L\subseteq \bold S_c{\sf Nr}_n\CA_{n+3}$, $\bold L$ is not elementary.

\item The classes  ${\sf CRCA}_n$ and ${\sf Nr}_n\CA_m$ for $n<m$ are pseudo--elementary but not elementary, nor pseudo--universal. 
Furthermore, their elementary 
theory is recursively enumerable. For any $n<m$, the class ${\sf Nr}_n\CA_m$ 
is not closed under $L_{\infty,\omega}$ equivalence.

\item There is an atomic $\R\in {\sf Ra}\CA_{\omega}\cap {\bf El}\sf CRRA$ that is not completely representable.
Also, there is an  atomic algebra  $\A\in {\sf Nr}_{n}\CA_{\omega}\cap {\bf El}{\sf  CRCA_n}$, 
that is not completely representable. 
In particular, both $\sf CRRA$ and  ${\sf CRCA}_n$ are not elementary \cite{HH}.
\end{enumarab}
\end{theorem}
\begin{proof}

(1) Fix finite $n>2$. We start by proving the weaker statement obtained from the required by replacing $\bold S_d$ by $\bold S_c$.
One takes the a rainbow--like $\CA_n$, call it $\C$, based on the ordered structure $\Z$ and $\N$.
The reds ${\sf R}$ is the set $\{\r_{ij}: i<j<\omega(=\N)\}$ and the green colours used 
constitute the set $\{\g_i:1\leq i <n-1\}\cup \{\g_0^i: i\in \Z\}$. 
In complete coloured graphs the forbidden triples are like 
the usual rainbow constructions based on $\Z$ and $\N$ specified above,   
but now 
the triple  $(\g^i_0, \g^j_0, \r_{kl})$ is also forbidden if $\{(i, k), (j, l)\}$ is not an order preserving partial function from
$\Z\to\N$.

We will show  in a moment that \pe\ has a \ws\ $\rho_k$ in the $k$--rounded game $G_k(\At\C)$ for all $k\in \omega$ \cite{mlq}.
Hence, using ultrapowers and an elementary chain argument  \cite[Corollary 3.3.5]{HHbook2}, one gets a countable algebra $\B$ 
such that $\B\equiv \C$, and  \pe\ has a \ws\ in $G_{\omega}(\At\B)$. 

The reasoning is as follows: We can assume that $\rho_k$ is deterministic. 
Let $\D$ be a non--principal ultrapower of $\C$.  Then \pe\ has a \ws\ $\sigma$ in $G_{\omega}(\D)$ --- essentially she uses
$\rho_k$ in the $k$'th component of the ultraproduct so that at each
round of $G_{\omega}(\D)$,  \pe\ is still winning in co--finitely many
components, this suffices to show she has still not lost. We can assume that $\C$ is countable by replacing it, without loss, by $\Tm\At\C$.
Winning strategies are preserved. Now one can use an
elementary chain argument to construct countable elementary
subalgebras $\C=\A_0\preceq\A_1\preceq\ldots\preceq\ldots \D$ in this manner.
One defines  $\A_{i+1}$ to be a countable elementary subalgebra of $\D$
containing $\A_i$ and all elements of $\D$ that $\sigma$ selects
in a play of $G_{\omega}(\D)$ in which \pa\ only chooses elements from
$\A_i$. Now let $\B=\bigcup_{i<\omega}\A_i$.  This is a
countable elementary subalgebra of $\D$, hence $\B\equiv \C$, because $\C\equiv \D$,  
and clearly \pe\ has a \ws\ in
$G_{\omega}(\B)$. Then $\B$ 
is completely representable by \cite[Theorem 3.3.3]{HHbook2}. 

Now we give \pe's \ws\ strategy in $G_k(\At\C)$ where $0<k<\omega$ is the number of rounds:  
Let $0<k<\omega$. We proceed inductively. Let $M_0, M_1,\ldots, M_r$, $r<k$ be the coloured graphs at the start of a play of $G_k$ just before round $r+1$.
Assume inductively, that \pe\ computes a partial function $\rho_s:\Z\to \N$, for $s\leq r:$

\begin{enumroman}
\item $\rho_0\subseteq \ldots \rho_t\subseteq\ldots\subseteq\ldots  \rho_s$ is (strict) order preserving; if $i<j\in \dom\rho_s$ then $\rho_s(i)-\rho_s(j)\geq  3^{k-r}$, where $k-r$
is the number of rounds remaining in the game,
and 
$$\dom(\rho_s)=\{i\in \Z: \exists t\leq s, \text { $M_t$ contains an $i$--cone as a subgraph}\},$$

\item for $u,v,x_0\in \nodes(M_s)$, if $M_s(u,v)=\r_{\mu,k}$, $\mu, k\in \N$, $M_s(x_0,u)=\g_0^i$, $M_s(x_0,v)=\g_0^j$,
where $i,j\in \Z$ are tints of two cones, with base $F$ such that $x_0$ is the first element in $F$ under the induced linear order,
then $\rho_s(i)=\mu$ and $\rho_s(j)=k$.
\end{enumroman} 
For the base of the induction \pe\ takes $M_0=\rho_0=\emptyset.$ 
Assume that $M_r$, $r<k$  ($k$ the number of rounds) is the current coloured graph and that \pe\ has constructed $\rho_r:\Z\to \N$ to be a finite order preserving partial map
such conditions (i) and (ii) hold. 

We show that (i) and (ii) can be maintained in a 
further round.
We check the most difficult case. Assume that $\beta\in \nodes(M_r)$, $\delta\notin \nodes(M_r)$ is chosen by \pa\ in his cylindrifier move,
such that $\beta$ and $\delta$ are apprexes of two cones having
same base and green tints $p\neq  q\in \Z$. 
Now \pe\ adds $q$ to $\dom(\rho_r)$ forming $\rho_{r+1}$ by defining the value $\rho_{r+1}(p)\in \N$ 
in such a way to preserve the (natural) order on $\dom(\rho_r)\cup \{q\}$, that is maintaining property (i).
Inductively, $\rho_r$ is order preserving and `widely spaced' meaning that the gap between its elements is
at least $3^{k-r}$, so this can be maintained in a further round.

Now \pe\  has to define a (complete) coloured graph 
$M_{r+1}$ such that $\nodes(M_{r+1})=\nodes(M_r)\cup \{\delta\}.$ 
In particular, she has to find a suitable 
red label for the edge $(\beta, \delta).$
Having $\rho_{r+1}$ at hand she proceeds as follows. Now that $p, q\in \dom(\rho_{r+1})$, 
she lets $\mu=\rho_{r+1}(p)$, $b=\rho_{r+1}(q)$. The red label she chooses for the edge $(\beta, \delta)$ is: (*)\ \  $M_{r+1}(\beta, \delta)=\r_{\mu,b}$.
This way she maintains property (ii) for $\rho_{r+1}.$  Next we show that this is a \ws\ for \pe. 

We check that \pe's strategy is a winning one which entails checking the consistency of newly created triangles proving that $M_{r+1}$ is a 
coloured graph completing the induction.  Since $\rho_{r+1}$ is chosen to preserve order, no new forbidden triple (involving two greens and one red) will be created.
Now we check red triangles only of the form $(\beta, y, \delta)$ in $M_{r+1}$ $(y\in \nodes(M_r)$). 
We can assume that  $y$ is the apex of a cone with base $F$ in $M_r$ and green tint $t$, say,
and that $\beta$ is the appex of the $p$--cone having the same base. 
Then inductively by condition (ii), taking $x_0$ to be the first element of $F$, and taking  
the nodes $\beta, y$, and the tints $p, t$, for $u, v, i, j$,  respectively, we have by observing that 
$\beta, y\in \nodes(M_r)$, $\beta, y\in \dom(\rho_r)$ and $\rho_r\subseteq \rho_{r+1}$, 
the following:  
$M_{r+1}(\beta,y)=M_{r}(\beta, y)=\r_{\rho_{r}(p), \rho_{r}(t)}=r_{\rho_{r+1}(p), \rho_{r+1}(t)}.$
By  her strategy, we have  $M_{r+1}(y,\delta)=\r_{\rho_{r+1}(t), \rho_{r+1}(q)}$ 
and we know by (*) that $M_{r+1}(\beta, \delta)=\r_{\rho_{r+1}(p), \rho_{r+1}(q)}$. 
The triple $(\r_{\rho_{r+1}(p), \rho_{r+1}(t)}, \r_{\rho_{r+1}(t), \rho_{r+1}(q)}, \r_{\rho_{r+1}(p), \rho_{r+1}(q)})$
of reds is consistent  and we are done with this case. 
All other edge labelling and colouring $n-1$ tuples in $M_{r+1}$ 
by yellow shades are  exactly like in \cite{HH}.

On the other hand, \pa\ has a \ws\ in $F^{n+3}(\At\C)$.
The idea here, is that, as is the case with \ws's of \pa\ in rainbow constructions, 
\pa\ bombards \pe\ with cones having distinct green tints demanding a red label from \pe\ to appexes of succesive cones.
The number of nodes are limited but \pa\ has the option to re-use them, so this process will not end after finitely many rounds.
The added order preserving condition relating two greens and a red, forces \pe\ to choose red labels, one of whose indices form a decreasing 
sequence in $\N$.  In $\omega$ many rounds \pa\ 
forces a win, 
so by lemma \ref{n}, $\C\notin \bold S_c{\sf Nr}_n\CA_{n+3}$.

He plays as follows: In the initial round \pa\ plays a graph $M$ with nodes $0,1,\ldots, n-1$ such that $M(i,j)=\w_0$
for $i<j<n-1$
and $M(i, n-1)=\g_i$
$(i=1, \ldots, n-2)$, $M(0, n-1)=\g_0^0$ and $M(0,1,\ldots, n-2)=\y_{\Z}$. This is a $0$ cone.
In the following move \pa\ chooses the base  of the cone $(0,\ldots, n-2)$ and demands a node $n$
with $M_2(i,n)=\g_i$ $(i=1,\ldots, n-2)$, and $M_2(0,n)=\g_0^{-1}.$
\pe\ must choose a label for the edge $(n+1,n)$ of $M_2$. It must be a red atom $r_{mk}$, $m, k\in \N$. Since $-1<0$, then by the `order preserving' condition
we have $m<k$.
In the next move \pa\ plays the face $(0, \ldots, n-2)$ and demands a node $n+1$, with $M_3(i,n)=\g_i$ $(i=1,\ldots, n-2)$,
such that  $M_3(0, n+2)=\g_0^{-2}$.
Then $M_3(n+1,n)$ and $M_3(n+1, n-1)$ both being red, the indices must match.
$M_3(n+1,n)=r_{lk}$ and $M_3(n+1, r-1)=r_{km}$ with $l<m\in \N$.
In the next round \pa\ plays $(0,1,\ldots n-2)$ and re-uses the node $2$ such that $M_4(0,2)=\g_0^{-3}$.
This time we have $M_4(n,n-1)=\r_{jl}$ for some $j<l<m\in \N$.
Continuing in this manner leads to a decreasing
sequence in $\N$.
Let $k\geq 3$ and let $\bold K$ be as in the statement. Then $\C\notin \bold K$, $\B\in \bold K\cap {\sf CRCA}_n$ and 
$\C\equiv \B$, we are done. 

Now we prove the (stronger) required. We first give the general idea. We still work with $\C$.
We can (and will) define a $k$--rounded atomic game stronger than $G_k$ call it $H_k$, for $k\leq \omega$,  
so that if $\B\in \CA_n$ is countable and atomic and \pe\ has a \ws\ in $H_{\omega}(\At\B)$, 
then 
(*) $\At\B\in \At{\sf Nr}_n\CA_{\omega}$ and $\Cm\At\B\in {\sf Nr}_n\CA_{\omega}$. 

Then we show that \pe\ has a \ws\ in $H_k(\At\C)$ for all 
$k\in \omega$, hence using ultrapowers and an elementary chain argument, we get that $\C\equiv \B$, 
for some countable completely representable $\B$ that satisfies the two conditions in (*).
Since $\B\subseteq_d \Cm\At\B$, we get the required result, because $\B\in \bold S_d{\sf Nr}_n\CA_{\omega}$
and as before $\C\notin \bold S_c{\sf Nr}_n\CA_{n+3}$ and $\C\equiv \B$. 

Now we prove the second part. Let $\bold L$ be as specified and $\B$ and $\C=\CA_{\Z, \N}$ be the algebras constructed above. Since an 
atom structure of an algebra is first order interpretable in the algebra, then we have $\B\equiv \C\implies \At\B\equiv \At\C$.
Furthermore $\At\B\in \At({\sf Nr}_n\CA_{\omega})\subseteq \bold L$ (though $\B$ might not be in ${\sf Nr}_n\CA_{\omega}$) 
and $\At\C\notin  \At(\bold S_c{\sf Nr}_n\CA_{n+3})\supseteq \bold L$. 
The last part follows from the fact that
if $\D\in \CA_n$ is atomic, 
then $\At\D\in   \At(\bold S_c{\sf Nr}_n\CA_{n+3})\iff \D\in \bold S_c{\sf Nr}_n\CA_{n+3}$. 
We conclude that $\bold L$ is not elementary

We define the game $H$.
But first some preparation.
Fix $2<n<\omega$.

For an $n$--dimensional atomic network on an atomic $\CA_n$ and for  $x,y\in \nodes(N)$, we set  $x\sim y$ if
there exists $\bar{z}$ such that $N(x,y,\bar{z})\leq {\sf d}_{01}$.
Define the  equivalence relation $\sim$ over the set of all finite sequences over $\nodes(N)$ by
$\bar x\sim\bar y$ iff $|\bar x|=|\bar y|$ and $x_i\sim y_i$ for all
$i<|\bar x|$. (It can be easily checked that this indeed an equivalence relation).

A \emph{ hypernetwork} $N=(N^a, N^h)$ over an atomic $\CA_n$
consists of an $n$--dimensional  network $N^a$
together with a labelling function for hyperlabels $N^h:\;\;^{<
\omega}\!\nodes(N)\to\Lambda$ (some arbitrary set of hyperlabels $\Lambda$)
such that for $\bar x, \bar y\in\; ^{< \omega}\!\nodes(N)$
if $\bar x\sim\bar y \Rightarrow N^h(\bar x)=N^h(\bar y).$
If $|\bar x|=k\in \N$ and $N^h(\bar x)=\lambda$, then we say that $\lambda$ is
a $k$-ary hyperlabel. $\bar x$ is referred to as a $k$--ary hyperedge, or simply a hyperedge.
We may remove the superscripts $a$ and $h$ if no confusion is likely to ensue.

A hyperedge $\bar{x}\in {}^{<\omega}\nodes(N)$ is {\it short}, if there are $y_0,\ldots, y_{n-1}$
that are nodes in $N$, such that
$N(x_i, y_0, \bar{z})\leq {\sf d}_{01}$
or $\ldots N(x_i, y_{n-1},\bar{z})\leq {\sf d}_{01}$
for all $i<|x|$, for some (equivalently for all)
$\bar{z}.$
Otherwise, it is called {\it long.}
A hypernetwork $N$
is called {\it $\lambda$--neat} if $N(\bar{x})=\lambda$ for all short hyperedges.

Concerning \pa's  moves, $H_m$ has  $m$ rounds, $m\leq \omega$.  
He can play a cylindrifier move, like before but now played on $\lambda$---neat hypernetworks
with $\lambda$ a constant label on short hyperedges.
Also \pa\ can play a \emph{transformation move} by picking a
previously played $\lambda$--neat hypernetwork $N$ and a partial, finite surjection
$\theta:\omega\to\nodes(N)$, this move is denoted $(N, \theta)$.  \pe's
response is mandatory. She must respond with $N\theta$.
Finally, \pa\ can play an
\emph{amalgamation move} by picking previously played $\lambda$--neat hypernetworks
$M, N$ such that
$M\restr {\nodes(M)\cap\nodes(N)}=N\restr {\nodes(M)\cap\nodes(N)},$
and $\nodes(M)\cap\nodes(N)\neq \emptyset$.
This move is denoted $(M,
N).$
To make a legal response, \pe\ must play a $\lambda$--neat
hypernetwork $L$ extending $M$ and $N$, where
$\nodes(L)=\nodes(M)\cup\nodes(N)$.
We claim that  \pe\ has a \ws\ in $H_{m}(\At\C)$ for each finite $m$.
The analogous proof for relation algebras is rather long
\cite[p.25--31]{r}. We assume that the 
claim is true and  take it from there, then we give the general idea skipping some of the intricate details.

Using the usual technique of forming ultrapowers followed by 
an elementary  chain argument, we get that there exists a countable (completely representable) algebra, which we continue to denote by a slight abuse of notation
also $\B$, such that $\C\equiv \B$, and  \pe\ has a \ws\ on $H(\At\B)$. For brevity, let $\alpha=\At\B$.
Using \pe's \ws\ in $H$, one builds an $\omega$--dilation $\D_a$ of $\B$
for every $a\in \At\B$, based on a structure $\M_a$ in some signature to be specified shortly. 
Strictly speaking, $\M_a$ will be a {\it weak model}, 
where assignments are {\it relativized}, they are required to agree 
co--finitely with a fixed sequence in $^{\omega}\M_a$.
This weak model $\M_a$ is taken in a signature $L$ consisting of one $n$--ary relation for each $b\in \At\B$ and 
a $k$--ary relation symbol  for each hyperedge of length $k$ labelled by $\lambda$ the constant neat hyperlabel. 

For $a\in \At\B$, the weak model $\M_a$ is the limit of the play $H$; in the sense that $\M_a$ is the union of the $\lambda$--neat 
hypernetworks on $\B$ played during the game $H$, with starting point the initial atom $a$ that \pa\ chose in the first move.  
Labels for the edges and hyperedges in $\M_a$ 
are defined the obvious way,  inherited from the $\lambda$--neat hypernetworks played during the $\omega$--rounderd game $H_{\omega}(\At\B)$. 
These are nested, so this labelling is well 
defined giving an 
interpretation of {\it only} the atomic formulas of $L$ in $\M_a$.  
There is some space here in `completing' the interpretation. One uses 
an extension $\L$, not necessarily a proper one, of  $L_{\omega, \omega}$ as a vehicle for constructing 
$\D_a$.  The algebra $\D_a$ will  be a  {\it weak set algebra} based on $\M_a$ of $\L$--formulas taken in the signature $L$. That is the 
base in the sense of \cite[Definition 3.1.1]{HMT2} of $\D_a$ is $\M_a$, 
and the set--theoretic operations of $\D_a$ are read off the connectives in $\L$. 
 In all cases, as long as $\L$ contains $L_{\omega, \omega}$ as a fragment, we get that $\B$ neatly embeds into $\D$, that is $\B\subseteq \Nr_n\D$, where 
$\D={\bf P}_{a\in \At\B}\D_a$.  

We are faced here with three possibilites measuring `how close' $\B$ is to $\Nr_n\D$. We go from the closest to the less close. 
Either (a) $\B=\Nr_n\D$  or (b) $\B\subseteq_d \Nr_n\D$ or  (c) $\B\subseteq_c \Nr_n\D$. From the first part, building 
$\D$ using the weaker game $G$ used in the proof of the previous item, we can get the last possibility. It is reasonable to expect that the stronger $\L$ is, the 
`more  control' $\At\B$ has over the hitherto obtained $\omega$--dilation $\D$; the closer $\B$ 
is  to the neat $n$--reduct of $\D$ based on $\L$-formulas. 
If (a) is true than any $\bold K$ between ${\sf Nr}_n\CA_{\omega}\cap \sf CRCA_n$ and $\bold S_c{\sf Nr}_n\CA_{n+3}$ would be non--elementary.
{\it We could not prove (a)}.
So let us approach the two remaining possibilities (b) and (c).
Suppose we take  $\L=L_{\infty, \omega}$. Then using the fact that in the $\lambda$--neat hypernetworks played during the game $H$ 
short hyperedges are constantly
labelled by $\lambda$, one can show that $\B$ and $\Nr_n\D$ have {\it isomorphic 
atom structures}, in symbols $\At\B\cong \At\Nr_n\D$ as follows.
For brevity, denote the hitherto obtained $\At\B$ by $\alpha$.

Fix some $a\in\alpha$. Using \pe\ s \ws\ in the game $H(\alpha)$ played on $\lambda$--neat hypernetworks $\lambda$ a constant label kept on short hyperedges,
one defines a
nested sequence $M_0\subseteq M_1,\ldots$ of $\lambda$--neat hypernetworks
where $M_0$ is \pe's response to the initial \pa-move $a$, such that:
If $M_r$ is in the sequence and $M_r(\bar{x})\leq {\sf c}_ia$ for an atom $a$ and some $i<n$,
then there is $s\geq r$ and $d\in\nodes(M_s)$
such that  $M_s(\bar{y})=a$,  $\bar{y}_i=d$ and $\bar{y}\equiv_i \bar{x}$.
In addition, if $M_r$ is in the sequence and $\theta$ is any partial
isomorphism of $M_r$, then there is $s\geq r$ and a
partial isomorphism $\theta^+$ of $M_s$ extending $\theta$ such that
$\rng(\theta^+)\supseteq\nodes(M_r)$ (This can be done using \pe's responses to amalgamation moves).

Now let $\M_a$ be the limit of this sequence, that is $\M_a=\bigcup M_i$, the labelling of $n-1$ tuples of nodes
by atoms, and hyperedges by hyperlabels done in the obvious way.
Let $L$ be the signature with one $n$-ary relation for
each $b\in\alpha=\At\B$, and one $k$--ary predicate symbol for
each $k$--ary hyperlabel $\lambda$.
{\it Now we work in $L_{\infty, \omega}.$}
For fixed $f_a\in\;^\omega\!\nodes(\M_a)$, let
$\U_a=\set{f\in\;^\omega\!\nodes(\M_a):\set{i<\omega:g(i)\neq
f_a(i)}\mbox{ is finite}}$.
Now we  make $\U_a$ into the base of an $L$ relativized structure 
${\cal M}_a$ like in \cite[Theorem 29]{r} except that we allow a clause for infinitary disjunctions.
$${\cal M}_a, f\models (\bigvee_{i\in I} \phi_i)\iff(\exists i\in I)({\cal M}_a,  f\models\phi_i).$$
We are now
working with (weak) set algebras  whose semantics is induced by $L_{\infty, \omega}$ formulas in the signature $L$,
instead of first order ones.
For any such $L$-formula $\phi$, write $\phi^{{\cal M}_a}$ for
$\set{f\in\U_a: {\cal M}_a, f\models\phi}.$
Let $D_a= \set{\phi^{{\cal M}_a}:\phi\mbox{ is an $L$-formula}}$ and
$\D_a$ be the weak set algebra with universe $D_a$. 
Let $\D=\bold P_{a\in \alpha} \D_a$. Then $\D$ is a {\it generalized weak set algebra} \cite[Definition 3.1.2 (iv)]{HMT2}.
Let $x\in \D$. Then $x=(x_a:a\in\alpha)$, where $x_a\in\D_a$.  For $b\in\alpha$ let
$\pi_b:\D\to \D_b$ be the projection map defined by
$\pi_b(x_a:a\in\alpha) = x_b$.  Conversely, let $\iota_a:\D_a\to \D$
be the embedding defined by $\iota_a(y)=(x_b:b\in\alpha)$, where
$x_a=y$ and $x_b=0$ for $b\neq a$.  
We show that $\alpha\cong \At\Nr_n\D$ and that $\Cm\alpha\cong \Nr_n\D$. 

Suppose $x\in\Nr_n\D\setminus\set0$.  Since $x\neq 0$,
then it has a non-zero component  $\pi_a(x)\in\D_a$, for some $a\in \alpha$.
Assume that $\emptyset\neq\phi(x_{i_0}, \ldots, x_{i_{k-1}})^{\D_a}= \pi_a(x)$, for some $L$-formula $\phi(x_{i_0},\ldots, x_{i_{k-1}})$.  We
have $\phi(x_{i_0},\ldots, x_{i_{k-1}})^{\D_a}\in\Nr_{n}\D_a$.
Pick
$f\in \phi(x_{i_0},\ldots, x_{i_{k-1}})^{\D_a}$  
and assume that ${\cal M}_a, f\models b(x_0,\ldots x_{n-1})$ for some $b\in \alpha$.
We show that
$b(x_0, x_1,\ldots, x_{n-1})^{\D_a}\subseteq
 \phi(x_{i_0},\ldots, x_{i_{k-1}})^{\D_a}$.  
Take any $g\in
b(x_0, x_1\ldots, x_{n-1})^{\D_a}$,
so that ${\cal M}_a, g\models b(x_0, \ldots, x_{n-1})$.  
The map $\{(f(i), g(i)): i<n\}$
is a partial isomorphism of ${\cal M}_a.$ Here that short hyperedges are constantly labelled by $\lambda$ 
is used.
This map extends to a finite partial isomorphism
$\theta$ of $M_a$ whose domain includes $f(i_0), \ldots, f(i_{k-1})$.
Let $g'\in {\cal M}_a$ be defined by
\[ g'(i) =\left\{\begin{array}{ll}\theta(i)&\mbox{if }i\in\dom(\theta)\\
g(i)&\mbox{otherwise}\end{array}\right.\] 
We have ${\cal M}_a,
g'\models\phi(x_{i_0}, \ldots, x_{i_{k-1}})$. But 
$g'(0)=\theta(0)=g(0)$ and similarly $g'(n-1)=g(n-1)$, so $g$ is identical
to $g'$ over $n$ and it differs from $g'$ on only a finite
set.  Since $\phi(x_{i_0}, \ldots, x_{i_{k-1}})^{\D_a}\in\Nr_{n}\D_a$, we get that
${\cal M}_a, g \models \phi(x_{i_0}, \ldots,
x_{i_k})$, so $g\in\phi(x_{i_0}, \ldots, x_{i_{k-1}})^{\D_a}$ (this can be proved by induction on quantifier depth of formulas).  
This
proves that 
$b(x_0, x_1\ldots x_{n-1})^{\D_a}\subseteq\phi(x_{i_0},\ldots,
x_{i_k})^{\D_a}=\pi_a(x),$ and so
$\iota_a(b(x_0, x_1,\ldots x_{n-1})^{\D_a})\leq
\iota_a(\phi(x_{i_0},\ldots, x_{i_{k-1}})^{\D_a})\leq x\in\D_a\setminus\set0.$

Now every non--zero element 
$x$ of $\Nr_{n}\D_a$ is above a non--zero element of the following form 
$\iota_a(b(x_0, x_1,\ldots, x_{n-1})^{\D_a})$
(some $a, b\in \alpha$) and these are the atoms of $\Nr_{n}\D_a$.  
The map defined  via $b \mapsto (b(x_0, x_1,\dots, x_{n-1})^{\D_a}:a\in \alpha)$ 
is an isomorphism of atom structures, 
so that $\alpha=\At\B\in \At{\sf Nr}_n\CA_{\omega}$.

Because we are working in $L_{\infty, \omega},$ infinite disjuncts exist in $\D_a$ $(a\in \alpha)$,
hence, they exist too in the dilation $\D=\bold P_{a\in\alpha}\D_a$.
Therefore $\D$ is complete, so $\Nr_n\D$ is complete, too.
Indeed, let  $X\subseteq \Nr_n\D$. Then by completeness of $\D$, we get that
$d=\sum^{\D}X$ exists.  Assume that  $i\notin n$, then
${\sf c}_id={\sf c}_i\sum X=\sum_{x\in X}{\sf c}_ix=\sum X=d,$
because the ${\sf c}_i$s are completely additive and ${\sf c}_ix=x,$
for all $i\notin n$, since $x\in \Nr_n\D$.
We conclude that $d\in \Nr_n\D$,
and so $\Nr_n\D$ is complete as claimed. 

Now $\D={\bf P}_{a\in \At\B}\D_a$  and 
its $n$--neat reduct $\Nr_n\D$ are complete. 
Accordingly, we can make the identification  
$\Nr_n\D\subseteq_d \Cm\At\B$.  By density,
we get that  $\Nr_n\D=\Cm\At\B$ (since $\Nr_n\B$ is complete),  
hence $\Cm\At\B\in {\sf Nr}_n\CA_{\omega}$.  

Using only $\Cm\At\B\in {\sf Nr}_n\CA_{\omega}$, 
we get  that $\B\in \bold S_d{\sf Nr}_n\CA_{\omega}$, 
because $\B$ is dense in its \de\ completion. 
Hence we attain the second possibility.
But it will now readily follows that any class $\bold K$, such that $\bold S_d{\sf Nr}_n\CA_{\omega}\cap {\sf CRCA_n}\subseteq \bold K\subseteq \bold S_c{\sf Nr}_n\CA_{n+3}$
is not elementary, where $\bold S_d$ denotes the operation of forming dense subalgebras.
Indeed, we have $\B\subseteq_d \Cm\At\B\in \bold {\sf Nr}_n\CA_{\omega}\cap \sf CRCA_n\subseteq \bold K$, 
$\C\notin \bold S_c{\sf Nr}_n\CA_{n+3}\supseteq \bold K$, and $\C\equiv \B$.
 Non--elementarity of $\bold N_k$ follows from $\C\equiv \B$, 
$\Cm\At\B\in {\sf Nr}_n{\sf CA}_{\omega}$ and $\C\notin {\sf Nr}_n\CA_{\omega}(\supseteq \bold S_c{\sf Nr}_n\CA_{n+3}$). 

For the last part, It suffices to consider classes between ${\sf Nr}_n\CA_{\omega}$ and $\bold S_d{\sf Nr}_n\CA_{\omega}$. 
One implication, namely $\Longleftarrow$  is trivial. For the other less trivial implication, assume for contradiction that there is such a class $\bold K$ that is elementary.
Then ${\bf El}{\sf Nr}_n\CA_{\omega}\subseteq \bold K$, because $\bold K$ is elementary.
It readily follows that  ${\sf Nr}_n\CA_{\omega}\subseteq {\bf El}{\sf Nr}_n\CA_{\omega}\subseteq \bold K\subseteq \bold S_d{\sf Nr}_n{\sf CA}_{\omega}$,
which is impossible by the 
given assumption that ${\bf El}{\sf Nr}_n\CA_{\omega}\subsetneq \bold S_d{\sf Nr}_n{\sf CA}_{\omega}$.\\

Now we give an outline of the \ws\ of \pe\ in $H_m(\At\C)$ for all finite $m$. 
For a start, we change the board of play but only formally. The play is now on {\it $\lambda$--neat hypergraphs}.  
Given a rainbow algebra $\A$, there is a one to one correspondence between coloured graphs on $\At\A$ and networks on $\At\A$ \cite[Half of p. 76]{HHbook2}
denote this correspondence, expressed by a bijection from coloured graphs to networks by  (*): 
$$\Gamma\mapsto N_{\Gamma}, \ \ \nodes(\Gamma)=\nodes(N_{\Gamma}).$$
Now the game $H$ can be re-formulated to be played on {\it $\lambda$--neat hypergraphs} on a rainbow algebra $\A$; these are of the  form
$(\Delta, N^h)$, where $\Delta$ is a coloured graph on $\At\A$, $\lambda$ is a hyperlabel, and 
$N^h$ is as before, $N^h: ^{<\omega}\nodes(\Delta)\to \Lambda$,
such that for $\bar x, \bar y\in\; ^{< \omega}\!\nodes(\Delta)$,
if $\bar x\sim\bar y \Rightarrow N^h(\bar x)=N^h(\bar y).$ Here $\bar{x}\sim \bar{y}$, making the obvious translation,
 is the equivalence relation defined by:  $x\sim y\iff$ $|x|=|y|$ and $N_{\Delta}(x_i, y_i, \bar{z})\leq {\sf d}_{01}$ for all $i<|x|$ and some 
$\bar{z}\in {}^{n-2}\nodes(\Delta)$.

All notions earlier defined for hypernetworks, in particular, $\lambda$--neat ones,  translate to 
$\lambda$--neat hypergraphs, using (*),
like short hyperdges, long hypedges, $\lambda$--neat hypergraphs, etc.
The game is played now on $\lambda$--neat hypergraphs on which the constant label
$\lambda$ is kept on the short hyperedges in $^{<\omega}\nodes(\Delta)$.
Labelling $\lambda$--neat hyperedges is done as follows:
In a play, \pe\ is required to play $\lambda$--neat hypernedges so she has no choice about the
the short edges, these are labelled by $\lambda$. In response to a cylindrifier move by \pa\
extending the current hypergraph providing a new node $k$,
and a previously played coloured hyperngraph $M$
all long hyperedges not incident with $k$ necessarily keep the hyperlabel they had in $M$.

All long hyperedges incident with $k$ in $M$
are given unique hyperlabels not occurring as the hyperlabel of any other hyperedge in $M$.
In response to an amalgamation move, which involves two hypernetworks required to be amalgamated, say $(M,N)$
all long hyperedges whose range is contained in $\nodes(M)$
have hyperlabel determined by $M$, and those whose range is contained in $\nodes(N)$ have hyperlabels determined
by $N$. If $\bar{x}$ is a long hyperedge of \pe\ s response $L$ where
$\rng(\bar{x})\nsubseteq \nodes(M)$, $\nodes(N)$ then $\bar{x}$
is given
a new hyperlabel, not used in 
any previously played hypernetwork and not used within $L$ as the label of any hyperedge other than $\bar{x}$.
This completes her strategy for labelling hyperedges.

We have already dealt with the `graph part' in the proof above  
We turn to the remaining amalgamation moves. We need some notation and terminology taken from \cite[pp.25]{r}.
Every edge of any hypergraph has an {\it owner \pa\ or \pe}, namely, the one who coloured this edge.
We call such edges \pa\ edges or \pe\ edges. Each long hyperedge $\bar{x}$ in $N^h$ of a hypergraph $N$
occurring in the play has {\it an envelope} $v_N(\bar{x})$ to be defined shortly.\\
In the initial round,  \pa\ plays $a\in \alpha$ and \pe\ plays $N_0$
then all edges of $N_0$ belongs to \pa.
There are no long hyperedges in $N_0$. If in a later move,
\pa\ plays the transformation move $(N,\theta)$
and \pe\ responds with $N\theta$, then owners and envelopes are inherited in the obvious way.
If \pa\ plays a cylindrifier move requiring a new node $k$ and \pe\ responds with $M$ then the owner
in $M$ of an edge not incident with $k$ is the same as it was in $N$
and the envelope in $M$ of a long hyperedge not incident with $k$ is the same as that it was in $N$.
All  edges $(l,k)$
for $l\in \nodes(N)\sim \{k\}$ belong to \pe\ in $M$.
if $\bar{x}$ is any long hyperedge of $M$ with $k\in \rng(\bar{x})$, then $v_M(\bar{x})=\nodes(M)$.\\
If \pa\ plays the amalgamation move $(M,N)$ (of two $\lambda$--neat hypergraphs) and \pe\ responds with $L$
then for $m\neq n\in \nodes(L)$ the owner in $L$ of a edge $(m,n)$ is \pa\ if it belongs to
\pa\ in either $M$ or $N$, in all other cases it belongs to \pe\ in $L$.
If $\bar{x}$ is a long hyperedge of $L$
then $v_L(\bar{x})=v_M(\bar{x})$ if $\rng(\bar{x})\subseteq \nodes(M)$, $v_L(\bar{x})=v_N(\bar{x})$ and  $v_L(\bar{x})=\nodes(M)$ otherwise.
This completes the definition of owners and envelopes.
The next claim, basically, reduces amalgamation moves to cylindrifier moves.
By induction on the number of rounds one can show:

{\bf Claim}:\label{r}  
Let $M, N$ occur in a play of $H_m$, $0<m\in \omega.$ in which \pe\ uses the above labelling
for hyperedges. Let $\bar{x}$ be a long hyperedge of $M$ and let $\bar{y}$ be a long hyperedge of $N$.
Then for any hyperedge $\bar{x}'$ with $\rng(\bar{x}')\subseteq v_M(\bar{x})$, if $M(\bar{x}')=M(\bar{x})$
then $\bar{x}'=\bar{x}$. 
If $\bar{x}$ is a long hyperedge of $M$ and $\bar{y}$ is a long hyperedge of $N$, and $M(\bar{x})=N(\bar{y}),$
then there is a local isomorphism $\theta: v_M(\bar{x})\to v_N(\bar{y})$ such that
$\theta(x_i)=y_i$ for all $i<|x|$. For any $x\in \nodes(M)\sim v_M(\bar{x})$ and $S\subseteq v_M(\bar{x})$, if $(x,s)$ belong to \pa\ in $M$
for all $s\in S$, then $|S|\leq 2$.

Next,  we proceed inductively with the inductive hypothesis exactly as before, except that now each $N_r$ is a
$\lambda$--neat hypergraph.
All what remains is the amalgamation move. With the above claim at hand,
this turns out an easy task to implement guided by \pe\ s
\ws\ in the graph part.\\ 
We consider an amalgamation move at round $0<r$, $(N_s,N_t)$ chosen by 
\pa\ in round $r+1$, \pe\ has to deliver an amalgam $N_{r+1}$.
 \pe\ lets $\nodes(N_{r+1})=\nodes(N_s)\cup \nodes (N_t)$, then she, for a start, 
has to choose a colour for each edge $(i,j)$ where $i\in \nodes(N_s)\sim \nodes(N_t)$ and $j\in \nodes(N_t)\sim \nodes(N_s)$.
Let $\bar{x}$ enumerate $\nodes(N_s)\cap \nodes(N_t).$
If $\bar{x}$ is short, then there are at most two nodes in the intersection
and this case is identical to the cylindrifier move.
If not, that is if $\bar{x}$ is long in $N_s$, then by the claim
there is a partial isomorphism $\theta: v_{N_s}(\bar{x})\to v_{N_t}(\bar{x})$ fixing
$\bar{x}$. We can assume that
$v_{N_s}(\bar{x})=\nodes(N_s)\cap \nodes (N_t)=\rng(\bar{x})=v_{N_t}(\bar{x}).$
It remains to label the edges $(i,j)\in N_{r+1}$ where $i\in \nodes(N_s)\sim \nodes (N_t)$ and $j\in \nodes(N_t)\sim \nodes(N_s)$.
Her strategy is now again similar to the cylindrifier move. If $i$ and $j$ are tints of the same cone she chooses a red using $\rho_{r+1}$ (constructed inductively like in the above proof),
if not she  chooses  a white. She never chooses a green.

Concerning $n-1$ tuples 
she  needs to label $n-1$ hyperedges by shades of yellow.
For each tuple $\bar{a}=a_0,\ldots a_{n-2}\in N_{r+1}^{n-1}$,   with no edge
$(a_i, a_j)$ coloured green (we have already labelled edges), then  \pe\ colours $\bar{a}$ by $\y_S$, where
$$S=\{i\in \Z: \text { there is an $i$ cone in $N_{r+1}$ with base $\bar{a}$}\}.$$
We have shown briefly 
that  \pe\ has a \ws\ in $H_{m}(\At\C)$ 
for each finite $m$.

(2)  The class ${\sf CRCA}_n$ is not elementary by the proof of the first item, cf. \cite{HH},  hence it is not pseudo--univeral. It is also not 
closed under $\bold S$: Take any representable algebra that is not completely representable,  
for example an infinite algebra that is not atomic. Other atomic examples is the term algebra $\Tm \bf At$ dealt with in the proof of theorem \ref{can} 
and $\C$ dealt with above. The former  is not completely representable because a complete representation of $\Tm\bf At$ induces a representation of $\Cm\bf At$ 
which we know is outside 
$\bold S{\sf Nr}_n\CA_{n+3}$. 
Call such an elgebra $\A$. 
Then $\A^+$  is completely representable, a classical result of Monk's \cite{HH} and $\A$ embeds into $\A^+$.
For pseudo--elementarity one proceeds like the relation algebra case \cite[pp. 279--280]{HHbook} 
defining complete representability 
in a two--sorted theory, undergoing the obvious modifications.
For pseudo--elementarity  for the class ${\sf Nr}_n\CA_{\beta}$ for any $2<n<\beta$  one easily adapts \cite[Theorem 21]{r} by defining  ${\sf Nr}_n\CA_\beta$ 
in a two--sorted theory, when $1<n<\beta<\omega$, and a three--sorted one, when
$\beta=\omega$. The first part is easy.  For the second part; one uses a sort for a $\CA_n$
$(c)$, the second sort is for the Boolean reduct of a $\CA_n$ $(b)$
and the third sort for a set of dimensions $(\delta).$

For any infinite ordinal $\mu$, the defining theory for ${\sf Nr}_n\CA_{\mu}={\sf Nr}_n{\sf CA}_{\omega}$,
includes sentences requiring that the constants $i^{\delta}$ for $i<\omega$
are distinct and that the last two sorts define
a $\CA_\omega$. There is a function $I^b$ from sort $c$ to sort $b$ and sentences forcing  that $I^b$ is injective and
respects the $\CA_n$ operations. For example, for all $x^c$ and $i<n$,
$I^b({\sf c}_i x^c)= {\sf c}_i^b(I^b(x^c)).$ The last requirement is that $I^b$ maps {\it onto} the set of $n$--dimensional elements. This can  be easily expressed
via (*)
$$\forall y^b((\forall z^{\delta}(z^{\delta}\neq 0^{\delta},\ldots (n-1)^{\delta}\implies  c^b(z^{\delta}, y^b)=y^b))\iff \exists x^c(y^b=I^b(x^c))).$$
In all cases, it is clear that any algebra of the right type is the first sort of a model of this theory.
Conversely, a model for this theory will consist of  $\A\in \CA_n$  (sort $c$),
and a $\B\in \CA_{\omega}$;  the dimension of the last is the cardinality of
the $\delta$--sorted elements which is $\omega$, such that by (*) $\A=\Nr_n\B$.
Thus this three--sorted theory defines the class of neat reducts;
furthermore, it is clearly recursive. Recursive enumerability follows from \cite[Theorem 9.37]{HHbook}.

We have proved 
non--elementarity.

(3) For the last required fix finite $n>2$.
In \cite[Remark 31]{r} a relation atomic algebra $\R$ having uncountably many atoms 
is constructed such that $\R$ has an $\omega$--dimensional cylindric basis $H$ and $\R$ is not completely representable. If one takes $\C=\Ca(H)$, then $\C\in \CA_{\omega}$,
$\C$ is atomless, and $\R=\Ra\C$. The required $\CA_n$ is $\B=\Nr_n\C$; $\B$ is atomic and has uncountably many atoms. 
Furthermore, $\B$ has no complete representation for a complete representation of $\B$ induces one of $\R$.

We show that $\B$ is in  ${\bf El}\sf CRCA_n$. Since $\B\in {\sf Nr}_n\CA_{\omega}$, 
then  by lemma \ref{n}, \pe\ has a \ws\ in $G_{\omega}(\At\B)$, hence  \pe\ has a \ws\ in $G_k(\At\B)$ for all $k<\omega$.
Using ultrapowers and an elementary chain argument \cite[Corollary 3.3.5]{HHbook2}, we get
that $\B\equiv \D$, for some countable atomic $\D$, and \pe\ has a \ws\ in $G_{\omega}(\At\D)$. Since $\D$ is countable
then by \cite[Theorem 3.3.3]{HHbook2} it is completely representable. We have
proved that 
$\B\in {\bf El}{\sf CRK}_n$. Since $\B\notin {\sf CRCA}_n$, then ${\sf CRCA}_n$ is not elementary. 

For relation algebras we have $\R\in \sf Ra\CA_{\omega}$ and $\R$ has no complete representation. The rest is like the $\CA$ case, 
using \cite[Theorem 33]{r},  when the dilation is $\omega$--dimensional, 
namely, $\R\in \bold S_c\sf Ra\CA_{\omega}\implies$ \pe\ has a \ws\ in $F^{\omega}$  equivalently in $G_{\omega}$ 
(the last two games formulated for 
$\RA$s the former as in \cite[Definition 28]{r}).
\end{proof}

\section{Splitting in relation Monk--like algebras and the neat embedding problem}

The model--theoretic ideas used in the theorems \ref{can} and the construction in \cite{weak} are quite similar.
The model theory used for both constructions is almost identical.
Nevertheless,  from the algebraic point of view,
there is a crucial difference.
The non--representability of $\Cm\At\A$ was tested by a game between the two players \pa\ and \pe.
The \ws's of the two players are independent, this is  reflected by the fact
that we have  two `independent parameters' $\sf G$ (the greens)  and $\sf R$ (the reds) that are 
more were finite irreflexive complete graphs. 
In  Monk--like constructions like the one used in \cite{weak} to show that $\RCA_n$ is not atom--canonical by constructing a countable atomic $\A$
the non represenatbility of $\Cm\At\A$ was also tested by a game between \pe\ and \pa. 
But in {\it op.cit}  \ws's are interlinked, one operates through the other; hence only one parameter is the source of colours.

This parameter in \cite{weak} is a graph $\GGG$ (a countable disjoint unon of $N$ cliques $n(n-1)/2\leq N<\omega$). Representability of the complex algebra $\Cm\At\A$ in
this case depends only on the {\it chromatic number of $\GGG$}, via an application of Ramseys' theorem.
{\it In both cases two players operate using `cardinality of a labelled graph'.
\pa\ trying to make this graph too large for \pe\ to cope, colouring some of
its edges suitably.} 

For the rainbow case, it is a red clique formed during the play as we have seen in theorem \ref{can} 
It might be clear in both cases (rainbow and Monk--like algebras),  to see that \pe\ cannot win the infinite game, but what might not
be clear is {\it when does this happens; we know it eventually happen in finitely many round, but till which round \pe\ has not lost yet}.
The non--representability of $\Cm\At\A$ amounts  to that $\Cm\At\A\notin \bold S{\sf Nr}_n\CA_{n+k}$ for some finite $k$ because
$\RCA_n=\bigcap_{k<\omega}\bold S{\sf Nr}_n\CA_{n+k}$. Can we {\bf pin down} the value of $k$? 
getting an estimate that is not `infinitely' loose. 
{\it By adjusting the number of greens in the proof of theorem \ref{can} to be $n+1$ 
one gets a finer result than Hodkinson's \cite{Hodkinson} where Hodkinson uses an `overkill' of infinitely many greens. 
By truncating  the greens to be $n+1$, 
we could tell when $\bold S{\sf Nr}_n\CA_{n+k}$, 
$3\leq k\leq \omega$ {\bf `stops to be atom--canonical'; when $\Cm\At\A$ {\bf `stopped to be representable'}}}

{\bf But sometimes Monk--like algebras, without a rainbow intervention using the independent parameter (of greens) $\sf G$ are  
efficient  in controlling `the number of spare dimensions needed' in a certain construction.
The relation algebras to be constructed next are an example. Though such 
algebras have  huge joins and meets arising from a {\it splitting of a finite set of atoms each into a 
finite set of subatoms}, here the cylindrifiers arising from a cylindric basis that such algebras happen to posses, are defined in such  way to be 
{\it correlated directly with the colours}. This way requires only looking at `enough objects' to see 
if the algebras posseses  
a cylindric basis or not (which it does).
In such a relation algebra (denoted below by $\R$, consider a top point that is connected by coloured edges to 
some of the intermediate points that are all connected to a bottom elements. 
The number of points (in this figure) is the number of colours plus one. 

So one gets the same control as rainbow algebras provided by (the second 
independent parameter) $\sf G$. The key idea here 
is that the proof of Ramsey in this context does not require an uncontrollable Ramsey number of spare dimensions, 
but only one more than the number of 
colours  used.}
Let us elaborate some more giving some but not all of the details:

In \cite{t}, the famous {\it Neat embedding Problem}, posed as \cite[Problem 2.12]{HMT2} for $\CA$s
is approached for diagonal free cylindric algebras $\Df$s, Pinter's substitution algebras $\Sc$s, polyadic algebras $\PA$s, quasi--polyadic algebras $\QA$,
$\PA$s with equality $\PEA$s, and $\QA$s with equality $\QEA$s
It is proved in {\it op. cit} that for any class $\K$ between $\Sc$ and $\QEA$, for any positive $k$,  and for any ordinal $\alpha>2$, 
the class $\bold S{\sf Nr}_\alpha\K_{\alpha+k+1}$ is not axiomatizable by a finite schema over 
$\bold S{\sf Nr}_n\K_{\alpha+k}$.  We strengthen this result when $\alpha\geq \omega$ and when we have diagonal
elements, namely, for any class $\K$ between $\CA_{\alpha}$ and $\QEA_{\alpha}$.

The construction used in this subsection is based on relation algebras.
Here we split atoms in a finite relation algebra each to finitely many subatoms to get a certain finite relation 
$\R\in \bold S\Ra\CA_n\sim \bold S\Ra\CA_{n+1}$. 

A typical such relation algebra $\R$ is a finite relation algebra with $n+1$ `special elements' for some finite $n>3$. 
Composition is defined so that any two distinct elements in the base of any representation of $\R$ (if there is one) are related by one
of these special elements, but no three elements are related by the same special element.

There is a finite bound $m$ say on the number  of elements in the base of a representation 
such that any colouring of the edges of a complete graph of size 
larger than $m$ and using only $n+1$ colours is bound to have a monochromatic triangle. 

Furthermore,  a represenation of $\R$ enforces 
that the base of the representation has to have 
more than $m$ objects.  This can be demonstrated  by a \ws\ of \pa\ in an infinite
(hyperbasis) game, allowing amalgamation moves, where the approximations to
the genuine (potential) representations, built during the game, are of bounded size. 
But one can prove that \pe\ has a \ws\ 
in this hypebasis game in $n+1$ rounds, implying that $\R\notin \bold S\Ra\CA_{n+1}$, 
so that it does not even have an $n+1$--flat 
represenation. But it will have an $n$--flat one.

Now we give the idea in some (but not full) 
detail below.

One can construct a $\CA_l$ in a natural way from an $l$--dimensional cylindric basis \cite[Definition 12.7]{HHbook}. 
(The construction will be given in a moment.)
For an atomic  relation algebra $\R$ and $l>3$, recall that we denote by $\sf Mat_l(\At\R)$ the set of all $l$--dimensioana basic matrices on $\R$.
Sometimes ${\sf Mat}_l(\At\R)$ is an $l$--dimensional cylindric basis constituting the atom structure of a $\CA_l$, 
but sometimes it is not. The relation algebra dealt with in our next theorem \ref{thm:cmnr} (denoted by $\R$)
is of the first kind. 

For $\tau:l\to l$ we write $(f\tau)$ for the function defined by
$(f\tau)(x, y)=f(\tau(x), \tau(y)).$ It is always the case that $f\tau\in {\sf Mat}_l(\At\R)$ for any $f\in \Mat_l(\At\R)$ and any $\tau:l\to l$,
so if ${\sf Mat}_l(\At\R)$ is an $l$--dimensional cylindric basis, then 
$\Cm{\sf Mat}_l(\At\R)$ can be expanded to a ${\sf QEA}_l$,  
by defining for $X\subseteq {\sf Mat}_l(\At\R)$ and transposition $\tau:l\to l$: 
$$\s_\tau(X)=\set{f\in {\sf Mat}_l(\At\R): f\tau \in X}.$$

The third item in our coming theorem \ref{thm:cmnr},
which is (the main theorem) \cite [Theorem 1.1]{t} is {\it strictly weaker} than the result 
used in proof of theorem \ref{2.12}, namely  (using the notation {\it op. cit}), 
that $\Pi_{r/U}\C(m, n, r)\in \RQEA_m$ (upon replacing $\C(m, n, r)$ by $\D(m, n, r)$.)

\begin{theorem}\label{thm:cmnr} Let $3\leq m\leq n$ and $r<\omega$.
\begin{enumerate}
\renewcommand{\theenumi}{\Roman{enumi}}
\item $\D(m, n, r)\in {\sf Nr}_m\sf QEA_n$,\label{en:one}
\item $\Rd_{\Sc}\D(m, n, r)\not\in \bold S{\sf Nr}_m\Sc_{n+1}$, \label{en:two}
\item $\Pi_{r/U} \D(m, n, r)$ is elementarily equivalent to a polyadic equality algebra $\C\in{\sf Nr}_m\sf QEA_{n+1}$.  
\label{en:four}
\end{enumerate}
\end{theorem}
We define the algebras $\D(m,n,r)$ for $3\leq m\leq n<\omega$ and $r$
and then give a sketch of \eqref{en:two} given in detail in \cite[p. 211--215]{t}. 
We start with.
\begin{definition}\label{def:cmnr}
Define a function $\kappa:\omega\times\omega\rightarrow\omega$ by $\kappa(x, 0)=0$
(all $x<\omega$) and $\kappa(x, y+1)=1+x\times\kappa(x, y))$ (all $x, y<\omega$).
For $n, r<\omega$ let
\[\psi(n, r)=
\kappa((n-1)r, (n-1)r)+1.\]
This is to ensure that $\psi(n, r)$ is sufficiently big compared to $n, r$ for the proof 
of non-embeddability to work.
The second parameter $r<\omega$ may be considered as a finite linear order of length $r$.

For any  $n<\omega$ and any linear order $r$, let
\[{\sf Bin}(n, r)=\set{\Id}\cup\set{a^k(i, j):i< n-1,\;j\in r,\;k<\psi(n, r)}\]
where $\Id, a^k(i, j)$ are distinct objects indexed by $k, i, j$.
Let $3\leq m\leq n<\omega$ and let $r$ be any linear order.
Let $F(m, n, r)$ be the set of all  functions $f:m\times m\to {\sf Bin}(n, r)$
such that $f$ is symmetric ($f(x, y)=f(y, x)$ for all $x, y<m$)
and for all $x, y, z<m$ we have $f(x, x)=\Id,\;f(x, y)=f(y, x)$, and $(f(x, y), f(y, z), f(x, z))\not\in {\sf Forb}$,
where ${\sf Forb}$ (the \emph{forbidden} triples) is the following set of triples
 \[ \begin{array}{c}
 \set{(\Id, b, c):b\neq c\in {\sf Bin}(n, r)}\\
 \cup \\
 \set{(a^k(i, j), a^{k'}(i,j), a^{k^*}(i, j')): k, k', k^*< \psi(n, r), \;i<n-1, \; j'\leq j\in r}.
 \end{array}\]
Here ${\sf Bin}(n,r)$ is essentially an atom structure of a finite relation relation
and ${\sf Forb}$ specifies its operations by the standard procedure of specifying forbidden triples \cite{HHbook}.
This atom structure defines a relation algebra $\R$;
that is very similar (but not identical) to $\A(n,r)$ defined in \cite[Definition 15.2]{HHbook2}, and used in the first part of the 
proof of the forthcoming theorem \ref{2.12}.
The set of $m$--basic matrices $F(m,n,r)$ on $\R$ is the universe 
of the $\QEA_m$ atom structure $\Mat_m(\At\R)$.

$\D(m,n,r)$ is defined to be the complex algebra of the $m$--dimensional 
atom structure $\Mat_m(\At\R)$, that is, 
$\D(m, n, r)=\Cm{\sf Mat}_m(\At\R)$. 
The accessibility relations corresponding to substitutions, cylindrifiers are defined, the usual way 
on $m$--dimensional matrices.
The universe of $\D(m, n, r)$ is the power set of $F(m, n, r)$ and the operators (lifting from the atom structure)
are (we use the notation in definition \ref{b}):
\begin{itemize}
\item  the Boolean operators $+, -$ are union and set complement,
\item  the diagonal $\diag xy=\set{f\in F(m, n, r):f(x, y)=\Id}$,
\item  the cylindrifier $\cyl x(X)=\set{f\in F(m, n, r): \exists g\in X\; f\equiv_xg }$ and
\item the polyadic $\s_\tau(X)=\set{f\in F(m, n, r): f\tau \in X}$,
\end{itemize}
for $x, y<m,\;  X\subseteq F(m, n, r)$ and  transposition $\tau:m\to m$.
 \end{definition}
\medskip
Unlike the algebras $\C(m,n,r)$ in the proof of theorem \ref{2.12}, 
the algebras $\D(m, n, r)$ are now finite.
It is not hard to see
that  $3\leq m,\; 2\leq n$ and $r<\omega$
the algebra $\D(m, n, r)$ satisfies all of the axioms defining $\QEA_m$
except, perhaps, the commutativity of cylindrifiers $\cyl x\cyl y(X)=\cyl y\cyl x(X)$, which it satisfies because
$F(m,n,r)$ is a symmetric cylindric basis, so that overlapping
matrices amalgamate.
Furthermore, if  $3\leq m\leq m'$ then $\C(m, n, r)\cong\Nr_m\C(m', n, r)$
via $$X\mapsto \set{f\in F(m', n, r): f\restr{m\times m}\in X}.$$

We give a sketch of proof of \ref{thm:cmnr}(\ref{en:two}), which is the heart and soul of the proof, and it is quite similar
to its $\CA$ analogue \cite[pp. 4.69-475]{HHbook}. 
We will also refer to the latter when the proofs overlap, or are very similar.
Assume hoping for a contradiction  that
$\Rd_{\Sc}\D(m, n, r)\subseteq\Nr_m\C$
for some $\C\in \Sc_{n+1}$, some finite $m, n, r$.

Then it can be shown inductively
that there must be a `large set' $S$ of distinct elements of $\C$,
satisfying certain inductive assumptions, which we outline next.
For each $s\in S$ and $i, j<n+2$ there is an element $\alpha(s, i, j)\in {\sf Bin}(n, r)$ obtained from $s$
by cylindrifying all dimensions in $(n+1)\setminus\set{i, j}$, then using substitutions to replace $i, j$ by $0, 1$.
Then one shows that $(\alpha(s, i, j), \alpha(s, j, k), \alpha(s, i, k))\not\in {\sf Forb}$.
The induction hypothesis say, most importantly, that $\cyl n(s)$ is constant, for $s\in S$,
and for $l<n$  there are fixed $i<n-1,\; j<r$ such that for all $s\in S$ we have $\alpha(s, l, n)\leq a(i, j)$.
This defines, like in the proof of theorem 15.8 in \cite{HHbook2} p.471, two functions $I:n\rightarrow (n-1),\; J:n\rightarrow r$
such that $\alpha(s, l, n)\leq a(I(l), J(l))$ for all $s\in S$.  The \emph{rank} ${\sf rk}(I, J)$ of $(I, J)$ (as defined in definition 15.9 in \cite{HHbook2}) is
the sum (over $i<n-1$) of the maximum $j$ with $I(l)=i,\; J(l)=j$ (some $l<n$) or $-1$ if there is no such $j$.

Next it is proved that there is a set $S'$ with index functions $(I', J')$, still relatively large
(large in terms of the number of times we need to repeat the induction step)
where the same induction hypotheses hold but where ${\sf rk}(I', J')>{\sf rk}(I, J)$.  (Cf. \cite{HHbook}, where for $t<nr$, $S'$ was denoted by $S_t$
and proof of property (6) in the induction hypothesis  on \cite[p.474]{HHbook}.)

By repeating this enough times (more than $nr$ times) we obtain a non-empty set $T$
with index functions of rank strictly greater than $(n-1)\times(r-1)$, an impossibility. 
(Cf. \cite{HHbook}, where for $t<nr$, $S'$ was denoted by $S_t$.)

We sketch the induction step.  Since $I$ cannot be injective there must be distinct $l_1, l_2<n$
such that $I(l_1)=I(l_2)$ and $J(l_1)\leq J(l_2)$.  We may use $l_1$ as a "spare dimension"
(changing the index functions on $l$ will not reduce the rank).
 Since $\cyl n(s)$ is constant, we may fix $s_0\in S$
and choose a new element $s'$ below $\cyl l s_0\cdot \sub n l\cyl  l s$,
with certain properties.  Let $S^*=\set{s': s\in S\setminus\set{s_0}}$.
We wish to re-establish the induction hypotheses for $S^*$, and many of these are simple to check.
Although suitable functions $I', J'$ may not exist on the whole of $S$, but $S$ remains
large enough to enable selecting a
subset $S'$ of $S^*$, still large in terms of the number of remaining times the induction step must be applied.
The required functions $I', J'$ now exist (for all but one value of $l<n$ the values $I'(l), J'(l)$ are determined by $I, J$,
for  one value of $l$ there are at most $(n-1)r$ possible values, hence on a large subset the choices agree).
Next it can be shown that $J'(l)\geq J(l)$ for all $l<n$.   Since
$$(\alpha(s, i, j), \alpha(s, j, k), \alpha(s, i, k))\not\in Forb$$
and by the definition of ${\sf Forb}$
either $\rng(I')$ properly extends $\rng(I)$ or there is $l<n$ such that $J'(l)>J(l)$, hence  ${\sf rk}(I', J')>{\sf rk} (I, J)$.

It is proved in \cite{t} that for any class $\K$ between $\sf Sc$ (Pinter's substitution algebras) 
and $\QEA$, for any positive $k$,  and for any ordinal $\alpha>2$, 
the class $\bold S{\sf Nr}_\alpha\K_{\alpha+k+1}$ is not axiomatizable by a finite schema over 
$\bold S{\sf Nr}_\alpha\K_{\alpha+k}$.  We strengthen this result when $\alpha\geq \omega$ in the availability of the 
diagonal
constants, namely, for any class $\K$ between $\CA$ and $\QEA$ using the same {\it lifting technique} in \cite{t}.  
Before embarking on the theorem and its proof, we first need to recall some notation from \cite{HMT2}. 

Let $\K\in \{\CA, \QEA\}$.
\begin{enumerate}
\item Let $\alpha$ and $\beta$ be ordinals and let $\rho:\alpha\rightarrow \beta$ be an injection.
For any $\B\in \K_\beta$, $\Rd^\rho\B$ as defined in \cite[Definition 2.6.1]{HMT2} (extrapolated the obvious way to $\QEA$s) 
is in $\K_\alpha$. The universe of $\Rd^{\rho}\B$ is the same as the universe of $\B$ and 
the indices in the operations are re-shuffled along $\rho$, so for example for $b\in B$ and 
$i<\alpha$, ${\sf c}_i^{\Rd^{\rho}\B}b={\sf c}_{\rho(i)}^{\B}b$. If $\alpha\subseteq \beta$ and $\rho(i)=i$ for all $i\in \alpha$, then $\Rd^{\rho}\B$ is the 
standard (in the universal algebraic sense) reduct, denoted 
above by $\Rd_{\alpha}\B(\in \K_{\alpha}$).  

\item For an ordinal $\alpha$ and $\B\in\K_\alpha$ and $x\in \B$, $\Rl_x\B$ obtained by {\it relativizing} 
$x$ has elements of $\B$ below $x$ \cite[Definition 2.2.1]{HMT2} and extra non--Boolean operations are intersected with $x$.
It is not always the case that $\Rl_x\B$ is a $\K_{\alpha}$ (we can lose commutativity of cylindrifiers). 
However, if $x$ is {\it  rectangular}, in the sense that all $i<j<\alpha$,  ${\sf c}_ix\cdot {\sf c}_jx=x$, then $\Rl_x\B\in \K_\alpha$ \cite[Theorem 2.2.10]{HMT2}. 
This is used the following proof. 

\item  For an ordinal $\alpha$ and $\Gamma\subseteq \alpha$ we write $\Nr_\Gamma\A$ for the algebra whose universe 
is $\{a\in\A: i \in \alpha\setminus\Gamma\rightarrow \c_i a=a\}$ endowed with  the operators of $\A$ whose indicies 
are contained in $\Gamma$.  
\end{enumerate}

The following result is stronger than that obtained in \cite[Theorem 3.1]{t} when in the last referred to theorem, we consider only $\CA$s and $\QEA$s. 
The structure of the proof though is essentially the same, lifiting the analogous result from `finite dimensions' to the transfinite. 
We use different finite dimensional  algebras. Recall that the neat embedding theorem for $\QEA_{\alpha}$ says
that  $\sf RQEA_{\alpha}=\bold S\Nr_{\alpha}\QEA_{\alpha+\omega}.$

\begin{theorem} \label{2.12} Let $\alpha$ be any ordinal $>2$ possibly infinite.  Then for any $r\in \omega$, and finite 
$k\geq 1$, there exists $\A_r\in \bold S{\sf Nr}_{\alpha}\QEA_{\alpha+k}$
such that $\Rd_{ca}\A_r\notin \bold S{\sf Nr}_{\alpha}\CA_{\alpha+k+1}$
and $\Pi_{r/U}\A_r\in \sf RQEA_{\alpha}$ for any non--principal ultrafilter
$U$ on $\omega$. Thus for any variety $\K$ whose signature is between the signatures of $\CA$ and $\QEA$ and for any 
positive $k\geq 1$, the variety $\bold S\Nr_{\alpha}\K_{\alpha+k+1}$ is not axiomatizable by a finite schema 
over $\bold S\Nr_{\alpha}\K_{\alpha+k}$ and ${\sf RK}_{\alpha}$ is not axiomatizable by a finite schema over 
$\bold S\Nr_{\alpha}\K_{\alpha+m}$  for all $m>0$.
\end{theorem}
\begin{proof} 

Though the idea used here is the same idea used in \cite[Theorem 3.1]{t},
the result that we lift for finite dimensions to the transfinite, is stronger than
that obtained for finite dimensions in \cite[Theorem 3.1]{t}. Consequently, 
the  result hitherto obtained for infinite dimensions is  also stronger than the result 
obtained in \cite[Theoren 3.1]{t} (when restricted to any
$\K$ between $\CA$ and $\QEA$).

Fix $2<m<n<\omega$. Let $\mathfrak{C}(m,n,r)$ be the algebra $\Ca(\bold H)$ where $\bold H=H_m^{n+1}(\A(n,r), \omega)),$
is the $\CA_m$ atom structure consisting of all $n+1$--wide $m$--dimensional
wide $\omega$ hypernetworks \cite[Definition 12.21]{HHbook}
on $\A(n,r)$  as defined in \cite[Definition 15.2]{HHbook}.   Then $\mathfrak{C}(m, n, r)\in \CA_m$, and it can be easily expanded
to a $\QEA_m$, since $\C(m, n, r)$ is `symmetric', in the sense  that it allows a polyadic equality 
expansion by defining substitution operations corresponding to transpositions.

This follows by observing that $\bold H$ is obviously symmetric in the following exact sense:  
For $\theta:m\to m$ and $N\in \bold H$, $N\theta\in \bold H,$
where $N\theta$ is defined by $(N\theta)(x, y)=N(\theta(x), \theta(y)).$ 
 Hence, the binary polyadic operations defined on the atom structure 
$\bold H$ the obvious way (by swapping co--ordinates) 
lifts to polyadic operations of its complex algebra $\mathfrak{C}(m, n, r)$. In more detail, for a transposition 
$\tau:m\to m$,  and $X\subseteq \bold H$,  define $\s_\tau(X)=\{N\in \bold H: N\tau \in X\}$.

Furthermore, for any $r\in \omega$ and $3\leq m\leq n<\omega$, $\C(m,n,r)\in {\sf Nr}_m{\sf QEA}_n$, $\Rd_{ca}\C(m,n,r)\notin {\bold  S}{\sf Nr}_m{\sf CA_{n+1}}$
and $\Pi_{r/U}\C(m,n,r)\in {\sf RQEA}_m$ by easily
adapting \cite[Corollaries 15.7, 5.10, Exercise 2, pp. 484, Remark 15.13]{HHbook}
to the $\QEA$ context.

Take
$$x_n=\{f\in H_n^{n+k+1}(\A(n,r), \omega); m\leq j<n\to \exists i<m, f(i,j)=\Id\}.$$
Then $x_n\in \C(n,n+k,r)$ and ${\sf c}_ix_n\cdot {\sf c}_jx_n=x_n$ for distinct $i, j<m$.
Furthermore (*),
$I_n:\C(m,m+k,r)\cong \Rl_{x_n}\Rd_m {\C}(n,n+k, r)$
via the map, defined for $S\subseteq H_m^{m+k+1}(\A(m+k,r), \omega)),$ by
$$I_n(S)=\{f\in H_n^{n+k+1}(\A(n,r), \omega):  f\upharpoonright {}^{\leq m+k+1}m\in S,$$
$$\forall j(m\leq j<n\to  \exists i<m,  f(i,j)=\Id)\}.$$
We have proved the (known) result for finite ordinals $>2$.

To lift the result to the transfinite,
we proceed like in \cite{t}, using the same lifting argument in {\it op.cit}.

Let $\alpha$ be an infinite ordinal. Let $I=\{\Gamma: \Gamma\subseteq \alpha,  |\Gamma|<\omega\}$.
For each $\Gamma\in I$, let $M_{\Gamma}=\{\Delta\in I: \Gamma\subseteq \Delta\}$,
and let $F$ be an ultrafilter on $I$ such that $\forall\Gamma\in I,\; M_{\Gamma}\in F$.
For each $\Gamma\in I$, let $\rho_{\Gamma}$
be an injective function from $|\Gamma|$ onto $\Gamma.$
Let ${\C}_{\Gamma}^r$ be an algebra similar to $\QEA_{\alpha}$ such that
$\Rd^{\rho_\Gamma}{\C}_{\Gamma}^r={\C}(|\Gamma|, |\Gamma|+k,r)$
and let
$\B^r=\Pi_{\Gamma/F\in I}\C_{\Gamma}^r.$
Then we have $\B^r\in \bold {\sf Nr}_\alpha\QEA_{\alpha+k}$ and
$\Rd_{ca}\B^r\not\in \bold S{\sf Nr}_\alpha\CA_{\alpha+k+1}$.
These can be proved exactly like the proof of the first two items in \cite[Theorem 3.1]{t}. The second part uses that the element $x_n$ is {\it $m$--rectangular} 
and {\it $m$--symmetric}, in the sense that 
for all  $i\neq j\in m$, ${\sf c}_ix_n\cdot {\sf c}_jx_n=x_n$ and ${\sf s}_i^jx_nx \cdot  {\sf s}_j^ix_n =x_n$ (This last two conditions are not entirely indepedent \cite{HMT2}). 
This is crucial to guarantee that in the algebra obtained after relativizing 
to  $x_n$, we do not lose commutativity of cylindrifiers.  The relativized algebra stays inside $\QEA_n$.

We know
from the finite dimensional case that $\Pi_{r/U}\Rd^{\rho_\Gamma}\C^r_\Gamma=\Pi_{r/U}\C(|\Gamma|, |\Gamma|+k, r) \subseteq \Nr_{|\Gamma|}\A_\Gamma$,
for some $\A_\Gamma\in\QEA_{|\Gamma|+\omega}=\QEA_{\omega}$.
Let $\lambda_\Gamma:\omega\rightarrow\alpha+\omega$
extend $\rho_\Gamma:|\Gamma|\rightarrow \Gamma \; (\subseteq\alpha)$ and satisfy
$\lambda_\Gamma(|\Gamma|+i)=\alpha+i$
for $i<\omega$.  Let $\F_\Gamma$ be a $\QEA_{\alpha+\omega}$ type algebra such that $\Rd^{\lambda_\Gamma}\F_\Gamma=\A_\Gamma$.
Then $\Pi_{\Gamma/F}\F_\Gamma\in\QEA_{\alpha+\omega}$, and we have proceeding like in the proof of item 3 in \cite[Theorem 3.1]{t}:

$\Pi_{r/U}\B^r=\Pi_{r/U}\Pi_{\Gamma/F}\C^r_\Gamma
\cong \Pi_{\Gamma/F}\Pi_{r/U}\C^r_\Gamma
\subseteq \Pi_{\Gamma/F}\Nr_{|\Gamma|}\A_\Gamma
=\Pi_{\Gamma/F}\Nr_{|\Gamma|}\Rd^{\lambda_\Gamma}\F_\Gamma
=\Nr_\alpha\Pi_{\Gamma/F}\F_\Gamma.$

But $\B=\Pi_{r/U}\B^r\in \bold S{\sf Nr}_{\alpha}\QEA_{\alpha+\omega}$
because $\F=\Pi_{\Gamma/F}\F_{\Gamma}\in \QEA_{\alpha+\omega}$
and $\B\subseteq \Nr_{\alpha}\F$, hence it is representable (here we use the neat embedding theorem).
The rest follows using a standard L\'os argument. 
\end{proof}

In the first part of the proof of theorem \ref{2.12}, we had $\Pi_{r/U}\C(m, n, r)\in \RQEA_m$.
As stated in the last item of teorem \ref{thm:cmnr}, we do not guarantee that the ultraproduct on $r$
of the $\D(m,n,r)$s ($2<m<n<\omega)$ is representable.
A standard L\"os argument shows that
$\Pi_{r/U}\C(m, n, r) \cong\C(m, n, \Pi_{r/U} r)$ and $\Pi_{r/U}r$
contains an infinite ascending sequence.
Here  one  extends the definition of $\psi$
by letting $\psi(n, r)=\omega,$ for any infinite linear order $r$.

The infinite algebra
$\D(m,n, J)\in {\bf El}{\sf Nr}_n\QEA_{n+1}$
when $J$ is the infinite linear order as above.
Since $\Pi_{r/U} r$ is such, then we get $\Pi_{r/U}\D(m, n, r)\in 
{\bf El}{\sf Nr}_m\QEA_{n+1}(\subseteq \bold S{\sf Nr}_m\QEA_{n+1}$), cf. \cite[pp.216-217]{t}.

This suffices to show that for any $\K$ having signature between $\Sc$ and $\QEA$, for any $2<m<\omega$, and for any positive $k$,
the variety $\bold S{\sf Nr}_m\K_{m+k+1}$ is not finitely axiomatizable over the variety  ${\bold S}{\sf Nr}_m{\sf K}_{m+k}$.
The result for $\K=\CA$ was obtained by Hirsch and Hodkinson, solving \cite [Problem 2.12]{HMT1} posed by Monk, 
which proved to be one of the most, if not {\it the most}, influential 
problem in the history of Tarskian algebraic 
logic.

\section{Rainbows versus splitting}

Splitting does not work for diagonal free algebras. 
Using a rianbow construction we now show that:

\begin{theorem}\label{main2} Let $2<n<\omega$. Then any variety between ${\sf RDf}_n$ and $\RQEA_n$ has no equational axiomatization 
using finitely many variables.
\end{theorem} 

By a class  $\K$ {\it between ${\sf RDf}_n$ and ${\sf RQEA}_n$} we mean that the signature of $t$ of $\K$ 
is between the signatures of ${\sf Df}_n$ and ${\sf QEA}_n$
and if $\A\in \K$ then their exists $\B\in {\sf RQEA}_n$ such that $\A\subseteq \Rd_t\B$, 
where $\Rd_t$ is the operation of discarding all the quasi--polyadic equality operations not in $t$. 
Furthermore, for all $\C\in \RQEA_n$, $\Rd_t\C\in \K$.

The result for any variety between $\RCA_n$ and $\RQEA_n$ (defined as above) is known \cite{Andreka}, though here we give an entirely different proof.
The result for diagonal free algebras is new, and it answers a question of Sain and Thompson 
formulated back in $1990$ \cite{ST}. Sain and Thompson proved only 
non--finite axiomatizability, which is a weaker result. In principal there can be 
infinite axiomatizations using only finitely many variables.

Throughout this subsection $n$ is fixed to be a finite ordinal $>2$. 
We deal with finite rainbow algebras of dimension $n$. Let $\kappa, \mu$ be finite ordinals $>0$.
The colours we use are:
\begin{itemize}

\item greens: $\g_i$ ($1\leq i\leq n-2)\cup \{\g^0_i: i\in \mu\},$
\item whites : $\w_i, i <n-1,$
\item reds:  $\r_{i}$, $i\in \kappa,$

\item shades of yellow : $\y_S: S\subseteq \mu.$
\end{itemize}

And coloured graphs are:

\begin{enumarab}

\item $M$ is a complete graph.

\item $M$ contains no triangles (called forbidden triples)
of the following types:

\vspace{-.2in}
\begin{eqnarray*}
&&\nonumber\\
(\g, \g^{'}, \g^{*}), (\g_i, \g_{i}, \w_i)
&&\mbox{any }1\leq i\leq n-1\  \\
(\g^j_0, \g^k_0, \w_0)&&\mbox{ any } j, k\in \mu\\
(\r_{i}, \r_{i}, \r_{j}) && i,j\in \kappa. \\
\end{eqnarray*}
and no other triple of atoms is forbidden. The part dealing with shades of yellos and cones is eactly like \cite{HH} (this does not involve any reds).
\end{enumarab}
We denote the complex rainbow polyadic equality (finite) algebra whose atom structure consists of 
(equivalence classes of ) surjections $a:n\to \Gamma$, $\Gamma$ a coloured graph
labelled by the above  specified  colours, by $\A_{\mu,\kappa}$.

Fix finite $m>1$. Let $\lambda=(n\times 2^m)^3$ and $\beta=[(\lambda+1)\times (\lambda+2)]/2.$
Let $\A=\A_{\lambda+2,\beta}$ and $\B=\A_{\lambda+2,\lambda}$. Recall that the dimension is fixed to be $n$ with $2<n<\omega$, 
so that $\A, \B\in {\sf QEA}_n$. The construction of $\A$ and $\B$ of course depends on $m$.
For rainbow algebras we use the graph version of the game $G_{\omega}$ 
played on coloured graphs \cite{HH}.

\begin{theorem}
\begin{enumarab}
\item \pa\ has a \ws\  in $G_{\lambda+2}(\At\B)$. 
\item \pe\ has a \ws\ for $G_{\omega}(\At\A)$.  
\end{enumarab}
\end{theorem}
\begin{proof}
We first show that \pa\ has a \ws\  in $G(\At\B)$  in $\lambda+2$ many rounds,  hence $\B\notin {\sf RQEA}_n$.
Since $\B$ is generated by the set $\{b\in \B: \Delta b\neq n\}$, then by \cite[Theorem 5.4.26]{HMT2}, 
it will follow that 
the diagonal free reduct $\Rd_{df}\B$ is not in ${\sf RDf}_n$.
Now \pa\ plays as follows: In his zeroth move, \pa\ plays a graph $\Gamma$ with
nodes $0, 1,\ldots, n-1$ and such that $\Gamma(i, j) = \w_0 (i < j <
n-1), \Gamma(i, n-1) = \g_i ( i = 1,\ldots, n-2), \Gamma(0, n-1) =
\g^0_0$, and $ \Gamma(0, 1,\ldots, n-2) = \y_{\lambda+2}$. This is a $0$-cone
with base $\{0,\ldots, n-2\}$. In the following moves, \pa\
repeatedly chooses the face $(0, 1,\ldots, n-2)$ and demands a node
$t$ with $\Phi(i,t) = \g_i$, $(i=1,\ldots, n-2)$ and $\Phi(0, t) = \g^t_0$,
in the graph notation -- i.e., an $t$-cone, $t\leq \lambda+2$,  on the same base.
\pe\ among other things, has to colour all the edges
connecting $\lambda+2$ nodes $n_0, n_1, \ldots n_{\lambda+1}$ created by \pa\ as apexes of cones based on the face $(0,1,\ldots, n-2)$ by red labels. 
But there are only $\lambda$ red labels, so there must be $0<i\neq j<\lambda+2$, such that in the last coloured graph 
$\Lambda$,  $\Lambda(n_0, n_i)=\Lambda(n_0, n_j)$. But
$(n_0, n_i, n_j)$ is inconsistent, so \pa\ wins.
The conclusion now follows.

But we claim that \pe\ has a \ws\ in $G_{\omega}(\At\A)$, hence $\A\in {\sf RQEA}_n.$
If \pa\ plays like before, now that \pe\ has more red labels, \pa\ cannot force a win. In fact, \pa\
can only force a red clique (a coloured graph in which every edge has a red label), of size of at most $\lambda+2$ indexed by 
$\{\g_0^i: i<\lambda+2\}$ not bigger. So \pe\ s
strategy within such red cliques is to choose a label for each edge using
a red colour and to ensure that each edge within the clique has a label unique to this edge.
Since there are $[(\lambda+1)\times (\lambda+2)]/2$ many red labels to choose from,  she can do that.
\end{proof}
We have shown that $\A\in {\RQEA}_n$ and $\Rd_{df}\B\notin {\sf RDf}_n$. 
Next we show that $m$--variable equations cannot witness this dichotomy.  Call a coloured graph {\it red}, if at least one of its edges are labelled red.
We write $\r$ for $a:n\to \Gamma$, where $\Gamma$ is a red graph, and we call it a red atom
Here we identify an atom with its representative, but no harm will ensue.
For a class $\K$ between ${\sf RDf}_n$ and $\RQEA_n$ and $\C\in {\sf RQEA}_n$, we write $\Rd_{\K}\C(\in \K)$ for the algebra obtained from $\C$ by discarding the operations 
in the signature of ${\sf QEA}_n$ that are outside the signature of  
$\K$.

\begin{theorem} Assume that $2<n<\omega$, $1<m<\omega$, and that $\sf K$ is 
any variety between ${\sf RDf}_n$ and ${\sf RQEA}_n$.  
Then for any $m$ variable equation $e$ in the signature of $\sf K$, $\Rd_{\sf K}\A\models e\iff \Rd_{\sf K}\B\models e$.
\end{theorem}
\begin{proof}
Fix $m>1$. The proof is similar to \cite[Lemma 17.10]{HHbook}. We consider the $\K$ reducts of $\A$ and $\B$, which we continue to denote, with a slight abuse of notation, 
by $\A$ and $\B$.
Let $\R$ be the set of red atoms of $\A$, and $\R'$ be the set of red atoms in $\B$. Then $|\R|>|\R'|$ and $|\R'|\geq n\times 2^m$. 
To see why, take distinct $\r, \r', \r''\in [(\lambda+2)\times(\lambda+1)]/2$ such that $\r\notin \lambda$. 
Let $\Gamma$ be the coloured graph containing the triangle $(\r, \r', \r'')$ with other edges labelled by $\w_0$.
This avoids inconsistent triples. One labels 
the $n-1$ tuples by $\y_{\lambda+2}$. This maintains the consistency condition for labelling $n-1$ tuples by shades of yellow.  Then $\Gamma\in \R\sim \R'$. 
Every three element subset  of $\lambda= (n\times 2^m)^3$ gives at least one red $n$--coloured graph, such that for two distinct such sets, 
the resulting $n$-coloured graphs are distinct.   
As before, given such subset, one completes labelling of the edges by $\w_0$ ane labels the $n-1$ tuples by $\y_{\lambda+2}$.  
So $\R'$ has at least $r$ elements, where $r$ is the number of three element subsets of
$(n\times 2^m)^3$, so  $r\geq n\times 2^m$, hence $|\R'|\geq n\times 2^m$. 

We show that if $e$ is an $m$ variable equation that does not hold in $\A$, then it does not hold in $\B$. 
The converse is entirely analogous. 
Assume that the equation $s=t$  in the signature of $\K$ using $m$ variables does not hold in $\A$.
Then there is an assignment $h:\{x_0, \ldots, x_{m-1}\}\to \A$, such that $\A, h\models s\neq t$.
We construct an assignment $h'$ from $\{x_0, \ldots, x_{m-1}\}$ into $\B$ that also falsifies $s=t$.

Now $\A$ has more red atoms, but $\A$ and $\B$ have identical non-red atoms. So for any non red atom $a$ of $\B$, and for any
$i<m$, let $a\leq h'(x_i)$ $\iff$ $a\leq h(x_i)$.
To complete the definition of $h'$ it remains to identify which red--atoms of $\B$ (which are only a part of those of $\A$)
lie below $h'(x_i)$ ($i<m$). The element $h'(x_i)\in \B$ will be the (finite) sum of these atoms. 
For $S\subseteq m=\{0,1\ldots, m-1\}$, let
$$\R_S=\{\r\in \R: \r\leq h(x_i) , i\in S, \r\cdot h(x_i)=0, i\in m\sim S\}.$$
Then $\mathfrak{P}=\{\R_S: S\subseteq m\}$ forms a partition of $\R$ with possibly some empty blocks. 
If $\r\in \R$, one takes $S=\{i\in m: \r\leq h(x_i)\}$, then $\r\in \R_S$. If $\r\cdot h(x_i)=0$ for all $i<m$, 
then $S=\emptyset$ and $\r\in \R_{\emptyset}=\{r\in \R: \r\cdot h(x_i)=0, i\in m\}$.
Assume that $S_1, S_2\subseteq m$, and that $j\in S_1\sim S_2$. Let $\r\in \R_{S_2}$ then 
$\r\leq h(x_i)$ for all $i\in S_2$ and $\r\cdot h(x_i)=0$ for all $i\in m\sim S_2$. 
So $\r\cdot  h(x_j)=0$, thus $\r\notin \R_{S_1}$
W have shown that $\mathfrak{P}$ is a partition.

Next, partition  $\R'$ into $2^m$ parts $\R'_S$ for $S\subseteq m$ 
such that $|\R'_S|=|\R_S|$ if $|\R_S|<n$, and $|\R_S'|\geq n$ if $|\R_S|\geq n$.
This is possible because $|\R'|, |\R|\geq n\times 2^m$. 
Now for each $i<m$ and each red atom $\r'$ in $\R'$, we complete the definition of
$h'(x_i)\in \B$ by specifying the red atoms below it. Recall that 
$\R'=\bigcup_{S\subseteq m}\R'_S$. We set $\r'\leq h'(x_i)$ $\iff$ $\r'\in \R'_S$ for some $S$ such that $i\in S$.
Fix $S\subseteq m$ and $i<m$. 
Then by definition $\R_S\subseteq h(x_i)\iff  \R'_S\subseteq h'(x_i).$
The non red atoms below $h(x_i)$ are the same as the non red atoms below
$h(x_i)$, hence $h(x_i)\sim \R=h'(x_i)\sim \R'$. Assume that $|\R\cap h(x_i)|<n$.  Since $\R_S\cap h(x_i)=\emptyset$ if $i\notin S$, and ${\sf R}_S\subseteq h(x_i)$ 
if $i\in S$,  then $|\R\cap h(x_i)|=|\bigcup_{S\subseteq m} (\R_S\cap h(x_i))|= \sum_{S\subseteq m, i\in S}|\R_S\cap h(x_i)|=\sum_{S\subseteq m, i\in S}|\R_S|=
\sum _{S\subseteq m, i\in S}|\R'_S|=|\R'\cap h'(x_i)|$.
So the number of red atoms below $h(x_i)$ is the same as the number of red atoms below $h'(x_i)$ if this number is $<n$.
Now assume that $|\R\cap h(x_i)|\geq n$. Then using the same reasoning as above
we have $\sum_{S\subseteq m, i\in S}|\R_S|=\sum _{S\subseteq m, i\in S}|\R'_S|=|\R'\cap h'(x_i)|\geq n$.
So $|\R\cap h(x_i)| \geq n\iff |\R'\cap h'(x_i)|\geq n$.

It  can be shown inductively 
that for any term $\tau$ using only the first $m$ variables
and $S\subseteq m$, that: 
$\R_S\subseteq h(\tau)\iff  \R'_S\subseteq h'(\tau),$
$h(\tau)\sim \R = h'(\tau)\sim \R',$
$|\R\cap h(\tau)| =|\R'\cap h'(\tau)|\iff |\R\cap h(\tau)|<n,$
and $|\R\cap h(\tau)| \geq n\iff |\R\cap h'(\tau)|\geq n.$
Thus $h'$ falsifies $s=t$ in $\B$. .
\end{proof}

Now we are ready to prove theorem \ref{main2}.  

{\bf Proof of the theorem:} Let $\K$ be a variety 
as specified in the hypothesis.
If $\Sigma$ is any $m$ variable equational theory then the $\K$  reduct of $\A$ 
and $\B$ either both validate
$\Sigma$ or neither do. Since one algebra is in $\K$ while the other is not, it follows that $\Sigma$ does not axiomatize 
$\K$. We have proved the required.

Observe that if $\A$ and $\B$ are simple, then one can prove that $\K$ has no finite variable universal prenex axiomatization.   
To see why, the $\K$ reducts 
of $\A$ and $\B$ are subdirectly irreducible and in  a discriminator variety every universal prenex formula 
is equivalent in subdirectly irreducible members to an
equation using the same number of variables. 
By the previous corollary we are done.

\section{Appendix} 

{\bf Proof of (technical) lemma \ref{Thm:n}:}
\begin{proof}
First a piece of notation. Let $m$ be a finite ordinal $>0$. An $\sf s$ word is a finite string of substitutions $({\sf s}_i^j)$ $(i, j<m)$,
a $\sf c$ word is a finite string of cylindrifications $({\sf c}_i), i<m$;
an $\sf sc$ word $w$, is a finite string of both, namely, of substitutions and cylindrifications.
An $\sf sc$ word
induces a partial map $\hat{w}:m\to m$:
\begin{itemize}

\item $\hat{\epsilon}=Id,$

\item $\widehat{w_j^i}=\hat{w}\circ [i|j],$

\item $\widehat{w{\sf c}_i}= \hat{w}\upharpoonright(m\smallsetminus \{i\}).$

\end{itemize}
If $\bar a\in {}^{<m-1}m$, we write ${\sf s}_{\bar a}$, or
${\sf s}_{a_0\ldots a_{k-1}}$, where $k=|\bar a|$,
for an  arbitrary chosen $\sf sc$ word $w$
such that $\hat{w}=\bar a.$
Such a $w$  exists by \cite[Definition~5.23 ~Lemma 13.29]{HHbook}.

Fix $2<n<m$. Assume that $\C\in\Sc_m$, $\A\subseteq_c{\mathfrak Nr}_n\C$ is an
atomic $\Sc_n$ and $N$ is an $\A$--network with $\nodes(N)\subseteq m$. Define
$N^+\in\C$ by
\[N^+ =
 \prod_{i_0,\ldots, i_{n-1}\in\nodes(N)}{\sf s}_{i_0, \ldots, i_{n-1}}{}N(i_0,\ldots, i_{n-1})\]
with the substitution operator defined as above.
For a network $N$ and  function $\theta$,  the network
$N\theta$ is the complete labelled graph with nodes
$\theta^{-1}(\nodes(N))=\set{x\in\dom(\theta):\theta(x)\in\nodes(N)}$,
and labelling defined by
$$(N\theta)(i_0,\ldots, i_{n-1}) = N(\theta(i_0), \theta(i_1), \ldots,  \theta(i_{n-1})),$$
for $i_0, \ldots, i_{n-1}\in\theta^{-1}(\nodes(N))$.
We start with a sketch of proof of the first item which  uses ideas
in \cite[Lemma 29, 26, 27]{r} formulated for relation algebras.
We assume that the ${\sf s}_i^j$s for  $i< j<m$ are completely additive in $\C$. (This condition is superfluous for any $\K$ between $\CA$ and $\QEA$.)
Then the following hold:
\begin{enumerate}
\item for all $x\in\C\setminus\set0$ and all $i_0, \ldots, i_{n-1} < m$, there is $a\in\At\A$, such that
${\sf s}_{i_0,\ldots, i_{n-1}}a\;.\; x\neq 0$,

\item for any $x\in\C\setminus\set0$ and any
finite set $I\subseteq m$, there is a network $N$ such that
$\nodes(N)=I$ and $x\cdot N^+\neq 0$. Furthermore, for any networks $M, N$ if
$M^+\cdot N^+\neq 0$, then
$M\restr {\nodes(M)\cap\nodes(N)}=N\restr {\nodes(M)\cap\nodes(N)},$

\item if $\theta$ is any partial, finite map $m\to m$
and if $\nodes(N)$ is a proper subset of $m$,
then $N^+\neq 0\rightarrow {(N\theta)^+}\neq 0$. If $i\not\in\nodes(N),$ then ${\sf c}_iN^+=N^+$.

\end{enumerate}
We using the above facts  to show that \pe\  has a \ws\ in $F^m$ then we prove them. She can always
play a network $N$ with $\nodes(N)\subseteq m,$ such that
$N^+\neq 0$.
In the initial round, let \pa\ play $a\in \At\A$.
\pe\ plays a network $N$ with $N(0, \ldots, n-1)=a$. Then $N^+=a\neq 0$.
Recall that here \pa\ is offered only one (cylindrifier) move.
At a later stage, suppose \pa\ plays the cylindrifier move, which we denote by
$(N, \langle f_0, \ldots, f_{n-2}\rangle, k, b, l).$
He picks a previously played network $N$,  $f_i\in \nodes(N), \;l<n,  k\notin \{f_i: i<n-2\}$,
such that $b\leq {\sf c}_l N(f_0,\ldots,  f_{i-1}, x, f_{i+1}, \ldots, f_{n-2})$ and $N^+\neq 0$.
Let $\bar a=\langle f_0\ldots f_{i-1}, k, f_{i+1}, \ldots f_{n-2}\rangle.$
Then by  second part of  (3)  we have that ${\sf c}_lN^+\cdot {\sf s}_{\bar a}b\neq 0$
and so  by first part of (2), there is a network  $M$ such that
$M^+\cdot{\sf c}_{l}N^+\cdot {\sf s}_{\bar a}b\neq 0$.
Hence $M(f_0,\dots, f_{i-1}, k, f_{i-2}, \ldots$ $, f_{n-2})=b$,
$\nodes(M)=\nodes(N)\cup\set k$, and $M^+\neq 0$, so this property is maintained.
In \cite{AGMNS}  it is proved that for $\C\in \QA_m$, 
${\sf s}_0^1{}^{\C}$ is completely additive $\iff \C$ is completely additive. 
By complete additivity for $\CA$s and $\QEA$s, we get the second and 
third required and we are done.

Now we prove the facts. Since $\A\subseteq _c\Nr_n \C$, then $\sum^{\C}\At\A=1$. For (1), we have, by assumption,  ${\sf s}^i_j$ is a
completely additive operator (any $i, j<m$), hence ${\sf s}_{i_0,\ldots, i_{n-1}}$
is, too.
So $\sum^{\C}\set{{\sf s}_{i_0\ldots, i_{n-1}}a:a\in\At(\A)}={\sf s}_{i_0\ldots i_{n-1}}
\sum^{\C}\At\A={\sf s}_{i_0\ldots, i_{n-1}}1=1$ for any $i_0,\ldots, i_{n-1}<m$.  Let $x\in\C\setminus\set0$.  Assume for contradiction
that  ${\sf s}_{i_0\ldots, i_{n-1}}a\cdot x=0$ for all $a\in\At\A$. Then  $1-x$ will be
an upper bound for $\set{{\sf s}_{i_0\ldots i_{n-1}}a: a\in\At\A}.$
But this is impossible
because $\sum^{\C}\set{{\sf s}_{i_0\ldots, i_{n-1}}a :a\in\At\A}=1.$
To prove the first part of (2), we repeatedly use (1).
We define the edge labelling of $N$ one edge
at a time. Initially, no hyperedges are labelled.  Suppose
$E\subseteq\nodes(N)\times\nodes(N)\ldots  \times\nodes(N)$ is the set of labelled hyperedges of $
N$ (initially $E=\emptyset$) and
$x\;.\;\prod_{\bar c \in E}{\sf s}_{\bar c}N(\bar c)\neq 0$.  Pick $\bar d$ such that $\bar d\not\in E$.
Then by (1) there is $a\in\At(\A)$ such that
$x\;.\;\prod_{\bar c\in E}{\sf s}_{\bar c}N(\bar c)\;.\;{\sf s}_{\bar d}a\neq 0$.
Include the hyperedge $\bar d$ in $E$.  We keep on doing this until eventually  all hyperedges will be
labelled, so we obtain a completely labelled graph $N$ with $N^+\neq 0$.
it is easily checked that $N$ is a network.

For the second part of $(2)$, we proceed contrapositively. Assume that there is
$\bar c \in{}\nodes(M)\cap\nodes(N)$ such that $M(\bar c )\neq N(\bar c)$.
Since edges are labelled by atoms, we have $M(\bar c)\cdot N(\bar c)=0,$
so
$0={\sf s}_{\bar c}0={\sf s}_{\bar c}M(\bar c)\;.\; {\sf s}_{\bar c}N(\bar c)\geq M^+\cdot N^+$.
A piece of notation. For $i<m$, let $Id_{-i}$ be the partial map $\{(k,k): k\in m\smallsetminus\{i\}\}.$
For the first part of (3)
(cf. \cite[lemma~13.29]{HHbook} using the notation in {\it op.cit}), since there is
$k\in m\setminus\nodes(N)$, \/ $\theta$ can be
expressed as a product $\sigma_0\sigma_1\ldots\sigma_t$ of maps such
that, for $s\leq t$, we have either $\sigma_s=Id_{-i}$ for some $i<m$
or $\sigma_s=[i/j]$ for some $i, j<m$ and where
$i\not\in\nodes(N\sigma_0\ldots\sigma_{s-1})$.
But clearly  $(N Id_{-j})^+\geq N^+$ and if $i\not\in\nodes(N)$ and $j\in\nodes(N)$, then
$N^+\neq 0 \rightarrow {(N[i/j])}^+\neq 0$.
The required now follows.  The last part is straightforward.
Using the above proven facts,  we are now ready to show that \pe\  has a \ws\ in $F^m$. She can always
play a network $N$ with $\nodes(N)\subseteq m,$ such that
$N^+\neq 0$.\\
In the initial round, let \pa\ play $a\in \At\A$.
\pe\ plays a network $N$ with $N(0, \ldots, n-1)=a$. Then $N^+=a\neq 0$.
Recall that here \pa\ is offered only one (cylindrifier) move.
At a later stage, suppose \pa\ plays the cylindrifier move, which we denote by
$(N, \langle f_0, \ldots, f_{n-2}\rangle, k, b, l).$
He picks a previously played network $N$,  $f_i\in \nodes(N), \;l<n,  k\notin \{f_i: i<n-2\}$,
such that $b\leq {\sf c}_l N(f_0,\ldots,  f_{i-1}, x, f_{i+1}, \ldots, f_{n-2})$ and $N^+\neq 0$.
Let $\bar a=\langle f_0\ldots f_{i-1}, k, f_{i+1}, \ldots f_{n-2}\rangle.$
Then by  second part of  (3)  we have that ${\sf c}_lN^+\cdot {\sf s}_{\bar a}b\neq 0$
and so  by first part of (2), there is a network  $M$ such that
$M^+\cdot{\sf c}_{l}N^+\cdot {\sf s}_{\bar a}b\neq 0$.
Hence $M(f_0,\dots, f_{i-1}, k, f_{i-2}, \ldots$ $, f_{n-2})=b$,
$\nodes(M)=\nodes(N)\cup\set k$, and $M^+\neq 0$, so this property is maintained.

\end{proof}

\end{document}